\documentclass[1pt,reqno]{amsart}
\usepackage{mathrsfs}
\usepackage{amssymb}
\usepackage{amsmath}
\usepackage{amsfonts}
\usepackage{amscd}
\usepackage{graphicx}
\usepackage[colorlinks, linkcolor=blue, anchorcolor=blue, citecolor=blue]{hyperref}
\begin{document}
\setlength{\oddsidemargin}{0cm} \setlength{\evensidemargin}{0cm}
\baselineskip=20pt

\theoremstyle{plain} \makeatletter
\newtheorem{theorem}{Theorem}[section]
\newtheorem{proposition}[theorem]{Proposition}
\newtheorem{lemma}[theorem]{Lemma}
\newtheorem{coro}[theorem]{Corollary}

\theoremstyle{definition}
\newtheorem{defi}[theorem]{Definition}
\newtheorem{notation}[theorem]{Notation}
\newtheorem{exam}[theorem]{Example}
\newtheorem{prop}[theorem]{Proposition}
\newtheorem{conj}[theorem]{Conjecture}
\newtheorem{prob}[theorem]{Problem}
\newtheorem{remark}[theorem]{Remark}
\newtheorem{claim}{Claim}

\newcommand{\SO}{{\mathrm S}{\mathrm O}}
\newcommand{\SU}{{\mathrm S}{\mathrm U}}
\newcommand{\Sp}{{\mathrm S}{\mathrm p}}
\newcommand{\so}{{\mathfrak s}{\mathfrak o}}
\newcommand{\Ad}{{\mathrm A}{\mathrm d}}
\newcommand{\m}{{\mathfrak m}}
\newcommand{\g}{{\mathfrak g}}
\newcommand{\p}{{\mathfrak p}}
\newcommand{\fk}{{\mathfrak k}}
\newcommand{\E}{{\mathrm E}}
\newcommand{\F}{{\mathrm F}}
\newcommand{\T}{{\mathrm T}}
\newcommand{\ffb}{{\mathfrak b}}
\newcommand{\fl}{{\mathfrak l}}
\newcommand{\G}{{\mathrm G}}
\newcommand{\R}{{\mathrm Ric}}
\newcommand{\yd}{\approx}

\numberwithin{equation}{section}

\title{New Non-naturally reductive Einstein metrics on Exceptional simple Lie groups}
\author{Huibin Chen, Zhiqi Chen and ShaoQiang Deng}
\address{School of Mathematical Sciences and LPMC, Nankai University,
Tianjin 300071, P.R. China}\email{chenhuibin@mail.nankai.edu.cn, chenzhiqi@nankai.edu.cn,dengsq@nankai.edu.cn}
\date{}
\maketitle
\begin{abstract}
In this article, we achieved several non-naturally reductive Einstein metrics on exceptional simple Lie groups, which are formed by the decomposition arising from general Wallach spaces. By using the decomposition corresponding to the two involutive automorphisms, we calculated the non-zero coefficients in the expression for the components of Ricci tensor with respect to the given metrics. The Einstein metrics are obtained as solutions of systems polynomial equations, which we manipulate by symbolic computations using Gr\"{o}bner bases.
\end{abstract}

\section{Introduction}
A Riemannian manifold $(M,g)$ is called Einstein if there exists a constant $\lambda\in\mathbb{R}$ such that the Ricci tensor $r$ with respect to $g$ satisfies $r=\lambda g$. The readers can go to Besse's book \cite{Be} for more details and results in this field before 1986. General existence results are difficult to obtain, as a result, many mathematicians pay more attention on the special examples for Einstein manifolds. Among the first important attempts, the works of G. Jensen \cite{GJ1973} and M. Wang, W. Ziller \cite{WaZi} made much contributions to the progress of this field. When the problem is restricted to Lie groups, D'Atri and Ziller in \cite{JW} obtained a large number of naturally reductive left-invariant metrics when $G$ is simple. Also in this paper, they raised a problem: Whether there exists non-naturally reductive Einstein metrics on compact Lie group $G$? 

In 1994, K. Mori \cite{Mo} discovered the first left-invariant Einstein metrics on compact simple Lie groups $\SU(n)$ for $n\geq 6$, which are non-naturally reductive. In 2008, Arvanitoyeorgos, Mori and Sakane proved the existence of new non-naturally reductive Einstein metrics for $\SO(n)(n\geq11)$, $\Sp(n)(n\geq3)$, $\E_6$, $\E_7$ and $\E_8$, using fibrations of a compact simple Lie group over a K\"{a}hler $C$-space with two isotropy summands (see \cite{ArMoSa}). In 2014, by using the methods of representation theory, Chen and Liang \cite{ChLi} found a non-naturally reductive Einstein metric on the compact simple Lie group $\F_4$. More recently, Chrysikos and Sakane proved that there exists non-naturally reductive Einstein metric on exceptional Lie groups, espeacially for $\G_2$, they gave the first example of non-naturally reductive Einstein metric (see \cite{ChSa}).

In this paper we consider new non-naturally reductive Einstein metrics on compact exceptional Lie groups $G$ which can be seen as the principal bundle over generalized Wallach spaces $M=G/K$. In 2014, classification for generalized Wallach spaces arising from a compact simple Lie group has been obtained by Nikonorov \cite{Ni} and Chen, Kang and Liang \cite{ChKaLi}, in particular, Nikonorov investigated the semi-simple case and gave the classification in \cite{Ni}. 

As is known to all, the involutive automorphisms play an important role in the development of homogeneous geometry.  The Riemannian symmetric pairs were classified by Cartan \cite{Ca}, in Lie algebra level, which can be treated as the structure of a Lie algebra with an involutive automorphism satisfying some topologic properties. Later on, the more general semi-simple symmetric pairs were studied by Berger \cite{Ber}, whose classification can be obtained in the view of involutive automorphism. Recently, Huang and Yu \cite{HuYu} classified the Klein four subgroups $\Gamma$ of $\mathrm{Aut}(\frak{u}_0)$ for each compact Lie algebra $\frak{u}_0$ up to conjugation by calculating the symmetric subgroups $\mathrm{Aut}(\frak{u}_0)^\theta$ and their involution classes. 

According to the article \cite{ChKaLi}, each kind of generalized Wallach spaces arising from simple Lie groups is associated with two commutative involutive automorphisms of $\g$, the Lie algebra of $G$. With these two involutive automorphisms, we have two different corresponding decompositions of $\g$, which are in fact irreducible symmetric pairs. According to these two irreducible symmetric pairs, we can get some linear equations for the non-zero coefficients in the expression of components of Ricci tensor with respect to the given metric. With the help of computer, we get the Einstein metrics from the solutions of  systems polynomial equations. We mainly deal with two kinds of generalized Wallach spaces, one of which is without centers in $\fk$ and the other is with a $1$-dimensional center in $\fk$. Along with the results in \cite{ChLi}, we list all the number of non-naturally reductive left-invariant Einstein metrics on exceptional simple Lie groups $G$ arising from generalized Wallach spaces with no center in $\fk$. 
\begin{table}[!hbp]
\begin{tabular}{c|c|c|c|c}
\hline
$G$ & Types & $K$ & $p+q$ & $N_{non-nn}$ \\
\hline
\hline
$\F_4$ & $\F_4$-I & $\SO(8)$ & $1+3$ & 1\cite{ChLi}\\
            & $\F_4$-II & $\SU(2)\times\SU(2)\times\SO(5)$ & $3+3$ & 3 new\\
\hline
$\E_6$ & $\E_6$-III & $\SU(2)\times\mathrm{Sp}(3)$ & $2+3$ & 4 new\\
            & $\E_6$-II & $\mathrm{U}(1)\times\SU(2)\times\SU(2)\times\SU(4)$ & $0+3+3$ & 7 new\\
\hline
$\E_7$ & $\E_7$-I & $\SU(2)\times\SU(2)\times\SU(2)\times\SO(8)$ & $4+3$ & 7 new \\
            & $\E_7$-III & $\SO(8)$ & $1+3$ & 1 \\
            & $\E_7$-II & $\mathrm{U}(1)\times\SU(2)\times\SU(6) $ & $0+2+3$ & 6 new \\
\hline
$\E_8$ & $\E_8$-I & $\SU(2)\times\SU(2)\times\SO(12)$ & $3+3$ & 11 new\\
            & $\E_8$-II & $\Ad(\SO(8)\times\SO(8))$ & $2+3$ & 2 new\\
\hline
\end{tabular}
\caption{Number of non-naturally reductive left-invariant Einstein metrics on exceptional simple Lie group $G$ arising from generalized Wallach spaces}
\end{table} 
In this table, we still use the notations in \cite{ChKaLi} to represent the type of generalized Wallach space, $N_{non-nn}$ represents the number of non-naturally reductive Einstein metrics on $G$ and $p,q$ coincides with the indices in the decomposition $\g=\fk_1\oplus\cdots\oplus\fk_p\oplus\m_1\oplus\cdots\oplus\m_q$, in fact $q=3$ for all types.

We describe our results in the following theorem.
\begin{theorem}
\begin{enumerate}
\item The compact simple Lie group $\F_4$ admits at least 6 new non-naturally reductive and non-isometric left-invariant Einstein metrics. These metrics are $\Ad(\SU(2)\times\SU(2)\times\SO(5))$-invariant.
\item The compact simple Lie group $\E_6$ admits at least 11 new non-naturally reductive and non-isometric left-invariant Einstein metrics. Four of these metrics are $\Ad(\SU(2)\times\mathrm{Sp}(3))$-invariant and thie other 7 are $\Ad(\mathrm{U}(1)\times\SU(2)\times\SU(2)\times\SU(4))$-invariant.
\item The compact simple Lie group $\E_7$ admits at least 13 new non-naturally reductive and non-isometric left-invariant Einstein metrics. Seven of these metrics are $\Ad(\SU(2)\times\SU(2)\times\SU(2)\times\SO(8))$-invariant and the other 6 are $\Ad(\mathrm{U}(1)\times\SU(2)\times\SU(6))$-invariant.
\item The compact simple Lie group $\E_8$ admits at least 13 new non-naturally reductive and non-isometric left-invariant Einstein metrics. Two of these metrics are $\Ad(\SO(8)\times\SO(8))$-invariant and the other 11 are $\Ad(\SU(2)\times\SU(2)\times\SO(12))$-invariant.
\end{enumerate}
\end{theorem}

The paper is organized as follows: In section 2 we will recall a formula for the Ricci tensor of $G$ when we see $G$ as a homogeneous space. In section 3 we will introduce how we operate our methods to solve the non-zero coefficients in the expressions for Ricci tensor, where we will classify the exceptional Lie groups by the number of simple ideals of $\fk$. Then for each case in section 3, we will discuss the non-naturally reductive Einstein metrics via the solutions of systems polynomial equations, which will be described in section 4.

\section{The Ricci tensor for reductive homogeneous spaces}
In this section we will recall an expression for the Ricci tensor with respect to a class of given metrics on a compact semi-simple Lie group and figure out whether a metric on $G$ is naturally reductive. 

Let $G$ be a compact semi-simple Lie group with Lie algebra $\g$, $K$ a connected closed subgroup of $G$ with Lie algebra $\fk$. Through this paper, we denote by $B$ the negative of the Killing form of $\g$, which is positive definite because of the compactness of $G$, as a result, $B$ can be treated as an inner product on $\g$. Let $\g=\fk\oplus\m$ be the reductive decomposition with respect to $B$ such that $[\fk,\m]\subset\m$, where $\m$ is the tangent space of $G/K$. We assume that $\m$ can be decomposed into mutually non-equivalent irreducible $\Ad(K)$-modules as follows:
\begin{equation*}
\m=\m_1\oplus\cdots\oplus\m_q.
\end{equation*}
We will write $\fk=\fk_0\oplus\fk_1\oplus\cdots\oplus\fk_p$, where $\fk_0=Z(\fk)$ is the center of $\fk$ and $\fk_i$ is the simple ideal for $i=1,\cdots,p$. Let $G\times K$ act on $G$ by $(g_1,g_2)g=g_1gg_2^{-1}$, then $G\times K$ acts almost effectively on $G$ with isotropy group $\Delta(K)=\{(k,k)|k\in K\}$. As a result, $G$ can be treated as the coset space $(G\times K)/\Delta(K)$ and we have $\g\oplus\fk=\Delta(\fk)\oplus\Omega$, where $\Omega\cong T_0((G\times K)/\Delta(K))\cong\g$ via the linear map $(X,Y)\rightarrow(X-Y)\in\g$, $(X,Y)\in\Omega$.

As is known, there exists an 1-1 corresponding between all $G$-invariant metrics on the reductive homogeneous space $G/K$ and $\Ad_G(K)$-invariant inner products on $\m$. A Riemannian homogeneous space $(M=G/K,g)$ with reductive complement $\m$ of $\fk$ in $\g$ is called $naturally$ $reductive$ if 
\begin{equation*}
([X,Y]_\m,Z)+(Y,[X,Z]_\m)=0,
\end{equation*} 
where $X,Y,Z\in\m$, $(\ ,\ )$ is the corresponding inner product on $\g$.

In \cite{JW}, D'Atri and Ziller study the naturally reductive metrics among left invariant metrics on compact Lie groups and they obtained a complete classification of such metrics in the simple case. The following theorem will play an important role to decide whether a left-invariant metric on a Lie group is naturally reductive. 
\begin{theorem}\label{nr}
For any inner product $b$ on the center $\fk_0$ of $\fk$, the following left-invariant metrics on $G$ is naturally reductive with respect to the action $(g,k)y=gyk^{-1}$ of $G\times K$:
\begin{equation*}
\langle\ ,\ \rangle=u_0b|_{fk_0}+u_1B|_{\fk_1}+\cdots+u_pB|_{\fk_p}+xB|_{\m},\quad(u_0,u_1,\cdots,u_p,x\in\mathbb{R}^+)
\end{equation*}
Conversely, if a left-invariant metric $\langle\ ,\ \rangle$ on a compact simple Lie group $G$ is naturally reductive, then there exists a closed subgroup $K$ of $G$ such that $\langle\ ,\ \rangle$ can be written as above.
\end{theorem}

Now we have a orthogonal decomposition of $\g$ with respect to the Killing form of $\g$: $\g=\fk_0\oplus\fk_1\oplus\cdots\oplus\fk_p\oplus\m_1\oplus\cdots\oplus\m_q=(\fk_0\oplus\fk_1\oplus\cdots\oplus\fk_p)\oplus(\fk_{p+1}\oplus\cdots\oplus\fk_{p+q})$, with $\m_1\oplus\cdots\oplus\m_q=\fk_{p+1}\oplus\cdots\oplus\fk_{p+q}$, respectively. In addition, we assume that $\mathrm{dim}_{\mathbb{R}}\fk_0\leq1$ and the ideals $\fk_i$ are mutually non-isomorphic for $i=1,\cdots,p$. Then we consider the following left-invariant metric on $G$ which is in fact $\Ad(K)$-invariant:
\begin{equation}\label{metric1}
\langle\ ,\ \rangle=x_0\cdot B|_{\fk_0}+x_1\cdot B|_{\fk_1}+\cdots+x_{p+q}\cdot B|_{\fk_{p+q}},
\end{equation}
where $x_i\in\mathbb{R}^+$ for $i=1,\cdots,p+q$.

Set from now on $d_i=\mathrm{dim}_{\mathbb{R}}\fk_i$ and $\{e^{i}_{\alpha}\}_{\alpha=1}^{d_i}$ be a $B$-orthonormal basis adapted to the decomposition of $\g$ which means $e_{\alpha}^{i}\in\fk_i$ and $\alpha$ is the number of basis in $\fk_i$. Then we consider the numbers $A_{\alpha,\beta}^{\gamma}=B([e_{\alpha}^{i},e_{\beta}^{j}],e_{\gamma}^{k})$ such that $[e_{\alpha}^{i},e_{\beta}^{j}]=\sum_{\gamma}A_{\alpha,\beta}^{\gamma}e_{\gamma}^{k}$, and set
\begin{equation*}
(ijk):=\Big[\begin{array}{cc}
i \\
j \ k
\end{array}\Big]=\sum(A_{\alpha,\beta}^{\gamma})^2,
\end{equation*}
where the sum is taken over all indices $\alpha,\beta,\gamma$ with $e_{\alpha}^{i}\in\fk_i, e_{\beta}^{j}\in\fk_j, e_{\gamma}^{k}\in\fk_k$. Then $(ijk)$ is independent of the choice for the $B$-orthonormal basis of $\fk_i,\fk_j,\fk_k$, and symmetric for all three indices which means $(ijk)=(jik)=(jki)$. 

In \cite{ArMoSa} and \cite{PaSa}, the authors obtained the formulas for the components of Ricci tensor with respect to the left-invariant metric given by (\ref{metric1}), which can be described by the following lemma:
\begin{lemma}\label{formula1}
Let $G$ be a compact connected semi-simple Lie group endowed with the left-invariant metric $\langle\ ,\ \rangle$ given by (\ref{metric1}). Then the components $r_0,r_1,\cdots,r_{p+q}$ of the Ricci tensor associated to $\left<\ ,\ \right>$ are expressed as follows:
\begin{equation*}
r_k=\frac{1}{2x_k}+\frac{1}{4d_k}\sum_{j,i}\frac{x_k}{x_jx_i}\Big[\begin{array}{cc}
k \\
j \ i
\end{array}\Big]-\frac{1}{2d_k}\sum_{j,i}\frac{x_j}{x_kx_i}\Big[\begin{array}{cc}
j \\
k \ i
\end{array}\Big],\quad(k=0,1,\cdots,p+q).
\end{equation*}
Here, the sums are taken over all $i=0,1,\cdots,p+q$. In particular, for each $k$ it holds that
\begin{equation*}
\sum_{i,j}^{p+q}\Big[\begin{array}{cc}
j \\
k \ i
\end{array}\Big]=\sum_{ij}^{p+q}(kij)=d_k.
\end{equation*}
\end{lemma}

\section{Calculations for non-zero coefficients in the expressions of Ricci tensor}
In this section, we will calculate the non-zero coefficients in the expressions for the components of the Ricci tensor with respect to the given metric (\ref{metric1}). First of all, we classify the exceptional Lie groups without center in $K$ into the following three types, namely $p=2$, $p=3$ and $p=4$, where $p$ represents the number of the simple ideals of $\fk$. For the case of $p=1$, according to the classification \cite{ChKaLi}, there are only $\F_4$-I and $\E_7$-III, the non-naturally reductive Einstein metrics on which were studied in \cite{ChLi} and [Lei], respectively. 

We recall the definition of generalized Wallach spaces. Let $G/K$ be a reductive homogeneous space, where $G$ is a semi-simple compact connected Lie group, $K$ is a connected closed subgroup of $G$, $\g$ and $\fk$ are the corresponding Lie algebras, respectively. If $\m$, the tangent space of $G/K$ at $o=\pi(e)$, can be decomposed into three $\mathrm{ad}(\fk)$-invariant irreducible summands pairwise orthogonal with respect to $B$ as:
\begin{equation*}
\m=\m_1\oplus\m_2\oplus\m_3,
\end{equation*}
satisfying $[\m_i,\m_i]\in\fk$ for $i\in\{1,2,3\}$ and $[\m_i,\m_j]\in\m_k$ for $\{i,j,k\}=\{1,2,3\}$, then we call $G/K$ a generalized Wallach space.

In \cite{ChKaLi} and \cite{Ni}, the authors gave the complete classification of generalized Wallach spaces in the simple case. Here we use the notations in \cite{ChKaLi}, then for $p=2$, there are $\E_6$-III and $\E_8$-II, for $p=3$, there are $\F_4$-II and $\E_8$-I and for $p=4$ there is only $\E_7$-I. For later use, we introduce the following lemma given in \cite{ArDzNi}.
\begin{lemma}\label{iii}
Let $\mathfrak{q}\subset\mathfrak{r}$ be arbitrary subalgebra in $\g$ with $\mathfrak{q}$ simple. Consider in $\mathfrak{q}$ an orthonormal (with respect to $B_\mathfrak{r}$) basis $\{f_j\}(1\leq j\leq \mathrm{dim}(\mathfrak{q}))$. Then 
\begin{equation*}
\sum_{i,j,k=1}^{\mathrm{dim}(\mathfrak{q})}(B_{\mathfrak{r}}([f_i,f_j],f_k))^2 =\alpha_{\mathfrak{q}}^{\mathfrak{r}}\cdot\mathrm{dim}(\mathfrak{q}),
\end{equation*}
where $\alpha_{\mathfrak{q}}^\mathfrak{r}$ is determined by the equation $B_{\mathfrak{q}}=\alpha_{\mathfrak{q}}^\mathfrak{r}\cdot B_{\mathfrak{r}}|_{\mathfrak{q}}$.
\end{lemma}

For $p=2$, we consider the following metrics on $\g$:
\begin{equation}\label{metric2}
<\ ,\ >=x_1B|_{\fk_1}+x_2B|_{\fk_2}+x_3B|_{\m_1}+x_4B|_{\m_2}+x_5B|_{\m_3}.
\end{equation}

Since $\fk_1,\fk_2$ are the simple ideals of $\fk$ and by the definition of generalized Wallach space, we can easily obtain the possible non-zero coefficients in the expressions for Ricci tensor as follows:
$$(111),(222),(133),(144),(155),(233),(244),(255),(345)$$

According to Lemma \ref{formula1}, we have
\begin{equation*}
\begin{split}
r_1&=\frac{1}{4d_1}\left(\frac{1}{x_1}(111)+\frac{x_1}{x_3^2}(133)+\frac{x_1}{x_4^2}(144)+\frac{x_1}{x_5^2}(155)\right),\\
r_2&=\frac{1}{4d_2}\left(\frac{1}{x_2}(222)+\frac{x_2}{x_3^2}(233)+\frac{x_2}{x_4^2}(244)+\frac{x_2}{x_5^2}(255)\right),\\
r_3&=\frac{1}{2x_3}+\frac{1}{2d_3}(345)\left(\frac{x_3}{x_4x_5}-\frac{x_4}{x_3x_5}-\frac{x_5}{x_3x_4}\right)-\frac{1}{2d_3}\left(\frac{x_1}{x_3^2}(133)+\frac{x_2}{x_3^2}(233)\right),\\
r_4&=\frac{1}{2x_4}+\frac{1}{2d_4}(345)\left(\frac{x_4}{x_3x_5}-\frac{x_3}{x_4x_5}-\frac{x_5}{x_3x_4}\right)-\frac{1}{2d_4}\left(\frac{x_1}{x_4^2}(144)+\frac{x_2}{x_4^2}(244)\right),\\
r_5&=\frac{1}{2x_5}+\frac{1}{2d_5}(345)\left(\frac{x_5}{x_4x_3}-\frac{x_4}{x_3x_5}-\frac{x_3}{x_5x_4}\right)-\frac{1}{2d_5}\left(\frac{x_1}{x_5^2}(155)+\frac{x_2}{x_5^2}(255)\right),
\end{split}
\end{equation*} 
and 
\begin{equation}\label{eqn2}
\begin{split}
&(111)+(133)+(144)+(155)=d_1,\\
&(222)+(233)+(244)+(255)=d_2,\\
&2(133)+2(233)+2(345)=d_3,\\
&2(144)+2(244)+2(345)=d_4,\\
&2(155)+2(255)+2(345)=d_5.
\end{split}
\end{equation}
We used the symmetric property of three indices in $(ijk)$ in above equations.

In order to calculate the non-zero coefficients $(111)$, $(222)$, $(133)$, $(144)$, $(155)$, $(233)$, $(244)$, $(255)$, $(345)$, we should know more about the structure of the corresponding Lie algebras. Since the structure of generalized Wallach space arising from a simple group can be decided by two commutative involutive automorphisms on $\g$, we can learn more information from these two automorphisms. 
\begin{lemma}\label{p=2}
In the case of $p=2$, the non-zero coefficients in the components of Ricci tensor with respect to metric (\ref{metric2}) are as follows:\\
for $\E_6$-III, the non-zero coefficients are
\begin{equation*}
\begin{split}
&(111)=\frac{1}{2}, (133)=0, (144)=\frac{7}{4}, (155)=\frac{3}{4},\\
&(222)=7, (233)=\frac{7}{2}, (244)=\frac{35}{4}, (255)=\frac{7}{4}, (345)=\frac{7}{2},
\end{split}
\end{equation*}
for $\E_8$-II, the non-zero coefficients are 
\begin{equation*}
\begin{split}
&(111)=(222)=\frac{28}{5}, (345)=\frac{256}{15},\\
&(133)=(144)=(233)=(244)=(155)=(255)=\frac{112}{15}.
\end{split}
\end{equation*}
\end{lemma}
\begin{proof}
Case of $\E_6$-III. For this case, we denote the two commutative involutive automorphisms by $\theta$ and $\tau$. By the conclusions in \cite{ChKaLi}, we know that each of the automorphisms corresponds to an irreducible symmetric pair. For the structure of $\E_6$-III, we have the following decomposition:
\begin{equation*}
\g=A_1\oplus C_3\oplus\p_1\oplus\p_2\oplus\p_3.
\end{equation*}
Let
\begin{equation}\label{E6III-1}
\g=\mathfrak{b}+\p, \quad \mathfrak{b}=\mathfrak{b}_1+\mathfrak{b}_2, \quad \ffb_1=A_1, \quad \mathfrak{b}_2=C_3+\m_1\cong A_5, \quad \p=\m_2+\m_3,
\end{equation}
then $(\g,\mathfrak{b})$ is in fact the irreducible symmetric pair corresponding to $\theta$.\\
Let
\begin{equation}\label{E6III-2}
\g=\mathfrak{b'}+\p', \quad \mathfrak{b}'=A_1+C_3+\m_2\cong\F_4, \quad \p'=\m_1+\m_3,
\end{equation}
then $(\g,\mathfrak{b}')$ is in fact the irreducible symmetric pair corresponding to $\tau$.
We consider the following metric on $\g$ with respect to the decomposition (\ref{E6III-1})
\begin{equation}\label{metric21}
<<\ ,\ >>_1=u_1B|_{A_1}+u_2B|_{A_5}+u_3|_\p,
\end{equation}
and denote the components of the Ricci tensor with respect to the metric $<<,>>_1$ by $\tilde{r}_1$, $\tilde{r}_2$ and $\tilde{r}_3$. If we let $x_1=u_1, x_2=x_3=u_2, x_4=x_5=u_3$, then the metric given by (\ref{metric2}) and (\ref{metric21}) are the same, thus their corresponding Ricci tensor are also the same, which means $r_1=\tilde{r}_1$, $r_2=r_3=\tilde{r}_2$, $r_4=r_5=\tilde{r}_3$. As a result, from the expressions of each Ricci components, we have
\begin{equation}\label{eqn21}
\begin{split}
&\frac{1}{4d_2}\left((222)+(233)\right)=\frac{1}{4}-\frac{1}{2d_3}(345),\\
&\frac{1}{4d_2}\left((244)+(255)\right)=\frac{1}{2d_3}(345),\\
&\frac{1}{2d_4}(144)=\frac{1}{2d_5}(155),\\
&(133)=0.
\end{split}
\end{equation}
With the same method, we consider the irreducible symmetric pair corresponding to the involutive automorphism $\tau$. The metric taken into consideration is as follows:
\begin{equation}\label{metric22}
<<\ ,\ >>_2=w_1B|_{\F_4}+w_2B|_{\p'},
\end{equation}
If we let $x_1=x_2=x_4=w_1$ and $x_3=x_5=w_2$, the Ricci tensors with respect to (\ref{metric2}) and (\ref{metric22}) are the same, by comparing the expression for both components, we have
\begin{equation}\label{eqn22}
\begin{split}
&\frac{1}{4d_1}\left((111)+(144)\right)=\frac{1}{4d_2}\left((222)+(244)\right)=\frac{1}{4}-\frac{1}{2d_4}(345),\\
&\frac{1}{4d_1}\left((133)+(155)\right)=\frac{1}{4d_2}\left((233)+(255)\right)=\frac{1}{2d_4}(345).
\end{split}
\end{equation}
We can easily calculate $\alpha_{A_1}^{\E_6}=\frac{8}{48}$, and $\alpha_{C_3}^{\E_6}=\frac{16}{48}$, since their corresponding roots have the same length, we obtain $(111)=\frac{1}{2}$ and $(222)=7$ according to Lemma \ref{iii}. From Table 1 in \cite{Ni}, we can find out $(345)=\frac{7}{2}$, along with the equations (\ref{eqn2}), (\ref{eqn21}) and (\ref{eqn22}) and $d_1=3$, $d_2=21$, $d_3=14$, $d_4=28$ and $d_5=12$, one can easily get the solutions given in the lemma.

Case of $\E_8$-II. For this case, we have the following decomposition:
\begin{equation*}
\g=D_4\oplus D_4\oplus\m_1\oplus\m_2\oplus\m_3.
\end{equation*}
Then we study the two commutative involutive automorphisms denoted by $\theta$ and $\tau$. In fact, the corresponding irreducible symmetric pair $(\g,\mathfrak{b})$ and $(\g,\mathfrak{b}')$ of $\theta$ and $\tau$ are respectively as follows:
\begin{equation*}\label{E8II-1}
\g=\mathfrak{b}+\p, \quad \mathfrak{b}=D_4+D_4+\m_1\cong D_8, \quad \p=\m_2+\m_3.
\end{equation*}
\begin{equation*}\label{E8II-2}
\g=\mathfrak{b}'+\p', \quad \mathfrak{b}'=D_4+D_4+\m_2\cong D_8, \quad \p=\m_1+\m_3.
\end{equation*}
By the same methods operated in the case of $\E_6$-III, we have

\begin{equation*}
\begin{split}
&(111)=(222)=\frac{28}{5}, (345)=\frac{256}{15},\\
&(133)=(144)=(155)=(233)=(244)=(255)\frac{112}{15}.
\end{split}
\end{equation*}
\end{proof}

For $p=3$, we have the following decomposition:
\begin{equation*}
\g=\fk_1+\fk_2+\fk_3+\m_1+\m_2+\m_3,
\end{equation*}
and we consider the metric as follows: 
\begin{equation}\label{metric3}
<\ ,\ >=x_1B|_{\fk_1}+x_2B|_{\fk_2}+x_3B|_{\fk_3}+x_4B|_{\p_1}+x_5B|_{\p_2}+x_6B|_{\p_3}.
\end{equation}
Since $\fk_i(i=1,2,3)$ is simple ideal of $\fk$ and according to the structure of generalized Wallach space, it is easy to know that the possible non-zero coefficients in the expression for the components of Ricci tensor with respect to the metric (\ref{metric3}) are as follows:
\begin{equation*}
(111), (222), (333), (144), (155), (166), (244), (255), (266), (344), (355), (366), (456).
\end{equation*}
Then the components of the Ricci tensor with respect to this metric are as follows according to the Lemma \ref{formula1}: 
\begin{equation*}
\begin{split}
r_1&=\frac{1}{4d_1}\left(\frac{1}{x_1}(111)+\frac{x_1}{x_4^2}(144)+\frac{x_1}{x_5^2}(155)+\frac{x_1}{x_6^2}(166)\right),\\
r_2&=\frac{1}{4d_2}\left(\frac{1}{x_2}(222)+\frac{x_2}{x_4^2}(244)+\frac{x_2}{x_5^2}(255)+\frac{x_2}{x_6^2}(266)\right),\\
r_3&=\frac{1}{4d_3}\left(\frac{1}{x_3}(333)+\frac{x_3}{x_4^2}(344)+\frac{x_3}{x_5^2}(355)+\frac{x_3}{x_6^2}(366)\right),\\
r_4&=\frac{1}{2x_4}+\frac{1}{2d_4}(456)\left(\frac{x_4}{x_5x_6}-\frac{x_6}{x_4x_5}-\frac{x_5}{x_4x_6}\right)-\frac{1}{2d_4}\left(\frac{x_1}{x_4^2}(144)+\frac{x_2}{x_4^2}(244)+\frac{x_3}{x_4^2}(344)\right),\\
r_5&=\frac{1}{2x_5}+\frac{1}{2d_5}(456)\left(\frac{x_5}{x_4x_6}-\frac{x_6}{x_4x_5}-\frac{x_4}{x_5x_6}\right)-\frac{1}{2d_5}\left(\frac{x_1}{x_5^2}(155)+\frac{x_2}{x_5^2}(255)+\frac{x_3}{x_5^2}(355)\right),\\
r_6&=\frac{1}{2x_6}+\frac{1}{2d_6}(456)\left(\frac{x_6}{x_5x_4}-\frac{x_4}{x_6x_5}-\frac{x_5}{x_4x_6}\right)-\frac{1}{2d_6}\left(\frac{x_1}{x_6^2}(166)+\frac{x_2}{x_4^2}(266)+\frac{x_3}{x_4^2}(366)\right).
\end{split}
\end{equation*}
and the following equations:
\begin{equation}\label{eqn3}
\begin{split}
&(111)+(144)+(155)+(166)=d_1,\\
&(222)+(244)+(255)+(266)=d_2,\\
&(333)+(344)+(355)+(366)=d_3,\\
&2(144)+2(244)+2(344)+2(456)=d_4,\\
&2(155)+2(255)+2(355)+2(456)=d_5,\\
&2(166)+2(266)+2(366)+2(456)=d_6.
\end{split}
\end{equation}

\begin{lemma}\label{p=3}
The possible non-zero coefficients in the expression for the components of Ricci tensor with respect to the metric (\ref{metric3}) are as follows:

for case of $\F_4$-II, we have 
\begin{equation*}
\begin{split}
&(111)=(222)=\frac{2}{3}, (333)=\frac{10}{3}, (144)=(244)=\frac{5}{3}, (155)=(266)=0, \\
&(166)=(255)=\frac{2}{3}, (355)=(366)=\frac{10}{9}, (344)=\frac{40}{9},(456)=\frac{20}{9}.
\end{split}
\end{equation*}

for case of $\E_8$-I, we have 
\begin{equation*}
\begin{split}
&(111)=(222)=\frac{1}{5},(333)=22,(144)=(244)=\frac{6}{5},(155)=(266)=0,\\
&(166)=(255)=\frac{8}{5},(344)=\frac{44}{5},(355)=(366)=\frac{88}{5},(456)=\frac{64}{5}.
\end{split}
\end{equation*}
\end{lemma}
\begin{proof}
Case of $\F_4$-II. We have the following decomposition according to the structure of $\F_4$-II:
\begin{equation*}
\g=A_1^1+A_1^2+C_2+\m_1+\m_2+\m_3.
\end{equation*}
If we let
\begin{equation*}
\frak{b}=A_1^1+A_1^2+C_2+\m_1,\quad \p=\m_2+\m_3,
\end{equation*}
then $(\g,\frak{b})$ is an irreducible symmetric pair. In fact, this decomposition can be achieved by the first involutive automorphism $\theta$, which means $\frak{b}\cong B_4$, therefore simple. Now, we consider the following metric on $\g$: 
\begin{equation}\label{metric31}
<<\ ,\ >>_1=u_1B|_{\frak{b}}+u_2B|_{\p},
\end{equation}
this metric is the same as the one in (\ref{metric3}) if we let $x_1=x_2=x_3=x_4=u_1$ and $x_5=x_6=u_2$. As a result, if we denote the components of the Ricci tensor with respect to (\ref{metric3}) by $\tilde{r}_1$ and $\tilde{r}_2$, then it holds $r_1=r_2=r_3=r_4=\tilde{r}_1$ and $r_5=r_6=\tilde{r}_2$. With comparing these equations, we have 
\begin{equation}\label{eqn31}
\begin{split}
&\frac{1}{4d_1}((111)+(144))=\frac{1}{4d_2}((222)+(244))=\frac{1}{4d_3}((333)+(344))=\frac{1}{4}-\frac{1}{2d_4}(456),\\
&\frac{1}{4d_1}((155)+(166))=\frac{1}{4d_2}((255)+(266))=\frac{1}{4d_3}((355)+(366))=\frac{1}{2d_4}(456).
\end{split}
\end{equation}
We consider the following decomposition of $\g$,
\begin{equation*}
\g=\frak{b}'+\p'=\frak{b}'_1+\frak{b}'_2+\p',
\end{equation*}
where $\frak{b}'=A_1^1+A_1^1+C_2+\m_2,\frak{b}'_1=A_1^1,\frak{b}'_2=A_1^2+C_2+\m_2,\p=\m_1+\m_3$. In fact, $\frak{b}'_2\cong C_3$, and this decomposition is corresponding to the second involutive automorphism $\tau$ on $\g$. Hence, $(\g,\frak{b}')$ is an irreducible symmetric pair. Now, we consider the metric on $\g$ as follows:
\begin{equation}\label{metric32}
<<\ ,\ >>_2=w_1B|_{A_1^1}+w_2B|_{C_3}+w_3B|_{\p^{'}}.
\end{equation}
If we set $x_1=w_1, x_2=x_3=x_5=w_2, x_4=x_6=w_3$ in (\ref{metric3}), then the two metrics are the same, as a result, we can get the following equations:
\begin{equation}\label{eqn32}
\begin{split}
&\frac{1}{4d_2}((222)+(255))=\frac{1}{4d_3}((333)+(355))=\frac{1}{4}-\frac{1}{2d_5}((456)-(155)),\\
&\frac{1}{4d_2}((244)+(266))=\frac{1}{4d_3}((344)+(366))=\frac{1}{2d_5}(456),\\
&\frac{1}{2d_5}(155)=0, \frac{1}{2d_4}(144)=\frac{1}{2d_6}(166).
\end{split}
\end{equation}
We can easily calculate $\alpha_{A_1}^{\F_4}=\frac{8}{36}$ and $\alpha_{C_2}^{\F_4}=\frac{12}{36}$ both with the long roots of $\F_4$, therefore, $(111)=(222)=\frac{2}{3}$ and $(333)=\frac{10}{3}$ according to Lemma \ref{iii}, besides we get $(456)=\frac{20}{9}$ from \cite{Ni}. Thus, along with the above equations (\ref{eqn3}), (\ref{eqn31}) and (\ref{eqn32}), one can easily get the solutions as the lemma, here $d_1=3, d_2=3, d_3=10, d_4=20, d_5=8,d_6=8$.

Case of $\E_8$-I. According to \cite{ChKaLi}, $\E_8$-I has the following decomposition:
\begin{equation*}
\g=A_{1}^{1}+A_{1}^{2}+D_6+\m_1+\m_2+\m_3.
\end{equation*}
According to the two commutative involutive automorphisms on $\E_8$-I in \cite{ChKaLi}, we have the following two decompositions which make $(\g,\ffb)$ and $(\g,\ffb')$ be two different irreducible symmetric pairs.
\begin{equation*}
\begin{split}
&\g=\ffb+\p, \ffb=A_{1}^{1}+A_{1}^{2}+D_6+\m_1\cong D_8, \p=\m_2+\m_3,\\
\g=\ffb'+\p', \ffb'=\ffb'_1+&\ffb'_2=A_1+A_1+D_6+\m_2, \ffb'_1=A_{1}^{1}, \ffb'_2=A_{1}^{2}+D_6+\m_2\cong \E_7, \p'=\m_1+\m_3.
\end{split}
\end{equation*}
By using the same methods above, we can calculate $(111)=(222)=\frac{1}{5}$ and $(333)=22$ by Lemma \ref{iii} and get $(456)=\frac{64}{5}$, as a result, we have
\begin{equation*}
\begin{split}
&(111)=(222)=\frac{1}{5},(333)=22,(144)=(244)=\frac{6}{5},(155)=(266)=0,\\
&(166)=(255)=\frac{8}{5},(344)=\frac{44}{5},(355)=(366)=\frac{88}{5},(456)=\frac{64}{5}.
\end{split}
\end{equation*}
\end{proof}

For $p=4$, there is only $\E_7$-I in this type, and the decomposition according to \cite{ChKaLi} is 
\begin{equation*}
\g=A_1^1+A_1^2+A_1^3+D_4+\p_1+\p_2+\p_3.
\end{equation*}
The considered metric is of the following form:
\begin{equation}\label{metric4}
 <\ ,\ >=x_1B|_{A_1^1}+x_2B|_{A_1^2}+x_3B|_{A_1^3}+x_4B|_{D_4}+x_5B|_{\m_1}+x_6B|_{\m_2}+x_7B|_{\m_3},
\end{equation}
It is easy to know that the possible non-zero coefficients in the expression for components of Ricci tensor with respect to the metric given by (\ref{metric4}) are
\begin{equation*}
(111), (222), (333), (444), (155), (166), (177), (255),(266), (277), (355), (366), (377), (455), (466), (477).
\end{equation*}
As a result, the components of Ricci tensor are:
\begin{equation*}
\begin{split}
r_1&=\frac{1}{4d_1}\left(\frac{1}{x_1}(111)+\frac{x_1}{x_5^2}(155)+\frac{x_1}{x_6^2}(166)+\frac{x_1}{x_7^2}(177)\right),\\
r_2&=\frac{1}{4d_2}\left(\frac{1}{x_2}(222)+\frac{x_2}{x_5^2}(255)+\frac{x_2}{x_6^2}(266)+\frac{x_2}{x_7^2}(277)\right),\\
r_3&=\frac{1}{4d_3}\left(\frac{1}{x_3}(333)+\frac{x_3}{x_5^2}(355)+\frac{x_3}{x_6^2}(366)+\frac{x_3}{x_7^2}(377)\right),\\
r_4&=\frac{1}{4d_4}\left(\frac{1}{x_4}(444)+\frac{x_4}{x_5^2}(455)+\frac{x_4}{x_6^2}(466)+\frac{x_4}{x_7^2}(477)\right),\\
r_5&=\frac{1}{2x_5}+\frac{1}{2d_5}(567)\left(\frac{x_5}{x_6x_7}-\frac{x_6}{x_7x_5}-\frac{x_7}{x_5x_6}\right)-\frac{1}{2d_5}\left(\frac{x_1}{x_5^2}(155)+\frac{x_2}{x_5^2}(255)+\frac{x_3}{x_5^2}(355)+\frac{x_4}{x_5^2}(455)\right),\\
r_6&=\frac{1}{2x_6}+\frac{1}{2d_6}(567)\left(\frac{x_6}{x_5x_7}-\frac{x_7}{x_6x_5}-\frac{x_5}{x_7x_6}\right)-\frac{1}{2d_6}\left(\frac{x_1}{x_6^2}(166)+\frac{x_2}{x_6^2}(266)+\frac{x_3}{x_6^2}(366)+\frac{x_4}{x_6^2}(466)\right),\\
r_7&=\frac{1}{2x_7}+\frac{1}{2d_7}(567)\left(\frac{x_7}{x_5x_6}-\frac{x_5}{x_6x_7}-\frac{x_6}{x_7x_6}\right)-\frac{1}{2d_7}\left(\frac{x_1}{x_7^2}(177)+\frac{x_2}{x_7^2}(277)+\frac{x_3}{x_7^2}(377)+\frac{x_4}{x_7^2}(477)\right).
\end{split}
\end{equation*}
\begin{lemma}\label{p=4}
The possible non-zero coefficients in the expression for components of Ricci tensor with respect to the metric given by (\ref{metric4}) are as follows:
\begin{equation*}
\begin{split}
&(111)=(222)=(333)=\frac{1}{3},(444)=\frac{28}{3},(567)=\frac{64}{9},\\
&(166)=(177)=(255)=(277)=(355)=(366)=\frac{4}{3},\\
&(155)=(266)=(377)=0,(455)=(466)=(477)=\frac{56}{9}.
\end{split}
\end{equation*}
\end{lemma}
\begin{proof}
According to the two commutative involutive automorphisms on $\E_7$-I in \cite{ChKaLi}, we have the following two decompositions which make $(\g,\ffb)$ and $(\g,\ffb')$ be two different irreducible symmetric pairs.
\begin{equation*}
\begin{split}
&\g=\ffb+\p, \ffb=A_1^1+A_1^2+A_1^3+D_4+\m_1=\ffb_1+\ffb_2,\\
& \ffb_1=A_1^1,\ffb_2=A_1^2+A_1^3+D_4+\m_1\cong D_6, \p=\m_2+\m_3,\\
&\g=\ffb'+\p', \ffb=A_1^1+A_1^2+A_1^3+D_4+\m_2=\ffb'_1+\ffb'_2, \\
&\ffb'_1=A_1^2,\ffb'_2=A_1^1+A_1^3+D_4+\m_2\cong D_6, \p=\m_1+\m_3.
\end{split}
\end{equation*}
Then we can use the similar methods to calculate out the coefficients given in the lemma.
\end{proof}
\textbf{Case of $\E_7$-II.} In \cite{ChKaLi}, $\E_7$-II has the following decomposition:
\begin{equation}
\E_7=\T\oplus A_1\oplus A_5\oplus\p_1\oplus\p_2\oplus\p_3,
\end{equation}
where $\T$ is the 1-dimensional center of $\fk$.\\
We consider left-invariant metrics on $\E_7$-II as follows:
\begin{equation}\label{metric7}
<\ ,\ >=u_0\cdot B|_{\T}+x_1\cdot B|_{A_1}+x_2\cdot B|_{A_5}+x_3\cdot B|_{\p_1}+x_4\cdot B|_{\p_2}+x_5\cdot B|_{\p_3},
\end{equation}
where $u_0, x_1, x_2, x_3, x_4, x_5\in\mathbb{R}^+$. According to the structure of generalized Wallach space, the possible non-zero coefficients in the expression for the components of Ricci tensor with respect to the metric (\ref{metric7}) are
$$(033), (044), (055), (111), (133), (144), (155), (222), (233), (244), (255), (345)$$
Therefore, by Lemma \ref{formula1}, the components of Ricci tensor with respect to the metric (\ref{metric7}) can be expressed as follows:
\[\left\{\begin{aligned}
r_0&=\frac{1}{4d_0}\left(\frac{u_0}{x_3^2}(033)+\frac{u_0}{x_4^2}(044)+\frac{u_0}{x_5^2}(055)\right),\\
r_1&=\frac{1}{4d_1}\left(\frac{1}{x_1}(111)+\frac{x_1}{x_3^2}(133)+\frac{x_1}{x_4^2}(144)+\frac{x_1}{x_5^2}(155)\right),\\
r_2&=\frac{1}{4d_2}\left(\frac{1}{x_2}(222)+\frac{x_2}{x_3^2}(233)+\frac{x_2}{x_4^2}(244)+\frac{x_2}{x_5^2}(255)\right),\\
r_3&=\frac{1}{2x_3}+\frac{1}{2d_3}(345)\left(\frac{x_3}{x_4x_5}-\frac{x_4}{x_5x_3}-\frac{x_5}{x_3x_4}\right)-\frac{1}{2d_3}\left(\frac{u_0}{x_3^2}(033)+\frac{x_1}{x_3^2}(133)+\frac{x_2}{x_3^2}(233)\right),\\
r_4&=\frac{1}{2x_4}+\frac{1}{2d_4}(345)\left(\frac{x_4}{x_5x_3}-\frac{x_5}{x_3x_4}-\frac{x_3}{x_4x_5}\right)-\frac{1}{2d_4}\left(\frac{u_0}{x_4^2}(044)+\frac{x_1}{x_4^2}(144)+\frac{x_2}{x_4^2}(244)\right),\\
r_5&=\frac{1}{2x_5}+\frac{1}{2d_5}(345)\left(\frac{x_5}{x_3x_4}-\frac{x_3}{x_4x_5}-\frac{x_4}{x_5x_3}\right)-\frac{1}{2d_5}\left(\frac{u_0}{x_5^2}(055)+\frac{x_1}{x_5^2}(155)+\frac{x_2}{x_5^2}(255)\right).
\end{aligned}\right.\]
and 
\begin{equation}\label{eqn777}
\begin{split}
&(033)+(044)+(055)=d_0,\\
&(111)+(133)+(144)+(155)=d_1,\\
&(222)+(233)+(244)+(255)=d_2,\\
&2(033)+2(133)+2(233)+2(345)=d_3,\\
&2(044)+2(144)+2(244)+2(345)=d_4,\\
&2(055)+2(155)+2(255)+2(345)=d_5,
\end{split}
\end{equation}
where we used the symmetric properties of the indices in $(ijk)$.
\begin{lemma}\label{coeff7}
In the case of $\E_7$-II, the possible non-zero coefficients in the expression for the components of Ricci tensor with respect to metric (\ref{metric7}) are as follows:
\begin{equation*}
\begin{split}
&(033)=\frac{4}{9}, (044)=\frac{5}{9}, (055)=0, (345)=\frac{20}{3},\\
&(111)=\frac{1}{3}, (133)=1, (144)=0, (155)=\frac{5}{3},\\
&(222)=\frac{35}{3}, (233)=\frac{35}{9}, (244)=\frac{70}{9}, (255)=\frac{35}{3}.
\end{split}
\end{equation*}
\end{lemma}
\begin{proof}
In fact, there are two involutive automorphisms on $\E_7$-II \cite{ChKaLi}, denoted by $\sigma$ and $\tau$, where $\sigma$ corresponds to the irreducible symmetric pair $(\g,\ffb)$ having the following decomposition:
\begin{equation}\label{decom71}
\g=\ffb\oplus\p,\quad\ffb=\T\oplus A_1\oplus A_5\oplus\p_1\cong A_7,\quad\p=\p_2\oplus\p_3,
\end{equation}
and $\tau$ corresponds to the irreducible symmetric pair $(\g,\ffb')$ having the following decomposition:
\begin{equation}\label{decom72}
\g=\ffb'\oplus\p',\quad\ffb'=\ffb_1'\oplus\ffb_2',\quad\ffb_1'=A_1,\quad\ffb_2'=\T\oplus A_5\oplus\p_2\cong D_6,\quad\p'=\p_1\oplus\p_3.
\end{equation}
We consider the following left-invariant metric on $\E_7$ with respect to the decomposition (\ref{decom71})
\begin{equation}\label{metric71}
(\ ,\ )_1=w_1\cdot B|_{\ffb}+w_2\cdot B|_{\p},
\end{equation}
if we let $u_0=x_1=x_2=x_3=w_1$ and $x_4=x_5=w_2$ in (\ref{metric7}), then these two metrics are the same, as a result, if we denote the components of Ricci tensor with respect to the metric (\ref{metric71}) by $\tilde{r}_1$ and $\tilde{r}_2$, then we have $r_0=r_1=r_2=r_3=\tilde{r}_1$ and $r_4=r_5=\tilde{r}_2$. With an easy calculation we get the following equations of the possible non-zero coefficients:
\begin{equation}\label{eqn71}
\begin{split}
&\frac{1}{4d_0}(033)=\frac{1}{d_1}\left((111)+(133)\right)=\frac{1}{d_2}\left((222)+(233)\right)=\frac{1}{4}-\frac{1}{2d_3}(345),\\
&\frac{1}{4d_0}\left((044)+(055)\right)=\frac{1}{4d_1}\left((144)+(155)\right)=\frac{1}{4d_2}\left((244)+(255)\right)=\frac{1}{2d_3}(345),
\end{split}
\end{equation}
where we used the equations in (\ref{eqn777}) for simplification.

On the other hand, we consider the left-invariant metrics on $\E_7$-II with respect to the decomposition (\ref{decom72}) as follows:
\begin{equation}\label{metric72}
(\ ,\ )_2=v_1\cdot B|_{A_1}+v_2\cdot B|_{D_6}+v_3\cdot B|_{\p'},
\end{equation}
if we let $x_1=v_1$, $u_0=x_2=x_4=v_2$ and $x_3=x_5=v_3$ in the metric (\ref{metric7}), then the components of Ricci tensor with respect to these two metrics are the same, which means if we denote the components of Ricci tensor with respect to the metric (\ref{metric72}) by $\tilde{r}'_1$, $\tilde{r}'_2$ and $\tilde{r}'_3$, we have $r_1=\tilde{r}'_1$, $r_0=r_2=r_4=\tilde{r}'_2$ and $r_3=r_5=\tilde{r}'_3$. With a short calculation, we have the following equations:
\begin{equation}\label{eqn72}
\begin{split}
&\frac{1}{4d_0}(044)=\frac{1}{4d_3}\left((222)+(244)\right)=\frac{1}{4}-\frac{1}{2d_4}(345),\\
&\frac{1}{4d_0}\left((033)+(055)\right)=\frac{1}{4d_2}\left((233)+(255)\right)=\frac{1}{2d_4}(345),
\end{split}
\end{equation}
where we used the equations in (\ref{eqn777}) for simplification.

By Lemma \ref{iii}, we can calculate that $(111)=\frac{1}{3}$ and $(222)=\frac{35}{3}$, with Table 1 in \cite{Ni}, we have $(345)=\frac{20}{3}$ and $d_0=1, d_1=1, d_2=35, d_3=24, d_4=30, d_5=40$, along with linear equations (\ref{eqn71}) and (\ref{eqn72}), we get the possible non-zero coefficients as follows:
\begin{equation}
\begin{split}
&(033)=\frac{4}{9}, (044)=\frac{5}{9}, (055)=0, (345)=\frac{20}{3},\\
&(111)=\frac{1}{3}, (133)=1, (144)=0, (155)=\frac{5}{3},\\
&(222)=\frac{35}{3}, (233)=\frac{35}{9}, (244)=\frac{70}{9}, (255)=\frac{35}{3}.
\end{split}
\end{equation}
\end{proof}

\textbf{Case of $\E_6$-II.} As is shown in \cite{ChKaLi}, $\E_6$-II can be decomposed as follows:
\begin{equation}\label{decom6}
\g=\T\oplus A_1^1\oplus A_1^2\oplus A_3\oplus\p_1\oplus\p_2\oplus\p_3,
\end{equation}
and we consider the following left-invariant metrics on $\E_6$ according to this decomposition:
\begin{equation}\label{metric6}
<\ ,\ >=u_0\cdot B|_{\T}+x_1\cdot B|_{A_1^1}+x_2\cdot B|_{A_1^2}+x_3\cdot B|_{A_3}+x_4\cdot B|_{\p_1}+x_5\cdot B|_{\p_2}+x_6\cdot B|_{\p_3},
\end{equation}
where $u_0, x_1, x_2, x_3, x_4, x_5, x_6\in\mathbb{R}^+$. Because of the structure of generalized Wallach space, the possible non-zero coefficients in the expressions for the components of Ricci tensor with respect to the metric (\ref{metric6}) are as follows:
$$(044), (055), (066), (111), (144), (155), (166), (222), (244), (255), (266), (333), (344), (355), (366), (456).$$

By Lemma \ref{formula1}, the components of Ricci tensor with respect to the metric (\ref{metric6}) can be simplified as follows:
\[\left\{\begin{aligned}
r_0&=\frac{1}{4d_0}\left(\frac{u_0}{x_4^2}(044)+\frac{u_0}{x_5^2}(055)+\frac{u_0}{x_6^2}(066)\right),\\
r_1&=\frac{1}{4d_1}\left(\frac{1}{x_1}(111)+\frac{x_1}{x_4^2}(144)+\frac{x_1}{x_5^2}(155)+\frac{x_1}{x_6^2}(166)\right),\\
r_2&=\frac{1}{4d_2}\left(\frac{1}{x_2}(222)+\frac{x_2}{x_4^2}(244)+\frac{x_2}{x_5^2}(255)+\frac{x_2}{x_6^2}(266)\right),\\
r_3&=\frac{1}{4d_3}\left(\frac{1}{x_3}(333)+\frac{x_3}{x_4^2}(344)+\frac{x_3}{x_5^2}(355)+\frac{x_3}{x_6^2}(366)\right),\\
r_4&=\frac{1}{2x_4}+\frac{1}{2d_4}(456)\left(\frac{x_4}{x_5x_6}-\frac{x_5}{x_6x_4}-\frac{x_6}{x_4x_5}\right)-\frac{1}{2d_4}\left(\frac{u_0}{x_4^2}(044)+\frac{x_1}{x_4^2}(144)+\frac{x_2}{x_4^2}(244)+\frac{x_3}{x_4^2}(344)\right),\\
r_5&=\frac{1}{2x_5}+\frac{1}{2d_5}(456)\left(\frac{x_5}{x_6x_4}-\frac{x_6}{x_4x_5}-\frac{x_4}{x_5x_6}\right)-\frac{1}{2d_5}\left(\frac{u_0}{x_5^2}(055)+\frac{x_1}{x_5^2}(155)+\frac{x_2}{x_5^2}(255)+\frac{x_3}{x_5^2}(355)\right),\\
r_6&=\frac{1}{2x_6}+\frac{1}{2d_6}(456)\left(\frac{x_6}{x_4x_5}-\frac{x_4}{x_5x_6}-\frac{x_5}{x_6x_4}\right)-\frac{1}{2d_6}\left(\frac{u_0}{x_6^2}(066)+\frac{x_1}{x_6^2}(166)+\frac{x_2}{x_6^2}(266)+\frac{x_3}{x_6^2}(366)\right).
\end{aligned}\right.\]
and
\begin{equation}\label{eqn666}
\begin{split}
&(044)+(055)+(066)=d_0,\\
&(111)+(144)+(155)+(166)=d_1,\\
&(222)+(244)+(255)+(266)=d_2,\\
&(333)+(344)+(355)+(366)=d_3,\\
&2(044)+2(144)+2(244)+2(344)+2(456)=d_4,\\
&2(055)+2(155)+2(255)+2(266)+2(456)=d_5,\\
&2(066)+2(166)+2(266)+2(366)+2(456)=d_6,
\end{split}
\end{equation}
where we used the symmetric properties of the indices in $(ijk)$.

\begin{lemma}\label{coeff6}
In the case of $\E_6$-II, the possible non-zero coefficients in the expression for the components of Ricci tensor with respect to metric (\ref{metric6}) are as follows:
\begin{equation*}
\begin{split}
&(044)=1/2, (055)=1/2, (066)=0, (456)=4,\\
&(111)=1/2, (144)=0, (155)=1, (166)=3/2,\\
&(222)=1/2, (244)=1, (255)=0, (266)=3/2,\\
&(333)=5, (344)=5/2, (355)=5/2, (366)=5.
\end{split}
\end{equation*}
\end{lemma}
\begin{proof}
In fact, according to \cite{ChKaLi}, there are two involutive automorphisms on $\E_6$-II, denoted by $\sigma$ and $\tau$, each of which corresponds an irreducible symmetric pair. Further, $\sigma$ corresponds to the irreducible symmetric pair $(\g,\ffb)$ with the following decomposition:
\begin{equation}\label{decom61}
\g=\ffb\oplus\p,\ \ffb=\ffb_1\oplus\ffb_2,\ \ffb_1=A_1^1,\ \ffb_2=\T\oplus A_1^2\oplus A_3\oplus\p_1\cong A_5,\ \p=\p_1\oplus\p_2,
\end{equation}
while $\tau$ corresponds to the irreducible symmetric pair $(\g,\b')$ with decomposition as follows:
\begin{equation}\label{decom62}
\g=\ffb'\oplus\p',\ \ffb'=\ffb'_1\oplus\ffb'_2,\ \ffb'_1=A_1^2,\ \ffb'_2=\T\oplus A_1^1\oplus A_3\oplus\p_2\cong A_5,\ \p'=\p_1\oplus\p_2.
\end{equation}
Then, we consider the following left-invariant metrics on $\E_6$ according to the decomposition (\ref{decom61}):
\begin{equation}\label{metric61}
(\ ,\ )_1=w_1\cdot B|_{A_1^1}+w_2\cdot B|_{A_5}+w_3\cdot B|_{\p}.
\end{equation}
If we let $u_0=x_2=x_3=x_4=w_2$, $x_1=w_1$ and $x_5=x_6=w_3$ in the metric (\ref{metric6}), then these two metrics are the same, as a result, the components of Ricci tensor with respect to these two metrics are equal respectively, which means if we denote the components of Ricci tensor with respect to metric (\ref{metric61}) by $\tilde{r}_1, \tilde{r}_2$ and $\tilde{r}_3$, then $r_0=r_2=r_3=r_4=\tilde{r}_2, r_1=\tilde{r}_1$ and $r_5=r_6=\tilde{r}_3$, with a short calculation, one can get the following system of equations:
\begin{equation}\label{eqn61}
\begin{split}
&\frac{1}{4d_0}(044)=\frac{1}{4d_2}\left((222)+(244)\right)=\frac{1}{4d_3}\left((333)+(344)\right)=\frac{1}{4}-\frac{1}{2d_4}(456),\\
&\frac{1}{4d_0}\left((055)+(066)\right)=\frac{1}{4d_2}\left((255)+(266)\right)=\frac{1}{4d_3}\left((355)+(366)\right)=\frac{1}{2d_4}(456),\\
&\frac{1}{d_5}(155)=\frac{1}{d_6}(166),\ (144)=0,
\end{split}
\end{equation}
where we used the equations in (\ref{eqn666}) for simplification.

On the other hand, we consider the following left-invariant metric on $\E_6$ with respect to the decomposition (\ref{decom62}):
\begin{equation}\label{metric62}
(\ ,\ )_2=v_1\cdot B|_{A_1^2}+v_2\cdot B|_{A_5}+v_3\cdot B|_{\p'}.
\end{equation}
If we let $u_0=x_1=x_3=x_5=v_2$, $x_2=v_1$ and $x_4=x_6=v_3$ in the metric (\ref{metric6}), then these two metrics are the same, as a result, the components of Ricci tensor with respect to these two metrics are equal respectively, which means if we denote the components of Ricci tensor with respect to metric (\ref{metric61}) by $\tilde{r}'_1, \tilde{r}'_2$ and $\tilde{r}'_3$, then $r_0=r_1=r_3=r_5=\tilde{r}'_2, r_2=\tilde{r}'_1$ and $r_4=r_6=\tilde{r}'_3$, with a short calculation, one can get the following system of equations:
\begin{equation}\label{eqn62}
\begin{split}
&\frac{1}{4d_0}(055)=\frac{1}{4d_1}\left((111)+(155)\right)=\frac{1}{4d_3}\left((333)+(355)\right)=\frac{1}{4}-\frac{1}{2d_5}(456),\\
&\frac{1}{4d_0}\left((044)+(066)\right)=\frac{1}{4d_1}\left((144)+(166)\right)=\frac{1}{4d_3}\left((344)+(366)\right)=\frac{1}{2d_5}(456),\\
&\frac{1}{d_6}(266)=\frac{1}{d_4}(244),\ (255)=0,
\end{split}
\end{equation}
where we used the equations in (\ref{eqn666}) for simplification.

By Lemma \ref{iii}, we have $(111)=(222)=\frac{1}{2}$ and $(333)=5$, from Table 1 in \cite{Ni}, we get $(456)=4$ and $d_0=1, d_1=d_2=3, d_3=15, d_4=d_5=16, d_6=24$, along with the equations (\ref{eqn61}) and (\ref{eqn62}), one can get the possible non-zero coefficients as given in the Lemma:
\end{proof}

\begin{remark}
Besides $\E_6$-III, for the other cases there are isomorphisms between the subalgebras of $\fk$ in the decomposition corresponding to the structure of generalized Wallach space, but these isomorphisms don't affect the behavior of the Ricci tensor. In particular, one can verify that $\R_{<\ ,\ >}(\fk_i,\fk_j)=0(i\neq j)$, where $\fk_i$ is isomorphism to $\fk_j$, As a result, $\R$ is still diagnal.
\end{remark}
\section{Discussions on non-naturally reductive Einstein metrics}
Now we have obtained the components of Ricci tensor for each case by Lemma \ref{p=2}, Lemma \ref{p=3} and Lemma \ref{p=4}, from which we will find non-naturally reductive Einstein metrics case by case in this section. 

\textbf{Case of $p=2$.} We will give the criterion to determine whether a left-invariant metric of the form (\ref{metric2}) is naturally reductive.
\begin{prop}\label{nrp=2}
If a left-invariant metric $<,>$ of the form (\ref{metric1}) on $G$ is naturally reductive with respect to $G\times L$ for some closed subgroup $L$ of $G$, then for the case of $p=2$, one of the following holds:\\
Case of $\E_6$-III: 1)$x_2=x_3$, $x_4=x_5$  2)$x_1=x_2=x_4$, $x_3=x_5$  3)$x_1=x_2=x_5$, $x_3=x_4$  4)$x_3=x_4=x_5$.  \\
Case of $\E_8$-II: 1)$x_1=x_2=x_3$, $x_4=x_5$  2)$x_1=x_2=x_4$, $x_3=x_5$  3)$x_1=x_2=x_5$, $x_3=x_5$  4)$x_3=x_4=x_5$.

Conversely, if one of  1), 2), 3), 4) is satisfied, then the metric of the form (\ref{metric2}) is naturally reductive with respect to $G\times L$ for some closed subgroup $L$ of $G$.
\end{prop}
\begin{proof}
Let $\fl$ be the Lie algebra of $L$. Then we have either $\fl\subset \fk$ or $\fl\not\subset\fk$. First we consider the case of $\fl\not\subset\fk$. Let $\frak{h}$ be the subalgebra of $\g$ generated by $\fl$ and $\fk$. Since $\g=\fk_1\oplus\fk_2\oplus\m_1\oplus\m_2\oplus\m_3$ for the case of $p=2$, $\frak{h}$ must contains only one of $\m_1,\m_2$ and $\m_3$. If $\m_1\subset\frak{h}$, then $\fk\oplus\m_1$ is a subalgebra of $\g$ according to \cite{ChKaLi}. In fact, for case of $\E_6$-III, it is isomorphic to $A_1\oplus A_5$, where $\fk_2\oplus\m_1\cong A_5$. Therefore by Theorem \ref{nr}, we have $x_2=x_3$, $x_4=x_5$. Similarly, for case of $\E_8$-II, we have $x_1=x_2=x_3$, $x_4=x_5$. By a similar way, we can obtain 2) and 3) in the proposition for each case.

Now we consider the case $\fl\subset\fk$. Since the orthogonal complement $\fl^{\perp}$ of $\fl$ with respect to $B$ contains the orthogonal complement $\fk^{\perp}$ of $\fk$, we see that $\m_1\oplus\m_2\oplus\m_3\subset\fl^{\perp}$. Since the invariant metric $<,>$ is naturally reductive with respect to $G\times L$, it follows that $x_3=x_4=x_5$ by Theorem \ref{nr}. The converse is a direct consequence of Theorem \ref{nr}.
\end{proof}

\textbf{Case of $\E_6$-III.} According Lemma \ref{p=2}, we have
\begin{equation*}
\begin{split}
r_1&=\frac{1}{12}\left(\frac{1}{2x_1}+\frac{7x_1}{4x_4^2}+\frac{3x_1}{4x_5^2}\right),\\
r_2&=\frac{1}{84}\left(\frac{7}{x_2}+\frac{7x_2}{2x_3^2}+\frac{35x_2}{4x_4^2}+\frac{7x_2}{4x_5^2}\right),\\
r_3&=\frac{1}{2x_3}+\frac{1}{8}\left(\frac{x_3}{x_4x_5}-\frac{x_4}{x_3x_5}-\frac{x_5}{x_3x_4}\right)-\frac{1}{8}\frac{x_2}{x_3^2},\\
r_4&=\frac{1}{2x_4}+\frac{1}{16}\left(\frac{x_4}{x_3x_5}-\frac{x_3}{x_4x_5}-\frac{x_5}{x_3x_4}\right)-\frac{1}{56}\left(\frac{7x_1}{4x_4^2}+\frac{35x_2}{4x_4^2}\right),\\
r_5&=\frac{1}{2x_5}+\frac{7}{48}\left(\frac{x_5}{x_4x_3}-\frac{x_4}{x_3x_5}-\frac{x_3}{x_5x_4}\right)-\frac{1}{24}\left(\frac{3x_1}{4x_5^2}+\frac{7x_2}{4x_5^2}\right).
\end{split}
\end{equation*}
We consider the system of equations
\begin{equation}\label{eq21}
r_1-r_2=0, r_2-r_3=0, r_3-r_4=0, r_4-r_5=0.
\end{equation}
Then finding Einstein metrics of the form (\ref{metric2}) reduces to finding the positive solutions of system (\ref{eq21}), and we normalize the equations by putting $x_5 = 1$. Then we obtain the system of equations:
\begin{equation}\label{eq211}
\begin{split}
g_1=&3\,{x_{{1}}}^{2}{x_{{4}}}^{2}x_{{2}}{x_{{3}}}^{2}-x_{{1}}{x_{{2}}}^{2}{x_{{3}}}^{2}{x_{{4}}}^{2}+7\,{x_{{1}}}^{2}x_{{2}}{x_{{3}}}^{2}-5\,{x_{{2}}}^{2}x_{{1}}{x_{{3}}}^{2}\\
&-2\,{x_{{2}}}^{2}x_{{1}}{x_{{4}}}^{2}-4\,x_{{1}}{x_{{4}}}^{2}{x_{{3}}}^{2}+2\,x_{{2}}{x_{{4}}}^{2}{x_{{3}}}^{2}=0,\\
g_2=&{x_{{2}}}^{2}{x_{{3}}}^{2}{x_{{4}}}^{2}-6\,x_{{2}}{x_{{3}}}^{3}x_{{4}}+6\,x_{{2}}x_{{3}}{x_{{4}}}^{3}+5\,{x_{{2}}}^{2}{x_{{3}}}^{2}\\
&+8\,{x_{{2}}}^{2}{x_{{4}}}^{2}-24\,x_{{2}}x_{{3}}{x_{{4}}}^{2}+4\,{x_{{4}}}^{2}{x_{{3}}}^{2}+6\,x_{{2}}x_{{3}}x_{{4}}=0,\\
g_3=&6\,{x_{{3}}}^{3}x_{{4}}-6\,{x_{{4}}}^{3}x_{{3}}+x_{{1}}{x_{{3}}}^{2}+5\,x_{{2}}{x_{{3}}}^{2}-4\,x_{{2}}{x_{{4}}}^{2}\\
&-16\,x_{{4}}{x_{{3}}}^{2}+16\,{x_{{4}}}^{2}x_{{3}}-2\,x_{{3}}x_{{4}}=0,\\
g_4=&3\,x_{{1}}{x_{{4}}}^{2}x_{{3}}+7\,x_{{2}}x_{{3}}{x_{{4}}}^{2}+8\,x_{{4}}{x_{{3}}}^{2}-48\,{x_{{4}}}^{2}x_{{3}}\\
&+20\,{x_{{4}}}^{3}-3\,x_{{1}}x_{{3}}-15\,x_{{2}}x_{{3}}+48\,x_{{3}}x_{{4}}-20\,x_{{4}}=0.
\end{split}
\end{equation}
We consider a polynomial ring $R=\mathbb{Q}[z, x_1, x_2, x_3, x_4]$ and an ideal $I$ generated by $\{g_1, g_2, g_3, g_4, \\z x_1 x_2 x_3 x_4-1\}$ to find non-zero solutions of equations (\ref{eq211}). We take a lexicographic order $>$ with $z>x_1 >x_2 >x_3 >x_4$ for a monomial ordering on $R$. Then with the aid of computer, the following polynomial is contained in the Gr\"{o}bner basis for the ideal $I$
\begin{equation*}
\left( x_{{4}}-1 \right)  \left( 5\,x_{{4}}-3 \right)  \left( 5\,x_{{4}}-19 \right) h(x_4),
\end{equation*}
where $h(x_4)$ is of the form
\begin{equation}
\begin{split}
&620527834748568712226625\,{x_{{4}}}^{46}-17142288459030942157682550\,{x_{{4}}}^{45}\\
&+253864478218386260238125175\,{x_{{4}}}^{44}-2655902456682476684982068196\,{x_{{4}}}^{43}\\
&+21712791139584353835485509485\,{x_{{4}}}^{42}-146403945857203174056695826906\,{x_{{4}}}^{41}\\
&+841869064160931856135565647035\,{x_{{4}}}^{40}-4221496823183250288515785683288\,{x_{{4}}}^{39}\\
&+18759663705905743422883726751607\,{x_{{4}}}^{38}-74772865457920384779255097278978\,{x_{{4}}}^{37}\\
&+269810159883362340386858437110705\,{x_{{4}}}^{36}-887921763230246899234386977810964\,{x_{{4}}}^{35}\\
&+2680934636604177050138806853307267\,{x_{{4}}}^{34}-7463374027396529763571324631419086\,{x_{{4}}}^{33}\\
&+19236417063475016928618367461353205\,{x_{{4}}}^{32}-46067840601646061767106985652544544\,{x_{{4}}}^{31}\\
&+102827687516709239196388118203922250\,{x_{{4}}}^{30}-214524790280392516306729959721167612\,{x_{{4}}}^{29}\\
&+419397032805870597215879980744285542\,{x_{{4}}}^{28}-
770243915769947385884764661220911880\,{x_{{4}}}^{27}\\&+
1332121344976070900463276148279367906\,{x_{{4}}}^{26}-
2174890268260219400375855326539637796\,{x_{{4}}}^{25}\\&+
3360427246597106223273731326789411118\,{x_{{4}}}^{24}-
4926216667916456054561642964481111312\,{x_{{4}}}^{23}\\&+
6868691169547446497713714767840930254\,{x_{{4}}}^{22}-
9130195555722867119052619469924183268\,{x_{{4}}}^{21}\\&+
11592691118354371158751244714914822658\,{x_{{4}}}^{20}-
14080199932731634104314039587318940936\,{x_{{4}}}^{19}\\&+
16371489061442382163179799619487885062\,{x_{{4}}}^{18}-
18224385478214023428533707109785905340\,{x_{{4}}}^{17}\\&+
19411619002013319176816159460886119530\,{x_{{4}}}^{16}-
19763525846081914361620786672655930400\,{x_{{4}}}^{15}\\&+
19206776730802850981024169581516423525\,{x_{{4}}}^{14}-
17785423401614562582040867258711520750\,{x_{{4}}}^{13}\\&+
15655271805998900455634624058209469875\,{x_{{4}}}^{12}-
13053707551812208998459860496371252500\,{x_{{4}}}^{11}\\&+
10256813778567948010832363355641230625\,{x_{{4}}}^{10}-
7536513577438512742744940874573881250\,{x_{{4}}}^{9}\\&+
5123727743209309471365339473434734375\,{x_{{4}}}^{8}-
3177971936471765628211397123442375000\,{x_{{4}}}^{7}\\&+
1766171205209537862901035966120796875\,{x_{{4}}}^{6}-
859452515630617218777390718317656250\,{x_{{4}}}^{5}\\&+
355236019137782072058927844356328125\,{x_{{4}}}^{4}-
119504109722879595751903187339062500\,{x_{{4}}}^{3}\\&+
30626968356732968210488474290234375\,{x_{{4}}}^{2}-
5304673599935287496386273558593750\,x_{{4}}\\&+
463705449010204912215012369140625
\end{split}
\end{equation}

In fact, in the Gr\"{o}bner basis of the ideal $I$, $x_1,x_2$ and $x_3$ can be written into polynomials of $x_4$. By solving $h(x_4)=0$, we have four solutions, namely $0.6711159524, 0.8439629969, 0.9167404817,2.171597540$ and all corresponding solutions of the system of equations $\{g_1=0,g_2=0,g_3=0,g_4=0,h(x_4)=0\}$ with $x_1x_2x_3x_4\neq0$ are as follows:
\begin{equation*}
\begin{split}
&\{x_1\approx0.1550362575,x_2\approx0.7478555199,x_3\approx1.042517676,x_4\approx0.6711159524\},\\
&\{x_1\approx1.366119112,x_2\approx0.2521439928,x_3\approx0.6761926469,x_4\approx0.8439629969\},\\
&\{x_1\approx0.09898950458,x_2\approx0.2192177752,x_3\approx0.5309066187,x_4\approx0.9167404817\},\\
&\{x_1\approx0.2432173551,x_2\approx0.4934849553,x_3\approx2.270522057,x_4\approx2.171597540\}.
\end{split}
\end{equation*}
Due to Proposition \ref{nrp=2}, we conclude that each of these four solutions induces a non-naturally reductive Einstein metric.\\
For $x_4=1$, the system $\{g_1=0,g_2=0,g_3=0,g_4=0\}$ has the following four solutions:
\begin{equation*}
\begin{split}
&\{x_1=1,x_2=1,x_3=1,x_4=1\},\\
&\{x_1\approx0.1030504001,x_2\approx0.3244706112,x_3\approx0.3244706112,x_4=1\},\\
&\{x_1\approx1.613068224,x_2\approx0.4009377358,x_3\approx0.4009377358,x_4=1\},\\
&\{x_1\approx0.1984241041,x_2\approx1.108447830,x_3\approx1.108447830,x_4=1\}.
\end{split}
\end{equation*}
For $x_4=\frac{3}{5}$, the system $\{g_1=0,g_2=0,g_3=0,g_4=0\}$ has only one solution given by
\begin{equation*}
\{x_1=x_2=x_4=\frac{3}{5},x_3=1\},
\end{equation*}
and for $x_4=\frac{19}{5}$, the system $\{g_1=0,g_2=0,g_3=0,g_4=0\}$ also has only one solution given by
\begin{equation*}
\{x_1=x_2=1,x_3=x_4=\frac{19}{5}\}.
\end{equation*}
According to Proposition \ref{nrp=2}, these corresponding left-invariant Einstein metrics are all naturally reductive.

In summarize, we find 4 different non-naturally reductive left-invariant Einstein metrics on $\E_6$-III.

\textbf{Case of $\E_8$-II.} According Lemma \ref{p=2}, we have
\begin{equation*}
\begin{split}
r_1&=\frac{1}{112}\left(\frac{28}{5x_1}+\frac{112x_1}{15x_3^2}+\frac{112x_1}{15x_4^2}+\frac{112x_1}{15x_5^2}\right),\\
r_2&=\frac{1}{112}\left(\frac{28}{5x_2}+\frac{112x_2}{15x_3^2}+\frac{112x_2}{15x_4^2}+\frac{112x_2}{15x_5^2}\right),\\
r_3&=\frac{1}{2x_3}+\frac{2}{15}\left(\frac{x_3}{x_4x_5}-\frac{x_4}{x_3x_5}-\frac{x_5}{x_3x_4}\right)-\frac{1}{128}\left(\frac{112x_1}{15x_3^2}+\frac{112x_2}{15x_3^2}\right),\\
r_4&=\frac{1}{2x_4}+\frac{2}{15}\left(\frac{x_4}{x_3x_5}-\frac{x_3}{x_4x_5}-\frac{x_5}{x_3x_4}\right)-\frac{1}{128}\left(\frac{112x_1}{15x_4^2}+\frac{112x_2}{15x_4^2}\right),\\
r_5&=\frac{1}{2x_5}+\frac{2}{15}\left(\frac{x_5}{x_4x_3}-\frac{x_4}{x_3x_5}-\frac{x_3}{x_5x_4}\right)-\frac{1}{128}\left(\frac{112x_1}{15x_5^2}+\frac{112x_2}{15x_5^2}\right),
\end{split}
\end{equation*} 
We consider the system of equations given by 
\begin{equation}\label{eq22}
r_1-r_2=0, r_2-r_3=0, r_3-r_4=0, r_4-r_5=0,
\end{equation}
and normalize the equations by putting $x_5 = 1$, then we obtain the system of equations:
\begin{equation}\label{eq222}
\begin{split}
g_1=&4\,{x_{{1}}}^{2}{x_{{3}}}^{2}{x_{{4}}}^{2}x_{{2}}-4\,x_{{1}}{x_{{2}}}^{2}{x_{{3}}}^{2}{x_{{4}}}^{2}+4\,{x_{{1}}}^{2}{x_{{3}}}^{2}x_{{2}}+4\,{x_{{1}}}^{2}{x_{{4}}}^{2}x_{{2}}\\&-4\,{x_{{2}}}^{2}x_{{1}}{x_{{3}}}^{2}-4\,{x_{{2}}}^{2}x_{{1}}{x_{{4}}}^{2}-3x_{{1}}{x_{{3}}}^{2}{x_{{4}}}^{2}+3x_{{2}}{x_{{3}}}^{2}{x_{{4}}}^{2}=0,\\
g_2=&8\,{x_{{2}}}^{2}{x_{{3}}}^{2}{x_{{4}}}^{2}-16\,{x_{{3}}}^{3}x_{{2}}x_{{4}}+16\,{x_{{4}}}^{3}x_{{2}}x_{{3}}+7\,x_{{1}}x_{{2}}{x_{{4}}}^{2}+8\,{x_{{2}}}^{2}{x_{{3}}}^{2}\\&+15\,{x_{{2}}}^{2}{x_{{4}}}^{2}-60\,x_{{2}}x_{{3}}{x_{{4}}}^{2}+6\,{x_{{3}}}^{2}{x_{{4}}}^{2}+16\,x_{{2}}x_{{3}}x_{{4}}=0,\\
g_3=&32\,{x_{{3}}}^{3}x_{{4}}-32\,{x_{{4}}}^{3}x_{{3}}+7\,x_{{1}}{x_{{3}}}^{2}-7\,x_{{1}}{x_{{4}}}^{2}+7\,x_{{2}}{x_{{3}}}^{2}-7\,x_{{2}}{x_{{4}}}^{2}\\&-60\,x_{{4}}{x_{{3}}}^{2}+60\,x_{{3}}{x_{{4}}}^{2}=0,\\
g_4=&7\,x_{{1}}x_{{3}}{x_{{4}}}^{2}+7\,x_{{2}}x_{{3}}{x_{{4}}}^{2}-60\,x_{{3}}{x_{{4}}}^{2}+32\,{x_{{4}}}^{3}-7\,x_{{1}}x_{{3}}\\&-7\,x_{{2}}x_{{3}}+60\,x_{{3}}x_{{4}}-32\,x_{{4}}=0.
\end{split}
\end{equation}
We consider a polynomial ring $R=\mathbb{Q}[z, x_1, x_2, x_3, x_4]$ and an ideal $I$ generated by $\{g_1, g_2, g_3, g_4, \\z x_1 x_2 x_3 x_4-1\}$ to find non-zero solutions of equations (\ref{eq222}). We take a lexicographic order $>$ with $z>x_1 >x_2 >x_3 >x_4$ for a monomial ordering on $R$. Then with the aid of computer, the following polynomial is contained in the Gr\"{o}bner basis for the ideal $I$
\begin{equation*}
\left( x_{{4}}-1 \right)  \left( 23\,x_{{4}}-7 \right)  \left( 7\,x_{{4}}-23 \right) h(x_4),
\end{equation*}
where 
\begin{equation*}
\begin{split}
h(x_4)=&18820892214681403392\,{x_{{4}}}^{24}-106573710368905887744\,{x_{{4}}}^{23}\\
&+367021480848929587200\,{x_{{4}}}^{22}-989697149383674494976\,{x_{{4}}}^{21}\\
&+1859094664559751753728\,{x_{{4}}}^{20}-
3257511072225679640576\,{x_{{4}}}^{19}\\&+4280088309639423272992\,{x_{{4}
}}^{18}-5679995572440505667140\,{x_{{4}}}^{17}\\&+6595970829340416842428
\,{x_{{4}}}^{16}-7511322681489787363579\,{x_{{4}}}^{15}\\&+
9419511263909486275350\,{x_{{4}}}^{14}-9260548321425771133485\,{x_{{4}
}}^{13}\\&+11189718816841142104820\,{x_{{4}}}^{12}-9260548321425771133485
\,{x_{{4}}}^{11}\\&+9419511263909486275350\,{x_{{4}}}^{10}-
7511322681489787363579\,{x_{{4}}}^{9}\\&+6595970829340416842428\,{x_{{4}}
}^{8}-5679995572440505667140\,{x_{{4}}}^{7}\\&+4280088309639423272992\,{x
_{{4}}}^{6}-3257511072225679640576\,{x_{{4}}}^{5}\\&+
1859094664559751753728\,{x_{{4}}}^{4}-989697149383674494976\,{x_{{4}}}
^{3}\\&+367021480848929587200\,{x_{{4}}}^{2}-106573710368905887744\,x_{{4
}}\\&+18820892214681403392.
\end{split}
\end{equation*}

By solving $h(x_4)=0$ numerically, we find positive four solutions which are given approximately by $x_4\approx0.3526915707$(we state this solution will make $x_2$ negative), $x_4\approx0.7261283537$, $x_4\approx2.835338531$ and $x_4\approx1.377166991$ and we split the corresponding solutions of the system of equations $\{g_1=0,g_2=0,g_3=0,g_4=0,h(x_4)=0\}$ with $x_1x_2x_3x_4\neq0$ into two groups as follows:
\[\mbox{Group 1.}\left\{\begin{aligned}
&\{x_1\approx1.304885525,x_2\approx0.4602586724,x_3\approx2.835338531,x_4\approx2.835338531\},\\
&\{x_1\approx0.4602586724,x_2\approx1.304885525,x_3\approx2.835338531,x_4\approx2.835338531\}.
\end{aligned}\right.\]

\[\mbox{Group 2.}\left\{\begin{aligned}
&\{x_1\approx0.1431443064,x_2\approx0.1431443064,x_3=1,x_4\approx0.7261283537\},\\
&\{x_1\approx0.1971336881,x_2\approx0.1971336881,x_3\approx1.377166991,x_4\approx1.377166991\}.
\end{aligned}\right.\]

For $x_4=1$,  the system $\{g_1=0,g_2=0,g_3=0,g_4=0\}$ has five solutions which can be split into the following three groups:
\[\mbox{Group 3.}\begin{aligned}
\{x_1=x_2=x_3=x_4=1\},
\end{aligned}\]
\[\mbox{Group 4.}\begin{aligned}
\{x_1=x_2=x_3=\frac{7}{23},x_4=1\},
\end{aligned}\]
\[\mbox{Group 5.}\left\{\begin{aligned}
&\{x_1\approx0.4602221254,x_2\approx0.1623293541,x_3\approx0.3526915707,x_4=1\},\\
&\{x_1\approx0.1623293541,x_2\approx0.4602221254,x_3\approx0.3526915707,x_4=1\}.
\end{aligned}\right.\]
\[\mbox{Group 6.}\begin{aligned}
\{x_1\approx0.1431443064,x_2\approx0.1431443064,x_3\approx0.7261283537,x_4=1\}.
\end{aligned}\]

For $x_4=\frac{7}{23}$ and $x_4=\frac{23}{7}$, the corresponding solutions of the system of equations $\{g_1=0,g_2=0,g_3=0,g_4=0\}$ are as follows respectively:
\[\mbox{Group 7.}\left\{\begin{aligned}
&\{x_1=x_2=x_4=\frac{7}{23},x_3=1\},\\
&\{x_1=x_2=1,x_3=x_4=\frac{23}{7}\}.
\end{aligned}\right.\]

Among these solutions, we remark that the solution in Group 3 induces the Killing metric, the solutions in Group 4 and Group 7 induce the same metrics  up to isometry which are naturally reductive due to Proposition \ref{nrp=2}, while the solutions in Group 1 and Group 5 induce the same metrics and the solutions in Group 2 and Group 6 also induce the same metrics  up to isometry which are all non-naturally reductive due to Proposition \ref{nrp=2}. 

Therefore, we find 2 non-naturally reductive Einstein metrics on $\E_8$-II.  

\textbf{Case of $p=3$}. With the similar reason, we give the following proposition to decide whether a left-invariant metric is naturally reductive.
\begin{prop}\label{nrp=3}
If a left-invariant metric $<,>$ of the form (\ref{metric1}) on $G$ is naturally reductive with respect to $G\times L$ for some closed subgroup $L$ of $G$, then for the case of $p=3$, one of the following holds: 1)$x_1=x_2=x_3=x_4$, $x_5=x_6$  2)$x_2=x_3=x_5$, $x_4=x_6$  3)$x_1=x_3=x_6$, $x_4=x_5$  4)$x_4=x_5=x_6$. 

Conversely, if one of  1), 2), 3), 4) is satisfied, then the metric of the form (\ref{metric3}) for the case of $p=3$ is naturally reductive with respect to $G\times L$ for some closed subgroup $L$ of $G$.
\end{prop}

\textbf{Case of $\F_4$-II.} According Lemma \ref{p=3}, we have
\begin{equation*}
\begin{split}
r_1&=\frac{1}{12}\left(\frac{2}{3x_1}+\frac{5x_1}{3x_4^2}+\frac{2x_1}{3x_6^2}\right),\\
r_2&=\frac{1}{12}\left(\frac{2}{3x_2}+\frac{5x_2}{3x_4^2}+\frac{2x_2}{3x_5^2}\right),\\
r_3&=\frac{1}{40}\left(\frac{10}{3x_3}+\frac{40x_3}{9x_4^2}+\frac{10x_3}{9x_5^2}+\frac{10x_3}{9x_6^2}\right),\\
r_4&=\frac{1}{2x_4}+\frac{1}{18}\left(\frac{x_4}{x_5x_6}-\frac{x_5}{x_6x_4}-\frac{x_6}{x_4x_5}\right)-\frac{1}{40}\left(\frac{5x_1}{3x_4^2}+\frac{5x_2}{3x_4^2}+\frac{40x_3}{9x_4^2}\right),\\
r_5&=\frac{1}{2x_5}+\frac{5}{36}\left(\frac{x_5}{x_6x_4}-\frac{x_6}{x_4x_5}-\frac{x_4}{x_5x_6}\right)-\frac{1}{16}\left(\frac{2x_2}{3x_5^2}+\frac{10x_3}{9x_5^2}\right),\\
r_6&=\frac{1}{2x_6}+\frac{5}{36}\left(\frac{x_6}{x_5x_4}-\frac{x_5}{x_6x_4}-\frac{x_4}{x_5x_6}\right)-\frac{1}{16}\left(\frac{2x_1}{3x_6^2}+\frac{9x_3}{9x_6^2}\right).
\end{split}
\end{equation*}
We consider the system of equations
\begin{equation}\label{eq31}
r_1-r_2=0, r_2-r_3=0, r_3-r_4=0, r_4-r_5=0,r_5-r_6=0.
\end{equation}
Then finding Einstein metrics of the form (\ref{metric3}) reduces to finding the positive solutions of system (\ref{eq31}), and we normalize the equations by putting $x_4 = 1$. Then we obtain the system of equations:
\begin{equation}\label{eq311}
\begin{split}
g_1=&5\,{x_{{1}}}^{2}x_{{2}}{x_{{5}}}^{2}{x_{{6}}}^{2}-5\,x_{{1}}{x_{{2}}}^
{2}{x_{{5}}}^{2}{x_{{6}}}^{2}+2\,{x_{{1}}}^{2}x_{{2}}{x_{{5}}}^{2}-2\,
x_{{1}}{x_{{2}}}^{2}{x_{{6}}}^{2}-2\,x_{{1}}{x_{{5}}}^{2}{x_{{6}}}^{2}
+2\,x_{{2}}{x_{{5}}}^{2}{x_{{6}}}^{2}=0,\\
g_2=&5\,{x_{{2}}}^{2}x_{{3}}{x_{{5}}}^{2}{x_{{6}}}^{2}-4\,x_{{2}}{x_{{3}}}^
{2}{x_{{5}}}^{2}{x_{{6}}}^{2}+2\,{x_{{2}}}^{2}x_{{3}}{x_{{6}}}^{2}-x_{
{2}}{x_{{3}}}^{2}{x_{{5}}}^{2}-x_{{2}}{x_{{3}}}^{2}{x_{{6}}}^{2}-3\,x_
{{2}}{x_{{5}}}^{2}{x_{{6}}}^{2}\\&+2\,x_{{3}}{x_{{5}}}^{2}{x_{{6}}}^{2}=0,\\
g_3=&-3\,x_{{1}}{x_{{5}}}^{2}x_{{6}}-3\,x_{{2}}{x_{{5}}}^{2}x_{{6}}-8\,x_{{
3}}{x_{{5}}}^{2}x_{{6}}-14\,{x_{{5}}}^{3}+36\,{x_{{5}}}^{2}x_{{6}}+6\,
x_{{5}}{x_{{6}}}^{2}+3\,x_{{2}}x_{{6}}+5\,x_{{3}}x_{{6}}\\&-36\,x_{{5}}x_
{{6}}+14\,x_{{5}}=0,\\
g_4=&-14\,{x_{{4}}}^{3}x_{{5}}+14\,x_{{4}}{x_{{5}}}^{3}+3\,x_{{1}}{x_{{5}}}^{2}-3\,x_{{2}}{x_{{4}}}^{2}+3\,x_{{2}}{x_{{5}}}^{2}-5\,x_{{3}}{x_{{4}}}^{2}+8\,x_{{3}}{x_{{5}}}^{2}\\
&+36\,{x_{{4}}}^{2}x_{{5}}-36\,x_{{4}}{x_{{5}}}^{2}-6\,x_{{4}}x_{{5}}=0,\\
g_5=&20\,{x_{{5}}}^{3}x_{{6}}-20\,x_{{5}}{x_{{6}}}^{3}+3\,x_{{1}}{x_{{5}}}^
{2}-3\,x_{{2}}{x_{{6}}}^{2}+5\,x_{{3}}{x_{{5}}}^{2}-5\,x_{{3}}{x_{{6}}
}^{2}-36\,{x_{{5}}}^{2}x_{{6}}+36\,x_{{5}}{x_{{6}}}^{2}=0.
\end{split}
\end{equation}
We consider a polynomial ring $R=Q[z, x_1, x_2, x_3, x_5,x_6]$ and an ideal $I$ generated by $\{g_1, g_2, g_3, g_4,g_5, \\z x_1 x_2 x_3 x_5x_6-1\}$ to find non-zero solutions of equations (\ref{eq311}). We take a lexicographic order $>$ with $z>x_1 >x_2 >x_3 >x_5>x_6$ for a monomial ordering on $R$. Then with the help of computer, the following polynomial is contained in the Gr\"{o}bner basis for the ideal $I$ 
\begin{equation}
( x_{{6}}-1 ) ( 7\,x_{{6}}-11 ) ( 2375\,
{x_{{6}}}^{3}-4195\,{x_{{6}}}^{2}+1960\,x_{{6}}-272
)\cdot h(x_6),
\end{equation}
where $h(x_6)$ is a polynomial of $x_6$ of degree 114. We put it in Appendix I for readers' convenience.

For $x_6=1$, with the polynomials in Gr\"{o}bner basis, we get four values of $x_5$, namely $0.2797176824$, $0.3650688296$, $1$, $1.121529277$, whose corresponding solutions of the system $\{g_1=0,g_2=0,g_3=0,g_4=0,g_5=0\}$ can be given as follows:
\[\mbox{Group 1.}\begin{aligned}
\{x_1=x_2=x_3=x_5=x_6=1\}.
\end{aligned}\]
\[\mbox{Group 2.}\left\{\begin{aligned}
&\{x_1\yd0.1355974584,x_2=x_3=x_5\yd0.2797176824,x_6=1\},\\
&\{x_1\yd1.653201132,x_2=x_3=x_5\yd0.3650688296,x_6=1\},\\
&\{x_1\yd0.2762168891,x_2=x_3=x_5\yd1.121529277,x_6=1\}.
\end{aligned}\right.\]

For $x_6=\frac{11}{7}$, we substitute it into the polynomials in the Gr\"{o}bner basis, then we get the following corresponding solutions of the system $\{g_1=0,g_2=0,g_3=0,g_4=0,g_5=0\}$ with $x_1x_2x_3x_5x_6\neq0$:
\[\mbox{Group 3.}\begin{aligned}
\{x_1=x_2=x_3=1,x_5=x_6=\frac{11}{7}\}.
\end{aligned}\]

By solving $x_{{6}}-1 ) ( 7\,x_{{6}}-11 ) ( 2375\,{x_{{6}}}^{3}-4195\,{x_{{6}}}^{2}+1960\,x_{{6}}-272=0$ numerically, there are three positive solutions given approximately by $x_6\yd0.2797176824$, $x_6\yd0.3650688296$ and $x_6\yd1.121529277$, whose corresponding solutions of the system $\{g_1=0,g_2=0,g_3=0,g_4=0,g_5=0,x_{{6}}-1 ) ( 7\,x_{{6}}-11 ) ( 2375\,{x_{{6}}}^{3}-4195\,{x_{{6}}}^{2}+1960\,x_{{6}}-272=0\}$ with $x_1x_2x_3x_5x_6\neq0$
are:
\[\mbox{Group 3.}\left\{\begin{aligned}
&\{x_2\yd0.1355974584,x_1=x_3=x_6\yd0.2797176824,x_5=1\},\\
&\{x_2\yd1.653201132,x_1=x_3=x_6\yd0.3650688296,x_5=1\},\\
&\{x_2\yd0.2762168891,x_1=x_3=x_6\yd1.121529277,x_5=1\}.
\end{aligned}\right.\]

By solving $h(x_6)=0$ numerically, we get 6 different positive solutions, namely $0.4941864913$, $0.7403305751$, $1.068217773$, $1.160571982$, $1.345214992$, $1.422410517$, whose corresponding solutions of the system $\{g_1=0,g_2=0,g_3=0,g_4=0,g_5=0,h(x_6)=0\}$ with $x_1x_2x_3x_5x_6\neq0$ can be split into the following three groups:
\[\mbox{Group 4.}\left\{\begin{aligned}
&\{x_1\yd0.1516435461,x_2\yd0.1404443065,x_3\yd0.2282956381,x_5\yd1.068217773,x_6\yd0.4941864913\},\\
&\{x_1\yd0.1404443065,x_2\yd0.1516435461,x_3\yd0.2282956381,x_5\yd0.4941864913,x_6\yd1.068217773\}.
\end{aligned}\right.\]
\[\mbox{Group 5.}\left\{\begin{aligned}
&\{x_1\yd0.1951256737,x_2\yd1.654400436,x_3\yd0.3012253093,x_5\yd1.160571982,x_6\yd0.7403305751\},\\
&\{x_1\yd1.654400436,x_2\yd0.1951256737,x_3\yd0.3012253093,x_5\yd0.7403305751,x_6\yd1.160571982\}.
\end{aligned}\right.\]
\[\mbox{Group 6.}\left\{\begin{aligned}
&\{x_1\yd0.3075015814,x_2\yd1.094420015,x_3\yd1.138681932,x_5\yd1.422410516,x_6\yd1.345214992\},\\
&\{x_1\yd1.094420015,x_2\yd0.3075015814,x_3\yd1.138681932,x_5\yd1.345214992,x_6\yd1.422410516\}.
\end{aligned}\right.\]

Among these metrics, we remark that the left-invariant Einstein metrics induced by the solutions in Group1, 2, 3 are all naturally reductive, while the left-invariant Einstein metrics induced by the solutions in Group 4, 5, 6 are all non-naturally reductive due to Proposition \ref{nrp=3}. In particular, the solutions in Group 2 and 3 induce the same metrics up to isometry respectively and the solutions in each of Group 4-6 induce a same metric up to isometry.

In conclusion, we find 3 different non-naturally reductive left-invariant Einstein metrics on $\F_4$-II.

\textbf{Case of $\E_8$-I.}  According Lemma \ref{p=3}, we have
\begin{equation}\label{eqnE8I}
\begin{split}
r_1&=\frac{1}{12}\left(\frac{1}{5x_1}+\frac{6x_1}{5x_4^2}+\frac{8x_1}{5x_6^2}\right),\\
r_2&=\frac{1}{12}\left(\frac{1}{5x_2}+\frac{6x_2}{5x_4^2}+\frac{8x_2}{5x_5^2}\right),\\
r_3&=\frac{1}{264}\left(\frac{22}{x_3}+\frac{44x_3}{5x_4^2}+\frac{88x_3}{5x_5^2}+\frac{88x_3}{5x_6^2}\right),\\
r_4&=\frac{1}{2x_4}+\frac{2}{15}\left(\frac{x_4}{x_5x_6}-\frac{x_5}{x_6x_4}-\frac{x_6}{x_4x_5}\right)-\frac{1}{96}\left(\frac{6x_1}{5x_4^2}+\frac{6x_2}{5x_4^2}+\frac{44x_3}{5x_4^2}\right),\\
r_5&=\frac{1}{2x_5}+\frac{1}{10}\left(\frac{x_5}{x_6x_4}-\frac{x_6}{x_4x_5}-\frac{x_4}{x_5x_6}\right)-\frac{1}{128}\left(\frac{8x_2}{5x_5^2}+\frac{88x_3}{88x_5^2}\right),\\
r_6&=\frac{1}{2x_6}+\frac{1}{10}\left(\frac{x_6}{x_5x_4}-\frac{x_5}{x_6x_4}-\frac{x_4}{x_5x_6}\right)-\frac{1}{128}\left(\frac{8x_1}{5x_6^2}+\frac{88x_3}{5x_6^2}\right).
\end{split}
\end{equation}
We consider the system of equations
\begin{equation}\label{eq32}
r_1-r_2=0, r_2-r_3=0, r_3-r_4=0, r_4-r_5=0, r_5-r_6=0.
\end{equation}
Then finding Einstein metrics of the form (\ref{metric3}) reduces to finding the positive solutions of system (\ref{eq32}), and we normalize the equations by putting $x_4 = 1$. Then we obtain the system of equations:
\begin{equation}\label{eq322}
\begin{split}
g_1=&6\,{x_{{1}}}^{2}x_{{2}}{x_{{5}}}^{2}{x_{{6}}}^{2}-6\,x_{{1}}{x_{{2}}}^
{2}{x_{{5}}}^{2}{x_{{6}}}^{2}+8\,{x_{{1}}}^{2}x_{{2}}{x_{{5}}}^{2}-8\,
x_{{1}}{x_{{2}}}^{2}{x_{{6}}}^{2}-x_{{1}}{x_{{5}}}^{2}{x_{{6}}}^{2}+x_
{{2}}{x_{{5}}}^{2}{x_{{6}}}^{2}
=0,\\
g_2=&6\,{x_{{2}}}^{2}x_{{3}}{x_{{5}}}^{2}{x_{{6}}}^{2}-2\,x_{{2}}{x_{{3}}}^
{2}{x_{{5}}}^{2}{x_{{6}}}^{2}+8\,{x_{{2}}}^{2}x_{{3}}{x_{{6}}}^{2}-4\,
x_{{2}}{x_{{3}}}^{2}{x_{{5}}}^{2}-4\,x_{{2}}{x_{{3}}}^{2}{x_{{6}}}^{2}
-5\,x_{{2}}{x_{{5}}}^{2}{x_{{6}}}^{2}\\&+x_{{3}}{x_{{5}}}^{2}{x_{{6}}}^{2
}=0,\\
g_3=&3\,x_{{1}}x_{{3}}{x_{{5}}}^{2}{x_{{6}}}^{2}+3\,x_{{2}}x_{{3}}{x_{{5}}}
^{2}{x_{{6}}}^{2}+30\,{x_{{3}}}^{2}{x_{{5}}}^{2}{x_{{6}}}^{2}+32\,x_{{
3}}{x_{{5}}}^{3}x_{{6}}-120\,x_{{3}}{x_{{5}}}^{2}{x_{{6}}}^{2}+32\,x_{
{3}}x_{{5}}{x_{{6}}}^{3}\\&+16\,{x_{{3}}}^{2}{x_{{5}}}^{2}+16\,{x_{{3}}}^
{2}{x_{{6}}}^{2}+20\,{x_{{5}}}^{2}{x_{{6}}}^{2}-32\,x_{{3}}x_{{5}}x_{{
6}}=0,\\
g_4=&-3\,x_{{1}}{x_{{5}}}^{2}x_{{6}}-3\,x_{{2}}{x_{{5}}}^{2}x_{{6}}-22\,x_{
{3}}{x_{{5}}}^{2}x_{{6}}-56\,{x_{{5}}}^{3}+120\,{x_{{5}}}^{2}x_{{6}}-8
\,x_{{5}}{x_{{6}}}^{2}+3\,x_{{2}}x_{{6}}\\&+33\,x_{{3}}x_{{6}}-120\,x_{{5
}}x_{{6}}+56\,x_{{5}}=0,\\
g_5=&16\,{x_{{5}}}^{3}x_{{6}}-16\,x_{{5}}{x_{{6}}}^{3}+x_{{1}}{x_{{5}}}^{2}
-x_{{2}}{x_{{6}}}^{2}+11\,x_{{3}}{x_{{5}}}^{2}-11\,x_{{3}}{x_{{6}}}^{2
}-40\,{x_{{5}}}^{2}x_{{6}}+40\,x_{{5}}{x_{{6}}}^{2}=0.
\end{split}
\end{equation}
We consider a polynomial ring $R=\mathbb{Q}[z, x_1, x_2, x_3, x_5,x_6]$ and an ideal $I$ generated by $\{g_1, g_2, g_3, g_4,g_5, \\z x_1 x_2 x_3 x_5x_6-1\}$ to find non-zero solutions of equations (\ref{eq322}). We take a lexicographic order $>$ with $z>x_1 >x_2 >x_3 >x_5>x_6$ for a monomial ordering on $R$. Then with the aid of computer, the following polynomial is contained in the Gr\"{o}bner basis for the ideal $I$ 
\begin{equation*}
(x_{{6}}-1 ) ( 7\,x_{{6}}-23 )  ( 864\,{
x_{{6}}}^{3}-1676\,{x_{{6}}}^{2}+973\,x_{{6}}-177)\cdot f(x_6)\cdot h(x_6),
\end{equation*}
where $f(x_6)$ is a polynomial of $x_6$ given by 
\begin{equation*}
\begin{split}
f(x_6)&=24313968\,{x_{{6}}}^{14}-271810080\,{x_{{6}}}^{13}+
1334881896\,{x_{{6}}}^{12}-4102312320\,{x_{{6}}}^{11}+9388266607\,{x_{
{6}}}^{10}\\&-17066486910\,{x_{{6}}}^{9}+25201149031\,{x_{{6}}}^{8}-
30982882320\,{x_{{6}}}^{7}+31894938304\,{x_{{6}}}^{6}-27360921600\,{x_
{{6}}}^{5}\\&+19523164352\,{x_{{6}}}^{4}-11276897280\,{x_{{6}}}^{3}+
5059512320\,{x_{{6}}}^{2}-1663672320\,x_{{6}}+301113344,
\end{split}
\end{equation*}
and $h(x_6)$ is a polynomial of $x_6$ of degree 114. For readers' convenience, we put it in Appendix II.

For $x_6=1$, from the Gr\"{o}bner basis of the ideal $I$ and with the aid of computer, we get the four solutions of the system $\{g_1=0,g_2=0,g_3=0,g_4=0,g_5=0\}$ which can be split into the following groups:
\[\mbox{Group 1.}\begin{aligned}
\{x_1=x_2=x_3=x_4=x_5=1\},
\end{aligned}\]
\[\mbox{Group 2.}\left\{\begin{aligned}
&\{x_2=x_3=x_5\yd0.4188876552,x_1\yd0.04273408738,x_6=1\},\\
&\{x_2=x_3=x_5\yd0.4617244620,x_1\yd1.543913333,x_6=1\},\\
&\{x_2=x_3=x_5\yd1.059202697,x_1\yd0.07225088447,x_6=1\}.
\end{aligned}\right.\]

By solving $(7\,x_{{6}}-23 )  ( 864\,{x_{{6}}}^{3}-1676\,{x_{{6}}}^{2}+973\,x_{{6}}-177=0$ numerically, there are three positive solutions which can be given approximately $x_6\yd0.4188876553, x_6\yd0.4617244621, x_6\yd1.059202697$ and the corresponding solutions of the system $\{g_1=0,g_2=0,g_3=0,g_4=0,g_5=0,7\,x_{{6}}-23 )  ( 864\,{x_{{6}}}^{3}-1676\,{x_{{6}}}^{2}+973\,x_{{6}}-177=0\}$ with $x_1x_2x_3x_5x_6\neq0$ are given by
\[\mbox{Group 3.}\left\{\begin{aligned}
&\{x_1=x_3=x_6\yd0.4188876552,x_2\yd0.04273408738,x_5=1\},\\
&\{x_1=x_3=x_6\yd0.4617244620,x_2\yd1.543913333,x_5=1\},\\
&\{x_1=x_3=x_6\yd1.059202697,x_2\yd0.07225088447,x_5=1\}.
\end{aligned}\right.\]

By solving $f(x_6)=0$ numerically, there are 6 different positive solutions, namely $0.7920673406$, $0.8040419514$, $1.075965351$, $1.681651936$, $2.596366999$, $3.419732659$, and the corresponding solutions of the system $\{g_1=0,g_2=0,g_3=0,g_4=0,g_5=0,f(x_6)=0\}$ with $x_1x_2x_3x_5x_6\neq0$ are given by
\[\mbox{Group 3.}\left\{\begin{aligned}
&\{x_1=x_2\yd1.211722573,x_3\yd0.2521819866,x_5=x_6\yd0.7920673405\},\\
&\{x_1=x_2\yd0.04116638566,x_3\yd0.2299652722,x_5=x_6\yd0.8040419514\},\\
&\{x_1=x_2\yd0.07360068971,x_3\yd1.138692978,x_5=x_6\yd1.075965351\},\\
&\{x_1=x_2\yd0.07189340238,x_3\yd0.3957889206,x_5=x_6\yd1.681651936\},\\
&\{x_1=x_2\yd1.241147181,x_3\yd0.6544562607,x_5=x_6\yd2.596366998\},\\
&\{x_1=x_2\yd0.1594378743,x_3\yd1.292216476,x_5=x_6\yd3.419732659\}.
\end{aligned}\right.\]

By solving $h(x_6)=0$ numerically, there are 10 different positive solutions, namely $0.4271280200$, $0.4742936355$, $0.7058209689$, $0.8630215200$, $1.008898001$, $1.010769751$, $2.058282527$, $2.099884282$, $3.40293\\1725$, $3.413270469$, and the corresponding solutions of the system $\{g_1=0,g_2=0,g_3=0,g_4=0,g_5=0,f(x_6)=0\}$ with $x_1x_2x_3x_5x_6\neq0$ can be split into the following 5 groups:
\[\mbox{Group 4.}\left\{\small\begin{aligned}
&\{x_1\yd0.0468426418,x_2\yd0.0433223582,x_3\yd0.4600814315,x_5\yd1.008898001,x_6\yd0.4271280200\},\\
&\{x_1\yd0.0433223582,x_2\yd0.0468426418,x_3\yd0.4600814315,x_5\yd0.4271280200,x_6\yd1.008898001\}.
\end{aligned}\right.\]
\[\mbox{Group 5.}\left\{\small\begin{aligned}
&\{x_1\yd0.05053229100,x_2\yd1.535627653,x_3\yd0.5100903370,x_5\yd1.01076975,x_6\yd0.4742936355\},\\
&\{x_1\yd1.535627653,x_2\yd0.05053229100,x_3\yd0.5100903370,x_5\yd0.4742936355,x_6\yd1.01076975\}.
\end{aligned}\right.\]
\[\mbox{Group 6.}\left\{\small\begin{aligned}
&\{x_1\yd1.075832320,x_2\yd0.04173283859,x_3\yd0.2381776636,x_5\yd0.8630215199,x_6\yd0.7058209688\},\\
&\{x_1\yd0.04173283859,x_2\yd1.075832320,x_3\yd0.2381776636,x_5\yd0.7058209688,x_6\yd0.8630215199\}.
\end{aligned}\right.\]
\[\mbox{Group 7.}\left\{\small\begin{aligned}
&\{x_1\yd0.0887989484,x_2\yd1.442031123,x_3\yd0.4977693932,x_5\yd2.099884281,x_6\yd2.058282526\},\\
&\{x_1\yd1.442031123,x_2\yd0.0887989484,x_3\yd0.4977693932,x_5\yd2.058282526,x_6\yd2.099884281\}.
\end{aligned}\right.\]
\[\mbox{Group 8.}\left\{\small\begin{aligned}
&\{x_1\yd0.1570295299,x_2\yd0.9524941307,x_3\yd1.170481952,x_5\yd3.413270468,x_6\yd3.402931724\},\\
&\{x_1\yd0.9524941307,x_2\yd0.1570295299,x_3\yd1.170481952,x_5\yd3.402931724,x_6\yd3.413270468\}.
\end{aligned}\right.\]

Among these solutions, we remark that the solution in Group 1 is Killing metric, the solutions in Group 2 and Group 3 induce the same metrics  up to isometry respectively which are naturally reductive due to Proposition \ref{nrp=3}, while the solutions in Group 3 induce 6 different left-invariant Einstein metrics which are non-naturally reductive and the solutions in each of Group 4-8 induce a same metric  up to isometry which is non-naturally reductive due to Proposition \ref{nrp=3}.

In conclusion, we find 11 different non-naturally reductive Einstein metrics on $\E_8$-I.

\textbf{Case of $p=4$.} By the same discussion in \ref{nrp=2}, we give the criterion to determine whether a left-invariant metric of the form (\ref{metric4}) is naturally reductive.
\begin{prop}\label{nrp=4}
If a left-invariant metric $<,>$ of the form (\ref{metric1}) on $G$ is naturally reductive with respect to $G\times L$ for some closed subgroup $L$ of $G$, then for the case of $p=4$, one of the following holds: 1)$x_2=x_3=x_4=x_5$, $x_6=x_7$  2)$x_1=x_3=x_4=x_6$, $x_5=x_7$  3)$x_1=x_2=x_4=x_7$, $x_5=x_6$  4)$x_5=x_6=x_7$. 

Conversely, if one of  1), 2), 3), 4) is satisfied, then the metric of the form (\ref{metric4}) for the case of $p=4$ is naturally reductive with respect to $G\times L$ for some closed subgroup $L$ of $G$.
\end{prop}

\textbf{Case of $\E_7$-I.} Due to Lemma \ref{p=4}, we have the following equations:
\begin{equation}\label{eqn4}
\begin{split}
r_1&=\frac{1}{12}\left(\frac{1}{3x_1}+\frac{4x_1}{3x_6^2}+\frac{4x_1}{3x_7^2}\right),\\
r_2&=\frac{1}{12}\left(\frac{1}{3x_2}+\frac{4x_2}{3x_5^2}+\frac{4x_2}{3x_7^2}\right),\\
r_3&=\frac{1}{12}\left(\frac{1}{3x_3}+\frac{4x_3}{3x_5^2}+\frac{4x_3}{3x_6^2}\right),\\
r_4&=\frac{1}{112}\left(\frac{28}{3x_4}+\frac{56x_4}{9x_5^2}+\frac{56x_4}{9x_6^2}+\frac{56x_4}{9x_7^2}\right),\\
r_5&=\frac{1}{2x_5}+\frac{1}{9}\left(\frac{x_5}{x_6x_7}-\frac{x_6}{x_7x_5}-\frac{x_7}{x_6x_5}\right)-\frac{1}{64}\left(\frac{4x_2}{3x_5^2}+\frac{4x_3}{3x_5^2}+\frac{56x_4}{9x_5^2}\right),\\
r_6&=\frac{1}{2x_6}+\frac{1}{9}\left(\frac{x_6}{x_7x_5}-\frac{x_7}{x_6x_5}-\frac{x_5}{x_7x_6}\right)-\frac{1}{64}\left(\frac{4x_1}{3x_6^2}+\frac{4x_3}{3x_6^2}+\frac{56x_4}{9x_6^2}\right),\\
r_7&=\frac{1}{2x_7}+\frac{1}{9}\left(\frac{x_7}{x_5x_6}-\frac{x_6}{x_7x_5}-\frac{x_5}{x_7x_6}\right)-\frac{1}{64}\left(\frac{4x_1}{3x_7^2}+\frac{4x_2}{3x_7^2}+\frac{56x_4}{9x_7^2}\right).
\end{split}
\end{equation}
We consider the system of equations given by 
\begin{equation}\label{eq4}
r_1-r_2=0, r_2-r_3=0, r_3-r_4=0, r_4-r_5=0, r_5-r_6=0, r_6-r_7=0.
\end{equation}
Since there are 7 variables in the system of equations, which is quite complicated for computer to calculate the Gr\"{o}bner bases, we normalize them by $x_6=x_7=1$, then it is easy to find that $x_2=x_3$ from the equations in (\ref{eqn4}), as a result, we have the following system of equations:
\begin{equation}\label{eq4}
\begin{split}
g_1=&8\,{x_{{1}}}^{2}x_{{3}}{x_{{5}}}^{2}-4\,x_{{1}}{x_{{3}}}^{2}{x_{{5}}}^{2}-4\,x_{{1}}{x_{{3}}}^{2}-x_{{1}}{x_{{5}}}^{2}+x_{{3}}{x_{{5}}}^{2}=0,\\
g_2=&4\,{x_{{3}}}^{2}x_{{4}}{x_{{5}}}^{2}-4\,x_{{3}}{x_{{4}}}^{2}{x_{{5}}}^{2}+4\,{x_{{3}}}^{2}x_{{4}}-2\,x_{{3}}{x_{{4}}}^{2}-3\,x_{{3}}{x_{{5}}}^{2}+x_{{4}}{x_{{5}}}^{2}=0,\\
g_3=&16\,{x_{{4}}}^{2}{x_{{5}}}^{2}-16\,x_{{4}}{x_{{5}}}^{3}+6\,x_{{3}}x_{{4}}+22\,{x_{{4}}}^{2}-40\,x_{{5}}x_{{4}}+12\,{x_{{5}}}^{2}=0,\\
g_4=&3\,x_{{1}}{x_{{5}}}^{2}+3\,x_{{3}}{x_{{5}}}^{2}+14\,x_{{4}}{x_{{5}}}^{2}+32\,{x_{{5}}}^{3}-72\,{x_{{5}}}^{2}-6\,x_{{3}}-14\,x_{{4}}+40\,x_{{5}}=0.
\end{split}
\end{equation}
We consider a polynomial ring $R=Q[z, x_1, x_3, x_4,x_5]$ and an ideal $I$ generated by $\{g_1, g_2, g_3, g_4, \\z x_1 x_3 x_4 x_5-1\}$ to find non-zero solutions of equations (\ref{eq4}). We take a lexicographic order $>$ with $z>x_1 >x_3 >x_4 >x_5$ for a monomial ordering on $R$. Then with the aid of computer, the following polynomial is contained in the Gr\"{o}bner basis for the ideal $I$ 
\begin{equation}\label{p4}
\left( x_{{5}}-1 \right)  \left( 4949\,{x_{{5}}}^{3}-9379\,{x_{{5}}}^{2}+5155\,x_{{5}}-875 \right) \cdot h(x_5),
\end{equation}
where $h(x_5)$ is a polynomial of degree $25$ given by
\begin{equation*}
\begin{split}
h(x_5)&=25101347481190400\,{x_{{5}}}^{25}-213612622522613760\,{x_{{5}}}^{24}+
1125174177049870336\,{x_{{5}}}^{23}\\&-4398212675755048960\,{x_{{5}}}^{22
}+13830794079039782912\,{x_{{5}}}^{21}-36611831495905378304\,{x_{{5}}}
^{20}\\&+83642128611649716224\,{x_{{5}}}^{19}-167796138043083587584\,{x_{
{5}}}^{18}+299027316357649125376\,{x_{{5}}}^{17}\\&-477090509137235365888
\,{x_{{5}}}^{16}+685232950401086713856\,{x_{{5}}}^{15}-
888981413909110722560\,{x_{{5}}}^{14}\\&+1043617887134219845504\,{x_{{5}}
}^{13}-1109092082064780894976\,{x_{{5}}}^{12}+1065873428902655206688\,
{x_{{5}}}^{11}\\&-924019385502926205728\,{x_{{5}}}^{10}+
719633332172147554621\,{x_{{5}}}^{9}-500367766771546562545\,{x_{{5}}}^
{8}\\&+307906916846444099368\,{x_{{5}}}^{7}-165616036629637255130\,{x_{{5
}}}^{6}+76466724878453429510\,{x_{{5}}}^{5}\\&-29517565522401301760\,{x_{
{5}}}^{4}+9135633690673393900\,{x_{{5}}}^{3}-2105338765392512650\,{x_{
{5}}}^{2}\\&+314822238961211625\,x_{{5}}-22360268064771875
\end{split}
\end{equation*}
In equation \ref{p4}, we can get four solutions, namely $1$, $0.3741245714$, $0.4352557643$, $1.085749994$.\\
For $x_5=1$, we have $x_5=x_6=x_7=1$, then by Proposition \ref{nrp=4}, we know the left-invariant Einstein metrics corresponding to these solutions are all naturally reductive.\\
For other three solutions, the corresponding solutions of the system of equations $\{g_1=0,g_2=0,g_3=0,g_4=0\}$ with $x_1x_3x_4x_5\neq0$ are as follows:
\begin{equation*}
\begin{split}
&\{x_1=0.06992197765, x_3=x_4=x_5=0.3741245714\},\\
&\{x_1=1.574157664, x_3=x_4=x_5=0.4352557643\},\\
&\{x_1=0.1259345024, x_3=x_4=x_5=1.085749994\}.
\end{split}
\end{equation*}
Due to Proposition \ref{nrp=4}, the left-invariant Einstein metrics induced by these solutions are all naturally reductive.

The solutions of $h(x_5)=0$ can be given approximately by $\{x_5=0.3952383758, x_5=0.4800791989, x_5=0.4889224428, x_5=0.6243909850, x_5=0.8764616162, x_5=0.9877146527, x_5=1.214528817\}$ and the corresponding solutions of the system of equations $\{g_1=0,g_2=0,g_3=0,g_4=0,h(x_5)=0\}$ with $x_1x_3x_4x_5\neq0$ are as follows:
\begin{equation*}
\begin{split}
&\{x_1=0.07185376030, x_3=0.08311707788, x_4=0.5173806186, x_5=0.3952383758\},\\
&\{x_1=1.505180802, x_3=0.09335085020, x_4=0.6214188681, x_5=0.4800791989\},\\
&\{x_1=0.07131646202, x_3=0.6268846419, x_4=0.2651770624, x_5=0.4889224428\},\\
&\{x_1=1.506498452, x_3=0.8045481479, x_4=0.3009635903, x_5=0.6243909850\},\\
&\{x_1=1.119004750, x_3=1.1136471437, x_4=1.063993479, x_5=0.8764616162\},\\
&\{x_1=0.08089954767, x_3=0.08095556860, x_4=0.2627271790, x_5=0.9877146527\},\\
&\{x_1=0.1003478332, x_3=1.505404155, x_4=0.3341288081, x_5=1.214528817\},\\
\end{split}
\end{equation*}
Due to Proposition \ref{nrp=4}, the left-invariant Einstein metrics induced by these solutions are all non-naturally reductive. 

In summarize, we find 7 different non-naturally reductive left-invariant Einstein metrics on $\E_7$-I.
\textbf{Case of $\E_7$-II.} By Lemma \ref{coeff7}, the components of Ricci tensor with respect to the metric (\ref{metric7}) are as follows:
\[\left\{\begin{aligned}
r_0&=\frac{1}{4}\left(\frac{4u_0}{9x_3^2}+\frac{5u_0}{9x_4^2}\right),\\
r_1&=\frac{1}{12}\left(\frac{1}{3x_1}+\frac{x_1}{x_3^2}+\frac{5x_1}{3x_5^2}\right),\\
r_2&=\frac{1}{140}\left(\frac{35}{3x_2}+\frac{35x_2}{9x_3^2}+\frac{70x_2}{9x_4^2}+\frac{35x_2}{3x_5^2}\right),\\
r_3&=\frac{1}{2x_3}+\frac{5}{36}\left(\frac{x_3}{x_4x_5}-\frac{x_4}{x_5x_3}-\frac{x_5}{x_3x_4}\right)-\frac{1}{48}\left(\frac{4u_0}{9x_3^2}+\frac{x_1}{x_3^2}+\frac{35x_2}{9x_3^2}\right),\\
r_4&=\frac{1}{2x_4}+\frac{1}{9}\left(\frac{x_4}{x_5x_3}-\frac{x_5}{x_3x_4}-\frac{x_3}{x_4x_5}\right)-\frac{1}{60}\left(\frac{5u_0}{9x_4^2}+\frac{70x_2}{9x_4^2}\right),\\
r_5&=\frac{1}{2x_5}+\frac{1}{12}\left(\frac{x_5}{x_3x_4}-\frac{x_3}{x_4x_5}-\frac{x_4}{x_5x_3}\right)-\frac{1}{80}\left(\frac{5x_1}{3x_5^2}+\frac{35x_2}{3x_5^2}\right).
\end{aligned}\right.\]

Then we will give a criterion to decide whether a metric of the form (\ref{metric7}) is naturally reductive.
\begin{prop}\label{nr7}
If a left-invariant metric $<\ ,\ >$ of the form (\ref{metric7}) on $G=\E_7$ is naturally reductive with respect to $G\times L$ for some closed subgroup $L$ of $G$, then one of the following holds:
\begin{equation*}
1)u_0=x_1=x_2=x_3, x_4=x_5\quad 2)u_0=x_2=x_4, x_3=x_5\quad 3)x_1=x_2=x_5, x_3=x_4\quad 4)x_3=x_4=x_5.
\end{equation*}
Conversely, if one of the conditions 1), 2), 3), 4) holds, then the metric $<\ ,\ >$ of the form (\ref{metric7}) is naturally reductive with respect to $G\times L$ for some closed subgroup $L$ of $G$.
\end{prop}
\begin{proof}
Let $\frak{l}$ be the Lie algebra of $L$. Then we have either $\frak{l}\subset\fk$ or $\frak{l}\not\subset\fk$. For the case of $\frak{l}\not\subset\fk$. Let $\frak{h}$ be the subalgebra of $\g$ generated by $\frak{l}$ and $\fk$. Since $\g=\T\oplus A_1\oplus A_5\oplus\p_1\oplus\p_2\oplus\p_3$ and the structure of generalized Wallach spaces, there must only one $\p_i$ contained in $\frak{h}$. If $\p_1\subset\frak{h}$, then $\frak{h}=\fk\oplus\p_1\cong A_7$, which is in fact the set of fixed points of the involutive automorphism $\sigma$. Due to Theorem \ref{nr}, we have $u_0=x_1=x_2=x_3, x_4=x_5$. If $\p_2\subset\frak{h}$, then $\frak{h}=\fk\oplus\p_2\cong A_1\oplus D_6$, which is in fact the set of the fixed points of the involutive automorphism $\tau$, as a result of Theorem \ref{nr}, we have $u_0=x_2=x_4, x_3=x_5$. If $\p_3\subset\frak{h}$, then $\frak{h}=\fk\oplus\p_3\cong\T\oplus\E_6$, which is corresponds to the involutive $\sigma\tau=\tau\sigma$ \cite{ChKaLi}, due to Theorem \ref{nr}, we have $x_1=x_2=x_5,x_3=x_4$.

We proceed with the case $\frak{l}\subset\frak{k}$. Because the orthogonal complement $\frak{l}^{\perp}$ of $\frak{l}$ with respect to $B$ contains the orthogonal complement $\fk^{\perp}$ of $\fk$, it follows that $\p_1\oplus\p_2\oplus\p_3\subset\frak{l}^{\perp}$. Since the invariant metric $<\ ,\ >$ is naturally reductive with respect to $G\times L$, we conclude that $x_3 = x_4 = x_5$ by Theorem \ref{nr}.

The converse is a direct conclusion of Theorem \ref{nr}.
\end{proof}

Recall that the homogeneous Einstein equation for the left-invariant metric $<\ ,\ >$ is given by
$$\{r_0-r_1=0, r_1-r_2=0, r_2-r_3=0, r_3-r_4=0,r_4-r_5=0\}.$$
We normalize the metric by setting $u_0=1$, then the homogeneous Einstein equation is equivalent to the following system of equations:
\[\left\{\begin{aligned}
g_0 = &-5\,{x_{{1}}}^{2}{x_{{3}}}^{2}{x_{{4}}}^{2}-3\,{x_{{1}}}^{2}{x_{{4}}}^
{2}{x_{{5}}}^{2}-{x_{{3}}}^{2}{x_{{4}}}^{2}{x_{{5}}}^{2}+5\,x_{{1}}{x_
{{3}}}^{2}{x_{{5}}}^{2}+4\,x_{{1}}{x_{{4}}}^{2}{x_{{5}}}^{2}=0,\\
g_1 = &5\,{x_{{1}}}^{2}x_{{2}}{x_{{3}}}^{2}{x_{{4}}}^{2}+3\,{x_{{1}}}^{2}x_{{
2}}{x_{{4}}}^{2}{x_{{5}}}^{2}-3\,x_{{1}}{x_{{2}}}^{2}{x_{{3}}}^{2}{x_{
{4}}}^{2}-2\,x_{{1}}{x_{{2}}}^{2}{x_{{3}}}^{2}{x_{{5}}}^{2}-x_{{1}}{x_
{{2}}}^{2}{x_{{4}}}^{2}{x_{{5}}}^{2}-3\,x_{{1}}{x_{{3}}}^{2}{x_{{4}}}^
{2}{x_{{5}}}^{2}\\&+x_{{2}}{x_{{3}}}^{2}{x_{{4}}}^{2}{x_{{5}}}^{2}=0,\\
g_2=&9\,x_{{1}}x_{{2}}{x_{{4}}}^{2}{x_{{5}}}^{2}+36\,{x_{{2}}}^{2}{x_{{3}}}
^{2}{x_{{4}}}^{2}+24\,{x_{{2}}}^{2}{x_{{3}}}^{2}{x_{{5}}}^{2}+47\,{x_{
{2}}}^{2}{x_{{4}}}^{2}{x_{{5}}}^{2}-60\,x_{{2}}{x_{{3}}}^{3}x_{{4}}x_{
{5}}+60\,x_{{2}}x_{{3}}{x_{{4}}}^{3}x_{{5}}\\&-216\,x_{{2}}x_{{3}}{x_{{4}
}}^{2}{x_{{5}}}^{2}+60\,x_{{2}}x_{{3}}x_{{4}}{x_{{5}}}^{3}+36\,{x_{{3}
}}^{2}{x_{{4}}}^{2}{x_{{5}}}^{2}+4\,x_{{2}}{x_{{4}}}^{2}{x_{{5}}}^{2}=0,\\
g_3=&-9\,x_{{1}}{x_{{4}}}^{2}x_{{5}}+56\,x_{{2}}{x_{{3}}}^{2}x_{{5}}-35\,x_
{{2}}{x_{{4}}}^{2}x_{{5}}+108\,{x_{{3}}}^{3}x_{{4}}-216\,{x_{{3}}}^{2}
x_{{4}}x_{{5}}-108\,x_{{3}}{x_{{4}}}^{3}+216\,x_{{3}}{x_{{4}}}^{2}x_{{
5}}\\&-12\,x_{{3}}x_{{4}}{x_{{5}}}^{2}+4\,{x_{{3}}}^{2}x_{{5}}-4\,{x_{{4}
}}^{2}x_{{5}}=0,\\
g_4=&9\,x_{{1}}x_{{3}}{x_{{4}}}^{2}+63\,x_{{2}}x_{{3}}{x_{{4}}}^{2}-56\,x_{
{2}}x_{{3}}{x_{{5}}}^{2}-12\,{x_{{3}}}^{2}x_{{4}}x_{{5}}-216\,x_{{3}}{
x_{{4}}}^{2}x_{{5}}+216\,x_{{3}}x_{{4}}{x_{{5}}}^{2}+84\,{x_{{4}}}^{3}
x_{{5}}\\&-84\,x_{{4}}{x_{{5}}}^{3}-4\,x_{{3}}{x_{{5}}}^{2}=0.
\end{aligned}\right.\]
Consider the polynomial ring $R=\mathbb{Q}[z, x_1, x_2, x_3, x_4, x_5]$ and the ideal $I$, generated by polynomials $\{ z x_1 x_2 x_3 x_4 x_5-1, g_0, g_1, g_2, g_3, g_4\}$. We take a lexicographic ordering $>$, with $z > x_1 > x_2 > x_3 > x_4 > x_5$ for a monomial ordering on $R$. Then, by the aid of computer, we see that a Gr\"{o}bner basis for the ideal $I$ contains a polynomial of $x_5$ given by
\begin{equation*}
 ( x_{{5}}-1)( 1067x_{{5}}-392)( 2x_{{5}}-7)( 875{x_{{5}}}^{3}-5155{x_{{5}}}^{2}+9379x_{{5}}\\-4949)h(x_5), 
\end{equation*}
where $h(x_5)$ is a polynomial of degree 78. Since the length of this polynomial may affect the readers to read, we put it in the Appendix III.

We remark that $x_1, x_2, x_3, x_4$ can be written into a polynomial of $x_5$ with coefficient of rational numbers. By solving $h(x_5)=0$ numerically, we get 6 solutions, namely $x_5\yd0.3954420465,  x_5\yd0.7869165511, x_5\yd1.022441180, x_5\yd1.525178916, x_5\yd2.907605999, x_5\yd3.996569735$. Further more, the corresponding solutions of the system of equations $\{g_0=0, g_1=0, g_2=0, g_3=0, g_4=0, h(x_5)=0\}$ with $x_1x_2x_3x_4x_5\neq0$ are as follows:
\begin{equation*}
\begin{split}
&\{x_1\yd0.6527831128,x_2\yd0.4342037927,x_3\yd0.7023547363,x_4\yd0.7181567785,x_5\yd0.3954420465\},\\
&\{x_1\yd1.238139339,x_2\yd0.2406838191,x_3\yd0.9516792542,x_4\yd0.6904065038,x_5\yd0.7869165511\},\\
&\{x_1\yd0.07716292844,x_2\yd0.2617256871,x_3\yd1.140229546,x_4\yd0.6923748274,x_5\yd1.022441180\},\\
&\{x_1\yd0.1125592068,x_2\yd0.3602427458,x_3\yd0.7175220814,x_4\yd1.576177679,x_5\yd1.525178916\},\\
&\{x_1\yd1.055549830,x_2\yd0.7580899024,x_3\yd0.9219014311,x_4\yd2.908338968,x_5\yd2.907605999\},\\
&\{x_1\yd0.3623080653,x_2\yd1.354302959,x_3\yd1.066348568,x_4\yd4.005454245,x_5\yd3.996569735\}.
\end{split}
\end{equation*}
Due to Proposition \ref{nr7}, we conclude that these six solutions induce six different non-naturally reductive left-invariant Einstein metrics on $\E_7$.

For $x_5=1, x_5=\frac{392}{1067}$ and $x_5=\frac{7}{2}$, the corresponding solution of the system of equations $\{g_0=0, g_1=0, g_2=0, g_3=0, g_4=0\}$ with $x_1x_2x_3x_4x_5\neq0$ are as follows:
\begin{equation*}
\begin{split}
&\{x_1=x_2=x_3=x_4=x_5=1\},\\
&\{x_1=x_2=x_5=\frac{392}{1067},x_3=x_4=\frac{742}{1067}\},\\
&\{x_1=x_2=x_3=1,x_4=x_5=\frac{7}{2}\}.
\end{split}
\end{equation*}
Due to Proposition \ref{nr7}, the left invariant Einstein metrics induced by these three solutions are all naturally reductive.

For $(875{x_{{5}}}^{3}-5155{x_{{5}}}^{2}+9379x_{{5}}-4949)=0$, the solutions of the system of equations $\{g_0=0, g_1=0, g_2=0, g_3=0, g_4=0,(875{x_{{5}}}^{3}-5155{x_{{5}}}^{2}+9379x_{{5}}-4949)=0\}$ with $x_1x_2x_3x_4x_5\neq0$ are as follows:
\begin{equation*}
\begin{split}
&\{x_1\yd0.1159884900,x_2=x_4=1,x_3=x_5\yd0.9210223401\},\\
&\{x_1\yd3.616626805,x_2=x_4=1,x_3=x_5\yd2.297499727\},\\
&\{x_1\yd0.1868949087,x_2=x_4=1,x_3=x_5\yd2.672906503\}.
\end{split}
\end{equation*}
Due to Proposition \ref{nr7}, the left invariant Einstein metrics induced by these three solutions are all naturally reductive.

In conclusion, we find six different left-invariant Einstein metrics on $\E_7$ which are non-naturally reductive.

\textbf{Case of $\E_6$-II.} By Lemma \ref{coeff6}, the components of Ricci tensor with respect to the metric (\ref{metric6}) can be expressed as follows:
\[\left\{\begin{aligned}
r_0&=\frac{1}{4}\left(\frac{u_0}{2x_4^2}+\frac{u_0}{2x_5^2}\right),\\
r_1&=\frac{1}{12}\left(\frac{1}{2x_1}+\frac{x_1}{x_5^2}+\frac{3x_1}{2x_6^2}\right),\\
r_2&=\frac{1}{12}\left(\frac{1}{2x_2}+\frac{x_2}{x_4^2}+\frac{3x_2}{2x_6^2}\right),\\
r_3&=\frac{1}{60}\left(\frac{5}{x_3}+\frac{5x_3}{2x_4^2}+\frac{5x_3}{2x_5^2}+\frac{5x_3}{x_6^2}\right),\\
r_4&=\frac{1}{2x_4}+\frac{1}{8}\left(\frac{x_4}{x_5x_6}-\frac{x_5}{x_6x_4}-\frac{x_6}{x_4x_5}\right)-\frac{1}{32}\left(\frac{u_0}{2x_4^2}+\frac{x_2}{x_4^2}+\frac{5x_3}{2x_4^2}\right),\\
r_5&=\frac{1}{2x_5}+\frac{1}{8}\left(\frac{x_5}{x_6x_4}-\frac{x_6}{x_4x_5}-\frac{x_4}{x_5x_6}\right)-\frac{1}{32}\left(\frac{u_0}{2x_5^2}+\frac{x_1}{x_5^2}+\frac{5x_3}{2x_5^2}\right),\\
r_6&=\frac{1}{2x_6}+\frac{1}{12}\left(\frac{x_6}{x_4x_5}-\frac{x_4}{x_5x_6}-\frac{x_5}{x_6x_4}\right)-\frac{1}{48}\left(\frac{3x_1}{2x_6^2}+\frac{3x_2}{2x_6^2}+\frac{5x_3}{x_6^2}\right).
\end{aligned}\right.\]

The we will give a criterion to decide whether a left-invariant metric of the form (\ref{metric6}) on $\E_6$ is naturally reductive.
\begin{prop}\label{nr6}
If a left-invariant metric $<\ ,\ >$ of the form (\ref{metric6}) on $G=\E_6$ is naturally reductive with respect to $G\times L$ for some closed subgroup $L$ of $G$, then one of the following holds:
\begin{equation*}
1)u_0=x_2=x_3=x_4, x_5=x_6\ 2)u_0=x_1=x_3=x_5, x_4=x_6\ 3)x_1=x_2=x_3=x_6, x_4=x_5\ 4)x_4=x_5=x_6.
\end{equation*}
Conversely, if one of the conditions 1), 2), 3), 4) holds, then the metric $<\ ,\ >$ of the form (\ref{metric6}) is naturally reductive with respect to $G\times L$ for some closed subgroup $L$ of $G$.
\end{prop}
\begin{proof}
Let $\frak{l}$ be the Lie algebra of $L$. Then we have either $\frak{l}\subset\fk$ or $\frak{l}\not\subset\fk$. For the case of $\frak{l}\not\subset\fk$. Let $\frak{h}$ be the subalgebra of $\g$ generated by $\frak{l}$ and $\fk$. Since $\g=\T\oplus A_1^1\oplus A_1^2\oplus A_3\oplus\p_1\oplus\p_2\oplus\p_3$ and the structure of generalized Wallach spaces, there must only one $\p_i$ contained in $\frak{h}$. If $\p_1\subset\frak{h}$, then $\frak{h}=\fk\oplus\p_1\cong A_1^1\oplus A_5$, which is in fact the set of fixed points of the involutive automorphism $\sigma$. Due to Theorem \ref{nr}, we have $u_0=x_2=x_3=x_4,x_5=x_6$. If $\p_2\subset\frak{h}$, then $\frak{h}=\fk\oplus\p_2\cong A_1^2\oplus A_5$, which is in fact the set of the fixed points of the involutive automorphism $\tau$, as a result of Theorem \ref{nr}, we have $u_0=x_1=x_3=x_5, x_4=x_6$. If $\p_3\subset\frak{h}$, then $\frak{h}=\fk\oplus\p_3\cong\T\oplus D_5$, which is corresponds to the involutive $\sigma\tau=\tau\sigma$ \cite{ChKaLi}, due to Theorem \ref{nr}, we have $x_1=x_2=x_3=x_6,x_4=x_5$.

We proceed with the case $\frak{l}\subset\frak{k}$. Because the orthogonal complement $\frak{l}^{\perp}$ of $\frak{l}$ with respect to $B$ contains the orthogonal complement $\fk^{\perp}$ of $\fk$, it follows that $\p_1\oplus\p_2\oplus\p_3\subset\frak{l}^{\perp}$. Since the invariant metric $<\ ,\ >$ is naturally reductive with respect to $G\times L$, we conclude that $x_4 = x_5 = x_6$ by Theorem \ref{nr}.

The converse is a direct conclusion of Theorem \ref{nr}.
\end{proof}

Recall that the homogeneous Einstein equation for the left-invariant metric $<\ ,\ >$ is given by
$$\{r_0-r_1=0, r_1-r_2=0, r_2-r_3=0, r_3-r_4=0,r_4-r_5=0\}.$$
Then finding Einstein metrics of the form (\ref{metric6}) reduces to finding the positive solutions of the above system and we normalize the metric by setting $x_6=1$, then the homogeneous Einstein equation is equivalent to the following system of equations:
\[\left\{\begin{aligned}
g_0 = &-3\,{x_{{1}}}^{2}{x_{{4}}}^{2}{x_{{5}}}^{2}+3\,u_{{0}}x_{{1}}{x_{{4}}}
^{2}+3\,u_{{0}}x_{{1}}{x_{{5}}}^{2}-2\,{x_{{1}}}^{2}{x_{{4}}}^{2}-{x_{
{4}}}^{2}{x_{{5}}}^{2}=0,\\
g_1 = &3\,{x_{{1}}}^{2}x_{{2}}{x_{{4}}}^{2}{x_{{5}}}^{2}-3\,x_{{1}}{x_{{2}}}^
{2}{x_{{4}}}^{2}{x_{{5}}}^{2}+2\,{x_{{1}}}^{2}x_{{2}}{x_{{4}}}^{2}-2\,
x_{{1}}{x_{{2}}}^{2}{x_{{5}}}^{2}-x_{{1}}{x_{{4}}}^{2}{x_{{5}}}^{2}+x_
{{2}}{x_{{4}}}^{2}{x_{{5}}}^{2}=0,\\
g_2=&3\,{x_{{2}}}^{2}x_{{3}}{x_{{4}}}^{2}{x_{{5}}}^{2}-2\,x_{{2}}{x_{{3}}}^
{2}{x_{{4}}}^{2}{x_{{5}}}^{2}+2\,{x_{{2}}}^{2}x_{{3}}{x_{{5}}}^{2}-x_{
{2}}{x_{{3}}}^{2}{x_{{4}}}^{2}-x_{{2}}{x_{{3}}}^{2}{x_{{5}}}^{2}-2\,x_
{{2}}{x_{{4}}}^{2}{x_{{5}}}^{2}+x_{{3}}{x_{{4}}}^{2}{x_{{5}}}^{2}=0,\\
g_3=&16\,{x_{{3}}}^{2}{x_{{4}}}^{2}{x_{{5}}}^{2}-24\,x_{{3}}{x_{{4}}}^{3}x_
{{5}}+24\,x_{{3}}x_{{4}}{x_{{5}}}^{3}+3\,u_{{0}}x_{{3}}{x_{{5}}}^{2}+6
\,x_{{2}}x_{{3}}{x_{{5}}}^{2}+8\,{x_{{3}}}^{2}{x_{{4}}}^{2}+23\,{x_{{3
}}}^{2}{x_{{5}}}^{2}\\&-96\,x_{{3}}x_{{4}}{x_{{5}}}^{2}+16\,{x_{{4}}}^{2}
{x_{{5}}}^{2}+24\,x_{{3}}x_{{4}}x_{{5}}=0,\\
g_4=&16\,{x_{{4}}}^{3}x_{{5}}-16\,x_{{4}}{x_{{5}}}^{3}+u_{{0}}{x_{{4}}}^{2}
-u_{{0}}{x_{{5}}}^{2}+2\,x_{{1}}{x_{{4}}}^{2}-2\,x_{{2}}{x_{{5}}}^{2}+
5\,x_{{3}}{x_{{4}}}^{2}-5\,x_{{3}}{x_{{5}}}^{2}-32\,{x_{{4}}}^{2}x_{{5
}}\\&+32\,x_{{4}}{x_{{5}}}^{2}=0,\\
g_5=&6\,x_{{1}}x_{{4}}{x_{{5}}}^{2}+6\,x_{{2}}x_{{4}}{x_{{5}}}^{2}+20\,x_{{
3}}x_{{4}}{x_{{5}}}^{2}-8\,{x_{{4}}}^{2}x_{{5}}-96\,x_{{4}}{x_{{5}}}^{
2}+40\,{x_{{5}}}^{3}-3\,u_{{0}}x_{{4}}-6\,x_{{1}}x_{{4}}-15\,x_{{3}}x_
{{4}}\\&+96\,x_{{4}}x_{{5}}-40\,x_{{5}}=0.
\end{aligned}\right.\]
Consider the polynomial ring $R=\mathbb{Q}[z, u_0, x_1, x_2, x_3, x_4, x_5]$ and the ideal $I$, generated by polynomials $\{ z u_0 x_1 x_2 x_3 x_4 x_5 -1, g_0, g_1, g_2, g_3, g_4, g_5\}$. We take a lexicographic ordering $>$, with $z > u_0 > x_1 > x_2 > x_3 > x_4 > x_5$ for a monomial ordering on $R$. Then, by the aid of computer, we see that a Gr\"{o}bner basis for the ideal $I$ contains a polynomial of $x_5$ given by $( x_{{5}}-1 )  ( 17\,x_{{5}}-31 )  ( 319\,{x_{{5}}}^{3}-585\,{x_{{5}}}^{2}+298\,x_{{5}}-46)\cdot
h(x_5)$, where $h(x_5)$ is a polynomial of degree 178, we put it in the Appendix IV, since its length may affect our readers to read. In fact, we remark that with the polynomials in the Gr\"{o}bner basis, $x_i$ can be written into a expression of $x_{i+1}, \cdots, x_5$ and $i=1,\cdots,5$, while $u_0$ can be written into a expression of $x_1,x_2,x_3,x_4,x_5$.

By solving $h(x_5)=0$ numerically, there exists 14 positive solutions which can be given approximately by $x_5\yd0.3190072071$, $x_5\yd0.3565775930$, $x_5\yd0.4054489785$, $x_5\yd0.4709163886$, $x_5\yd0.5455899299$, $x_5\yd0.7832400305$, $x_5\yd1.000773211$, $x_5\yd1.002658584$, $x_5\yd1.003465783$, $x_5\yd1.006528315$, $x_5\yd1.069488872$, $x_5\yd1.155548556$, $x_5\yd1.646506483$, $x_5\yd1.695781258$. Moreover, the corresponding solutions of the system of equations $\{g_0=0,g_1=0,g_2=0,g_3=0,g_4=0,g_5=0,h(x_5)=0\}$ with $u_0x_1x_2x_3x_4x_5\neq0$ can be split into the following 7 groups:
\[\mbox{1.}\left\{\footnotesize\begin{aligned}
&\{u_0\yd0.3120392058,x_1\yd0.1471819373,x_2\yd0.1040632043,x_3\yd0.4015791280,x_4\yd1.003465783,x_5\yd0.3190072071\},\\
&\{u_0\yd0.3120392058,x_1\yd0.1040632043,x_2\yd0.1471819373,x_3\yd0.4015791280,x_4\yd0.3190072071,x_5\yd1.003465783\},\\
\end{aligned}\right.\]
\[\mbox{2.}\left\{\footnotesize\begin{aligned}
&\{u_0\yd0.3832451893,x_1\yd0.4156187592,x_2\yd0.1033703333,x_3\yd0.2795983424,x_4\yd1.000773211,x_5\yd0.3565775930\},\\
&\{u_0\yd0.3832451893,x_1\yd0.1033703333,x_2\yd0.4156187592,x_3\yd0.2795983424,x_4\yd0.3565775930,x_5\yd1.000773211\},\\
\end{aligned}\right.\]
\[\mbox{3.}\left\{\footnotesize\begin{aligned}
&\{u_0\yd0.4028095641,x_1\yd0.1658986529,x_2\yd1.591316257,x_3\yd0.5073650261,x_4\yd1.006528315,x_5\yd0.4054489785\},\\
&\{u_0\yd0.4028095641,x_1\yd1.591316257,x_2\yd0.1658986529,x_3\yd0.5073650261,x_4\yd0.4054489785,x_5\yd1.006528315\},\\
\end{aligned}\right.\]
\[\mbox{4.}\left\{\footnotesize\begin{aligned}
&\{u_0\yd0.5209148934,x_1\yd0.5695949321,x_2\yd1.598554538,x_3\yd0.3242377273,x_4\yd1.002658584,x_5\yd0.4709163886\},\\
&\{u_0\yd0.5209148934,x_1\yd1.598554538,x_2\yd0.5695949321,x_3\yd0.3242377273,x_4\yd0.4709163886,x_5\yd1.002658584\},\\
\end{aligned}\right.\]
\[\mbox{5.}\left\{\footnotesize\begin{aligned}
&\{u_0\yd0.7642362204,x_1\yd0.1166446353,x_2\yd0.1088139526,x_3\yd0.2444045194,x_4\yd1.069488872,x_5\yd0.5455899299\},\\
&\{u_0\yd0.7642362204,x_1\yd0.1088139526,x_2\yd0.1166446353,x_3\yd0.2444045194,x_4\yd0.5455889929,x_5\yd1.069488872\},\\
\end{aligned}\right.\]
\[\mbox{6.}\left\{\footnotesize\begin{aligned}
&\{u_0\yd1.095971585,x_1\yd0.1445747761,x_2\yd1.600102216,x_3\yd0.3092235071,x_4\yd1.155548556,x_5\yd0.7832400305\},\\
&\{u_0\yd1.095971585,x_1\yd1.600102216,x_2\yd0.1445747761,x_3\yd0.3092235071,x_4\yd0.7832400305,x_5\yd1.155548556\},\\
\end{aligned}\right.\]
\[\mbox{7.}\left\{\footnotesize\begin{aligned}
&\{u_0\yd2.260526144,x_1\yd0.2562898903,x_2\yd1.059701210,x_3\yd1.147117560,x_4\yd1.695781258,x_5\yd1.646506483\},\\
&\{u_0\yd2.260526144,x_1\yd1.059701210,x_2\yd0.2562898903,x_3\yd1.147117560,x_4\yd1.646506483,x_5\yd1.695781258\}.
\end{aligned}\right.\]

By solving $319\,{x_{{5}}}^{3}-585\,{x_{{5}}}^{2}+298\,x_{{5}}-46=0$ numerically, there exists three different positive solutions which can be given approximately by $x_5\yd0.3244770611$, $x_5\yd0.4009373579$ and $x_5\yd1.108447830$. Further, the corresponding solutions of the system of equations $\{g_0=0,g_1=0,g_2=0,g_3=0,g_4=0,g_5=0,319\,{x_{{5}}}^{3}-585\,{x_{{5}}}^{2}+298\,x_{{5}}-46=0\}$ with $u_0x_1x_2x_3x_4x_5\neq0$ are:
\[\mbox{8.}\left\{\begin{aligned}
&\{u_0=x_1=x_3=x_5\yd0.3244770611,x_2\yd0.1030504001,x_4=1\},\\
&\{u_0=x_1=x_3=x_5\yd0.4009373579,x_2\yd1.613068224,x_4=1\},\\
&\{u_0=x_1=x_3=x_5\yd1.108447830,x_2\yd0.1984241041,x_4=1\}.
\end{aligned}\right.\]

For $x_5=1$, there exists four solutions of the system $\{g_1=0,g_2=0,g_3=0,g_4=0,g_5=0\}$ with $u_0x_1x_2x_3x_4x_5\neq0$ which can be given approximately by
\[\mbox{9.}\begin{aligned}
\{u_0=x_1=x_2=x_3=x_4=x_5=1\},
\end{aligned}\]
\[\mbox{10.}\left\{\begin{aligned}
&\{u_0=x_2=x_3=x_4\yd0.3244770611,x_1\yd0.1030504001,x_5=1\},\\
&\{u_0=x_2=x_3=x_4\yd0.4009373579,x_1\yd1.613068224,x_5=1\},\\
&\{u_0=x_2=x_3=x_4\yd1.108447830,x_1\yd0.1984241041,x_5=1\}.
\end{aligned}\right.\]

For $x_5=\frac{31}{17}$, the corresponding solution of the system of equations $\{g_1=0,g_2=0,g_3=0,g_4=0,g_5=0\}$ with $u_0x_1x_2x_3x_4x_5\neq0$ is 
\[\mbox{11.}\begin{aligned}
\{u_0=\frac{737}{289},x_1=x_2=x_3=1,x_4=x_5=\frac{31}{17}\}.
\end{aligned}\]

Among these solutions, we remark that the left-invariant Einstein metrics induced by the solutions in Group 8-11 are all naturally reductive, while the left-invariant Einstein metrics induced by the solutions in Group 1-7 are all non-naturally reductive according to Proposition \ref{nr6}. In particular, the metrics induced by the solutions in Group 8 and Group 10 are the same up to isometry respectively and the solutions in each of Group 1-7 induce a same metric up to isometry.

In conclusion, we find 7 different non-naturally left-invariant Einstein metrics on $\E_6$-II.

\section{Appendix I}
We put $h(x_6)$ in Section 4 here\\
\tiny$h(x_6)=288548666216102485319490298372480722534400000000000000000000000000000000000000000000000
\,{x_{{6}}}^{114}-\\
7335395742115974539503349380923216834120908800000000000000000000000000000000000000000000
\,{x_{{6}}}^{113}+\\
94078381310558994585051331187125707811005556326400000000000000000000000000000000000000000
\,{x_{{6}}}^{112}-\\
815372756534232158773437000571967604719385984368640000000000000000000000000000000000000000
\,{x_{{6}}}^{111}+\\
5394020114040582693624836326072703055152460459723980800000000000000000000000000000000000000
\,{x_{{6}}}^{110}-\\
29147068861160512043456310394532638973009479245456474112000000000000000000000000000000000000
\,{x_{{6}}}^{109}+\\
134330326689565122860196509693349702085611553735614524293120000000000000000000000000000000000
\,{x_{{6}}}^{108}-\\
543950722347413709091314078263315359554321702464653838631567360000000000000000000000000000000
\,{x_{{6}}}^{107}+\\
1977156456198764081665489415283217701559164441214210196526596096000000000000000000000000000000
\,{x_{{6}}}^{106}-\\
6553932091117232855452398232625779163507500833393670938453396684800000000000000000000000000000
\,{x_{{6}}}^{105}+\\
20051771138222265162787867070150309155860454357487955338163262586880000000000000000000000000000
\,{x_{{6}}}^{104}-\\
57149190441408854243012248124815005821521968707890639680575132139520000000000000000000000000000
\,{x_{{6}}}^{103}+\\
152835776655049481507606713389342085792929023876647000321045284939366400000000000000000000000000
\,{x_{{6}}}^{102}-\\
385756462944628713137839357523266365536724194341244461787260182897623040000000000000000000000000
\,{x_{{6}}}^{101}+\\
923246441937253023544548551112293271009138083357097396744185678294679552000000000000000000000000
\,{x_{{6}}}^{100}-\\
2103412757751515961253892845131681054271482548768957772344571292651880448000000000000000000000000
\,{x_{{6}}}^{99}+\\
4576635509329089592980580819532376494698679620316682908084360030156887162880000000000000000000000
\,{x_{{6}}}^{98}-\\
9536363644246112135881661787938304131231765937926496419549356264244148961280000000000000000000000
\,{x_{{6}}}^{97}+\\
19075075896720745327516498667217050539300287648144531385033627602887114752000000000000000000000000
\,{x_{{6}}}^{96}-\\
36702573270534659785989982819949577303959575255730333060378328961655387287715840000000000000000000
\,{x_{{6}}}^{95}+\\
68056216913739631749322574839527960190346820317441692917509489044571240158724096000000000000000000
\,{x_{{6}}}^{94}-\\
121812029770982949418572416294490232060904869125587642005056843136001221680798105600000000000000000
\,{x_{{6}}}^{93}+\\
210768703644041099247097074305412845727257394478296076350620300770567300126175395840000000000000000
\,{x_{{6}}}^{92}-\\
353023485401160038105010334971232204765337446216424784610083373458822887516598697984000000000000000
\,{x_{{6}}}^{91}+\\
573094311879980091977951108679734497114275489362396555522847525055228969066686316544000000000000000
\,{x_{{6}}}^{90}-\\
902776547435297518619872204025127546436588215347707699177613095959712708113907843072000000000000000
\,{x_{{6}}}^{89}+\\
1381470179511776696238079869975770761332760056220154172745621510511802459843682998681600000000000000
\,{x_{{6}}}^{88}-\\
2055691660816228681808727467311341844197624685508666419978584648469236044884728275271680000000000000
\,{x_{{6}}}^{87}+\\
2977509595428067235101258282418760365285755126075924171615103016656820208031235928227840000000000000
\,{x_{{6}}}^{86}-\\
4201725120405361863667354990201606709573190236890861384608467887244322436946675328090112000000000000
\,{x_{{6}}}^{85}+\\
5781756483073179454917236955355314508646903884452703502012201893138820119263630187298816000000000000
\,{x_{{6}}}^{84}-\\
7764371676592897388676886890577299023889575649802583977971700446782610657379596330493542400000000000
\,{x_{{6}}}^{83}+\\
10183620496243087634800999266392091503767020139559664690681345852294086238598894375770521600000000000
\,{x_{{6}}}^{82}-\\
13054516827199874689647995931541533885437441220682700405586454195035408459346608479398789120000000000
\,{x_{{6}}}^{81}+\\
16367178143760644912462581706756749405506470367424836184653013293361886227319582250458152960000000000
\,{x_{{6}}}^{80}-\\
20082208595708058444257930092579772418036354526175667544909927092935162381791985074106793984000000000
\,{x_{{6}}}^{79}+\\
24128089109767676333163701531096577498630224048672867437805663399205390974551999451865415680000000000
\,{x_{{6}}}^{78}-\\
28401200715041387036896988893111857222372577414558286924609319958231877059210419539825708236800000000
\,{x_{{6}}}^{77}+\\
32768861687769238956167161989233205835442693551136264891648443112259550388012676313583688089600000000
\,{x_{{6}}}^{76}-\\
37075430186581614315913211547784967892352876656501255953200659995382589809183727308920499077120000000
\,{x_{{6}}}^{75}+\\
41151153785704934764308791987090636719181983801375625528153654823770872229589873780129323048960000000
\,{x_{{6}}}^{74}-\\
44823088360265402011083462355710485240467134423408550926180182695515577114948033664144657645568000000
\,{x_{{6}}}^{73}+\\
47927115385399894568622394434700182348901704209940419379929530409209547107317917236377852489728000000
\,{x_{{6}}}^{72}-\\
50319904861589947132816254744282238135435415320599883838412063668440441095639280488162700008243200000
\,{x_{{6}}}^{71}+\\
51889630250816021795835597281647403593459135837539633376498567673844740216014707791176178306022400000
\,{x_{{6}}}^{70}-\\
52564350198681855220054382985091068588471068882030190034292454885932208798043191092370630644910080000
\,{x_{{6}}}^{69}+\\
52317215639869872635966701755733856675226626488683021867912886301886830804537686698081061757064960000
\,{x_{{6}}}^{68}-\\
51168007224088062314848347765733799961772588113940585157751524190533262394091745932641046693611264000
\,{x_{{6}}}^{67}+\\
49180910182259851588727130149233817252851530615710856300448164229057873608404771272526968737850176000
\,{x_{{6}}}^{66}-\\
46458838246393871693574869345586158471422065209303270911696345268890771337201589894328021757505497600
\,{x_{{6}}}^{65}+\\
43134972355331848138944108783221020809381262458884310419495192086629751470296379376274718216741116800
\,{x_{{6}}}^{64}-\\
39362439497633534511245212818869498004709129802059023254302574691891728808858041908647936337387500800
\,{x_{{6}}}^{63}+\\
35303193191972277228439777065137656416167494386497276870790422335802930644442653095711597112409271840
\,{x_{{6}}}^{62}-\\
31117160063399164337139437688865039829295017277594793219988925986005316978924338694150335736704040128
\,{x_{{6}}}^{61}+\\
26952596800815813143356092122470908475673026431828753639697781507002683163957207516646044377535302416
\,{x_{{6}}}^{60}-\\
22938385143521621229111024324925922016048441025457404563508543696129189690123364595841333795232375488
\,{x_{{6}}}^{59}+\\
19178716787417680585229756047009699670913747539285298670733754230262805598913938949076314126676954520
\,{x_{{6}}}^{58}-\\
15750326215528480002321380215653164362612161106504843101453491079248331235708830001179564074205178096
\,{x_{{6}}}^{57}+\\
12702155582893911210516581025463366185635794710186618524034771427770882482957248082633410219471019268
\,{x_{{6}}}^{56}-\\
10057112312755464433983066058320195619101802597197233470570592201692916613350578305131786195545412768
\,{x_{{6}}}^{55}+\\
7815426696830648061914604250266266973243034097843367116915760277296648396423251065688813292827296958
\,{x_{{6}}}^{54}-\\
5959041500741504739342676230879226001046706924977493457884142602249121692721907973267744556413515708
\,{x_{{6}}}^{53}+\\
4456465144171912344654772807258183612540227615811733523998081212234586349153413960140835167900125575
\,{x_{{6}}}^{52}-\\
3267582019130579458571931040636411368414210107466090031294076856865039471544160629816891300649138914
\,{x_{{6}}}^{51}+\\
2348019298622537219931845106040325224449376841127936407087717058427605648457057613218067167306832799
\,{x_{{6}}}^{50}-\\
1652797965624984944364734501531548328098700190410203729083639848818702551406718676413560442288648794
\,{x_{{6}}}^{49}+\\
1139126533749155188323113784381862130506344272474638522684929310780995606920131814752352645015571154
\,{x_{{6}}}^{48}-\\
768312629130087932486704909567133368795476393558489294068272136432855429858675357821521262760285172
\,{x_{{6}}}^{47}+\\
506859330039629945929036123576903762986470201996591911552482105191542415407606632542148413760601644
\,{x_{{6}}}^{46}-\\
326874849254477840682384310229946957489771134083669968516994727481563164015454638090887344294345670
\,{x_{{6}}}^{45}+\\
205955910185909378300314035545530956983528846439074840056626177783231992263223908507756385078838752
\,{x_{{6}}}^{44}-\\
126710981824123666170968229538116813326554434163226695218742665671948434421611801101640777890740102
\,{x_{{6}}}^{43}+\\
76075669993045192563804968832318949836021878712164829516745508159362085161521962927802256246872538
\,{x_{{6}}}^{42}-\\
44546123774203858562737203316507694133062777619959630084625346029037382741728458814705943928706820
\,{x_{{6}}}^{41}+\\
25424079948814984629207093374323401542188106890965494590130148254438572882740924871317462129769286
\,{x_{{6}}}^{40}-\\
14134754167722643633726939221604262601223529929465692488570987952206549624337306012138781911925014
\,{x_{{6}}}^{39}+\\
7650210550322154812189001607124054641751922374805051507835273092274564010735771078329022936177926
\,{x_{{6}}}^{38}-\\
4028400549539187018992894114891089589580150128650696195170202803559075981096739802818599329458446
\,{x_{{6}}}^{37}+\\
2062504690081829900018028989594844151118352103047456363540289957813121404803502918698507464289716
\,{x_{{6}}}^{36}-\\
1026089059219399650760857906769747803720756098122009477430762995721873270212987522789252292364364
\,{x_{{6}}}^{35}+\\
495703837079474212874671053140765282382801340577813553149110407378105516432073961287160046768876
\,{x_{{6}}}^{34}-\\
232390733552737045335629185654999832651320240145225223444250737827560540662735451238646960136598
\,{x_{{6}}}^{33}+\\
105652022802101393188141512641862236798001179702090179364971313726411114144364573456998404552476
\,{x_{{6}}}^{32}-\\
46546959217278241385295256857636422335592436377929517788382230391252368928319967388969151737398
\,{x_{{6}}}^{31}+\\
19858100541803106926904408901752742130340672229494971316383032353051521167461457189672834072094
\,{x_{{6}}}^{30}-\\
8197455866919393826435150321412801803120761170805735702648847245471335334010157085431416016164
\,{x_{{6}}}^{29}+\\
3271585422148232882984609479876199064691579238439876178430988984960889820781934191606548102735
\,{x_{{6}}}^{28}-\\
1261233266374010693827434138105796490286236015474740910991540133982120592512441565670641934156
\,{x_{{6}}}^{27}+\\
469228512106326617112460403969728143919350145959862950061782396864465244470750494493757108597
\,{x_{{6}}}^{26}-\\
168301715723186009564910019276060906015115581793854255120274409266120834901127339810646645928
\,{x_{{6}}}^{25}+\\
58134776061640584899879938134874407675849646364688326565254499838031295975064452921070505454
\,{x_{{6}}}^{24}-\\
19315871932373280332127853706359356286529593585317343480616017086638110758962679878721592032
\,{x_{{6}}}^{23}+\\
6165442090907758467003182162168410728084866700934563866146171038163427792394988732447742576
\,{x_{{6}}}^{22}-\\
1887871491252749071165094213916356664798522111340622135482992091045000367356910998396504640
\,{x_{{6}}}^{21}+\\
553685436076448230582095063647153018919049767435862585449491653881668083955425859890407200
\,{x_{{6}}}^{20}-\\
155270069022930306455064658339299828485707651572993944747264551460206031752289211982987200
\,{x_{{6}}}^{19}+\\
41554020009974686461242824300388850412981599048711760342154161186385016909464516477310800
\,{x_{{6}}}^{18}-\\
10590245914774140828048927818702238404535327456947357218449153450593999656631562693776000
\,{x_{{6}}}^{17}+\\
2563987767983678891215691704653502301548996316078479440915761179635744715385150414860000
\,{x_{{6}}}^{16}-\\
588100297307408593160912180295955225449377126353066214406197496763592683855730919488000
\,{x_{{6}}}^{15}+\\
127394819920587345997715014771148502743677895787382759422010047188016607123337846352000
\,{x_{{6}}}^{14}-\\
25968706971106515156435215854797406123655066708180261542604377597361944349585098624000
\,{x_{{6}}}^{13}+\\
4960507940027602159108171414133540542870357682494160638808628474544074120268857600000
\,{x_{{6}}}^{12}-\\
883568435934032491854600142821810738907025437706278188778615895234588463202734080000
\,{x_{{6}}}^{11}+\\
145897682005731748694196345430889511740879526344443757233716282574010926547568640000
\,{x_{{6}}}^{10}-\\
22175866222857978558060838480620319985670100890081201714818580060768067723673600000
\,{x_{{6}}}^{9}+\\
3075862118928405963111852344998231570651617836342389807394973285218950438502400000
\,{x_{{6}}}^{8}-\\
385105259410167229527440401287446716975632740993559048809336040524404994867200000
\,{x_{{6}}}^{7}+\\
42917339672778567692638781664562933815548814423377227738175562693341224960000000
\,{x_{{6}}}^{6}-\\
4178590536688350662228868634718089493254270404833603948496989437288710144000000
\,{x_{{6}}}^{5}+\\
346352374181513467040566508305589632427483211407599391709775625992142848000000
\,{x_{{6}}}^{4}-\\
23523257372893535305090108878056327495383233042527291749486130479759360000000
\,{x_{{6}}}^{3}+\\
1230919766784379333820283776146657594954326996169558341670191811788800000000
\,{x_{{6}}}^{2}-\\
44258346696601099692147834894822419343469825540627109426180915200000000000
\,x_{{6}}+\\
822759441845983293555551467431826001867074646280506498917335040000000000$
\normalsize\section{Appendix II}
We put $h(x_6)$ here as follows:\\
\tiny$h(x_6)=4483391625806314902278623053770051604405925444922144197732882524176570461483849883534229504
\,{x_{{6}}}^{114}\\-
205326408642245030427921175421329504773285549623799716231703635104414332571732690022581665792
\,{x_{{6}}}^{113}\\+
4748227343288912990296885343966373457069938283464960086266393172774393766368333307469432356864
\,{x_{{6}}}^{112}\\-
73857027765851476864907995448674285385990573072187964007497919962287095629599881653594965934080
\,{x_{{6}}}^{111}\\+
867664040149060106328280687969035749162428713361957811145995338185100965861669204561179292205056
\,{x_{{6}}}^{110}\\-
8190861944847112540343332023107832167424206756712692909839631199237350246201450512424812054315008
\,{x_{{6}}}^{109}\\+
64524201420812387655648394773450794368378133978768740061962179983314070957918412138538943926763520
\,{x_{{6}}}^{108}\\-
434712852282437200305047238826737804111974281774961322415630966710723431930513864104548900129472512
\,{x_{{6}}}^{107}\\+
2545970008909317278697022351939013486712452518283817781567560897727987929782362553710098265197772800
\,{x_{{6}}}^{106}\\-
13097028510927733556711072703012531491895193895358908087070737594767573219855012966233119733042905088
\,{x_{{6}}}^{105}\\+
59489010993776772645935266065220062725433224359410810434669261324137366389534849780109565418018963456
\,{x_{{6}}}^{104}\\-
238513405926928676903444856966833432897697263280851985483727297580965902485083469530724927053045432320
\,{x_{{6}}}^{103}\\+
837192767901400659327388305673834160162972688262551443375142899312686418524934321672790512987409481728
\,{x_{{6}}}^{102}\\-
2513099136708432936794293404949731029600341919853348088292159358568388896888677689779264842174440669184
\,{x_{{6}}}^{101}\\+
6053719990602192749861843286183348238861603167575816936543214471122284892335479278172815542584764006400
\,{x_{{6}}}^{100}\\-
9163876311642991055395998749631150596871330962512532720142436345754112557269606001994254021795015294976
\,{x_{{6}}}^{99}\\-
9024249443985225373330180204805081933528916042035915597129425510528112051248275814341201244956282322944
\,{x_{{6}}}^{98}\\+
142074934217896015771539864320501366171593067667530296183577155353072731641661116166010165465113457328128
\,{x_{{6}}}^{97}\\-
705841999687383448773452929656752250551625929461698468891233310881014725571830333537346557215306246258688
\,{x_{{6}}}^{96}\\+
2556963252800373545141667373003943578430092496177297523595625717533131530755386426887189567118496556384256
\,{x_{{6}}}^{95}\\-
7506000937205795890819942467284655011933626507544074468697398990596153361536969409902109321077199164407808
\,{x_{{6}}}^{94}\\+
17936874296373676114196887995833491364108238894730113646733427249384695776872725472056415982690935691542528
\,{x_{{6}}}^{93}\\-
31994491935711999747347405170149967952484298573433497955772482422845439648875035140492738459008304894967808
\,{x_{{6}}}^{92}\\+
23577072416713937683736594033035793565791544287767289977647307789915294750305008817333193966475822808170496
\,{x_{{6}}}^{91}\\+
118049483336789073704806907341232718697539837973168098200793109647025686330009533187642821903486004605485056
\,{x_{{6}}}^{90}\\-
734226593700496983659859704822162571968318879856660029116807065471505272263760698526703379649368421159665664
\,{x_{{6}}}^{89}\\+
2676100767811761358159438941280732258286679394686887068007192972643646754163029477957216740488747100107964416
\,{x_{{6}}}^{88}\\-
7662993760316617988424167386186942948965267517817521495920041905265592014721029928346467644366995617925300224
\,{x_{{6}}}^{87}\\+
18228878401501112954518648203563585882625088337782727970218694078395658528409378636434793400122030758366806016
\,{x_{{6}}}^{86}\\-
35497917132999416914255497739983952213960106444786351129608800669158645578268238705407671446370670634791337984
\,{x_{{6}}}^{85}\\+
50080604073204066831006607591646819608839953930614131764384353535761210999449666445853487954362644268536823808
\,{x_{{6}}}^{84}\\-
15720285635577805336432985158241452724031542073857594702537987971249459895284267203702419415647310137933168640
\,{x_{{6}}}^{83}\\-
213253922155301306434562560925404886975837336511684803944554016208255017287642490436704424404810363342977761280
\,{x_{{6}}}^{82}\\+
1007493043455866176533503224292999744184162166789427844011779822050534712435567593774778075593340183108260986880
\,{x_{{6}}}^{81}\\-
3171041860497603522808595547606604744997187422941016394874266237599677292995197145911648361293608523220490649600
\,{x_{{6}}}^{80}\\+
8192737511252059459769691589000259072047018568076663448429160778728067226324093815029594653616870160763925299200
\,{x_{{6}}}^{79}\\-
18305135177471481148456590646858825040465716298190833048619372186905924029582005984292092620693283163176802713600
\,{x_{{6}}}^{78}\\+
35624173877351729177444371747291195446773624694183712274885750085250404507310578866325771229695512146484731576320
\,{x_{{6}}}^{77}\\-
58582234594937660061939858792579634811160139147261998875044075344563326032022628102019721902715587786812839428096
\,{x_{{6}}}^{76}\\+
71819099318483048562659053709646755971169118190910895195097395111871506597473238712641461855192748334374136578048
\,{x_{{6}}}^{75}\\-
22093940243587190007287340760642189308946381991953536838792383303959476123643969362831745194258873365455764979712
\,{x_{{6}}}^{74}\\-
233025217903294298517141287464717280418226382056033072040305034274672769945242846989255002152281094724055960911872
\,{x_{{6}}}^{73}\\+
1025523298653789662058798736792859638602868572150673287005471885539852081213238769000767234977739727429349020794880
\,{x_{{6}}}^{72}\\-
3059734288811407594538221696447826114213564934841831098202348820839496834181221537195234953305884612856030091542528
\,{x_{{6}}}^{71}\\+
7724840108834128111122985253623743171066155405872104137836175939697010487564442474281477549811357135551193290637312
\,{x_{{6}}}^{70}\\-
17598555229915204052409275721605938532617305913212823977587619976905852836472667837758344972101004218032171429593088
\,{x_{{6}}}^{69}\\+
37214168808436892955139992394629913811032837847887399991546743647308501001802312702530950997073919761611388968124416
\,{x_{{6}}}^{68}\\-
74164349973524924164044748487950527176785143354532194711993428539554783092197599993287548621237068947503033156239360
\,{x_{{6}}}^{67}\\+
140599225063026401498793383882444543241421231293952449285214127461284194362568804729645646538437155823194867224264704
\,{x_{{6}}}^{66}\\-
255134929665028448462123935651676393218192518641856186992058689294677223050628809197576711116283639555716651134943232
\,{x_{{6}}}^{65}\\+
445115345117851916722473459965705479745480779085200616816510631773582234441390714046969812170649847704933366696257536
\,{x_{{6}}}^{64}\\-
749061719034523494343938879903371231292812041784789121517707667677755583798652397656826217325621436188989758776354816
\,{x_{{6}}}^{63}\\+
1219007780972759381537137388267547796560314187219082470186252125095981682570335264334237104391399268072460319401508864
\,{x_{{6}}}^{62}\\-
1922267833843986235022933568021913995489998232190319502655612143006018494863365132399299614081058546517837130018135296
\,{x_{{6}}}^{61}\\+
2942049957640812163906274015735702461296336164943877137584275801313056626469838939156398925995441529615603165326043200
\,{x_{{6}}}^{60}\\-
4376243911416590305490787122087133601237261355330019247902839015057483055317677721200907371667145695529192616110802304
\,{x_{{6}}}^{59}\\+
6333726504703618856653502310224966798566019634504342266079892781888679281304936160549551502977632695189789118671499520
\,{x_{{6}}}^{58}\\-
8927675097091887804745226828453944747104206326592990075536994371766100547562570977756777636313020791928669151806494976
\,{x_{{6}}}^{57}\\+
12265685956882439991845820057377301977701307980221174648635325190086145361805524317571531714569142062396466085421260528
\,{x_{{6}}}^{56}\\-
16436954639181339199725274838984659482882699396892670565503454008734824997698238810780224586923809935787495073755682048
\,{x_{{6}}}^{55}\\+
21497350258453638415397279374929990280857180860199552360851908786111509181565393506839861175772235405584492192924051136
\,{x_{{6}}}^{54}\\-
27453825635407692526495522227723181436782015685051421530035594447587824862989263800564974129586632784329227564794286480
\,{x_{{6}}}^{53}\\+
34250139598857372273893286430557004669085816584142659388356305799891035604666301778847885029162327546724314447540592476
\,{x_{{6}}}^{52}\\-
41756199327444679394343041648930798665430872130161334400565034986479139562256167975306892977457785457429926023677226648
\,{x_{{6}}}^{51}\\+
49763342019653154352562591967729846304052752675277957675977061260505683059276922971558559201463011741849416976582325488
\,{x_{{6}}}^{50}\\-
57987487462746366335335991718002373485838879355098968647341803895525756487993561779421879523681815240678851464094886056
\,{x_{{6}}}^{49}\\+
66081292616855572255171806158941767112015260139978995153901025159263362839587014955768679394364920959353261947428964737
\,{x_{{6}}}^{48}\\-
73655294673820815925288413980144343074212843219082144168473872042616657762760922761859179435971211674668985781534517500
\,{x_{{6}}}^{47}\\+
80306691545910150245367112937968208511672761873112402793431287741215694385854056693545710303318473848966530575545751580
\,{x_{{6}}}^{46}\\-
85653093862349827565580393514587961471802292127448490876587618595025473333644309850362532742862630521792781327819145020
\,{x_{{6}}}^{45}\\+
89367533191496865457171836824211611038413396796619272675309994300077047958693420361198462461574043052964279954786643780
\,{x_{{6}}}^{44}\\-
91210449257553934247002472035018984390462182691618369432805884278529261913185953934119479451380811117892501648130466020
\,{x_{{6}}}^{43}\\+
91054455426155996855357744213192997670275564845348405630767309072963233095896146453409671919815167821476644134852675444
\,{x_{{6}}}^{42}\\-
88898437541174711715857750586848730886368844816411220259916897384768072942558579930738004664931842901116211216035133972
\,{x_{{6}}}^{41}\\+
84868888741265046378290537973934804950595327870250375948494626877567996926661092217987732177349541474335293535257240726
\,{x_{{6}}}^{40}\\-
79208115369114766990089175674872559209294113338184568307577444706375155818086412275304344000524148250370963873485090316
\,{x_{{6}}}^{39}\\+
72250776886153494923569109186657985156939727734775421593425799130418669640659884616431522278010448820239559548579599460
\,{x_{{6}}}^{38}\\-
64391829980467870144752860336339707615215827827081427038901326665274904763149563765091447902175939632575049165759240844
\,{x_{{6}}}^{37}\\+
56050056059301403897498004796897349226648237681855325925861150312204486441739258979455207815222504548273765606248331496
\,{x_{{6}}}^{36}\\-
47631779511787796571969634019449416455725165799777979791064852380476760453638349613177170209221510076284992362837866884
\,{x_{{6}}}^{35}\\+
39499080619739491579277073599405766205548262015368623778229693050332744389406034558016435258661721158636673516939597564
\,{x_{{6}}}^{34}\\-
31945859788755370092456761074054472831651732972799157275898200584015213478472443103682886677025951103427156238146420388
\,{x_{{6}}}^{33}\\+
25183722561334978336454160054965989376608115536889702639786203461738916756192687653176293304462703554397983017958570457
\,{x_{{6}}}^{32}\\-
19338099618373722669059669684408378980599446416023939193013170475797613589593234875606362932851846780945957933674850408
\,{x_{{6}}}^{31}\\+
14453573286857473064473938716154482163802115435134318539543214024100524130195352073381460755410255125893797529219895456
\,{x_{{6}}}^{30}\\-
10506283893683950982237258985379197783560704043783843231538653100486294357045515762738566190667231697911368567221530200
\,{x_{{6}}}^{29}\\+
7420676999572068797167588141139250346102101153695723803051284950741365894561917099189383422998601170039729736052998528
\,{x_{{6}}}^{28}\\-
5087758778644159009567284077404521921647088174034583785106555486576588627972776012025436255495025920499344227899039344
\,{x_{{6}}}^{27}\\+
3382383075548043015684431304657228198049197356435193103644820049023079706343604156193866348986611530702250237289093120
\,{x_{{6}}}^{26}\\-
2177758814225798082716970231656255637946844467794189997304066903819557627554011437473728015218092776210095863631577936
\,{x_{{6}}}^{25}\\+
1356166832243486057136360018201916252457358988130700500991581357957438071388815746386879543929710277448847638477085960
\,{x_{{6}}}^{24}\\-
815645666020237828883820614975859631653770088733044655641730207118001267968093449599423029741642182490137553503053984
\,{x_{{6}}}^{23}\\+
473021446550315591942800621896765768992471722688223069494892337368139735748855607037515729677471830411607294688173696
\,{x_{{6}}}^{22}\\-
264050411498286501527450861864478799401697473834016975193519596212174085659602735695340944356689801385833734708694944
\,{x_{{6}}}^{21}\\+
141604493847024673563277897583433233231512714531396898135844096095557682072804603943045270097096431539411007607985856
\,{x_{{6}}}^{20}\\-
72798151358188646984916605951368515149104134569690822065973025532312726888482172688778782109993500493509669967860224
\,{x_{{6}}}^{19}\\+
35791915960104934870095512437794470624977306091544210630932214416509751900942559643587069631054972130794391619136832
\,{x_{{6}}}^{18}\\-
16785065526227397551430648223230953785768377099815666232468321802711730642283913098885866848845514746866586876890112
\,{x_{{6}}}^{17}\\+
7486039386167758165498756320666538440668058449690300911364834080457175291113335212717971648404758925373300286691344
\,{x_{{6}}}^{16}\\-
3164690591337089269122651792673430984332571623895791846860594990256707408956006069818488869855195323713547291274240
\,{x_{{6}}}^{15}\\+
1263371862412637336118077448284129277800934557049604208353286016322956282676072492375524618321518497537179081405440
\,{x_{{6}}}^{14}\\-
474237174459865673622950650903008236257362680808571814423376106553991218384112852650618724470295923203219808460800
\,{x_{{6}}}^{13}\\+
166567588485401531601351575658375428881889405229134541528417132083554220056821494460709982665900520238239579509760
\,{x_{{6}}}^{12}\\-
54429856659678273150275880156007449410860423905145212704433824330499298121721064388192700925034251046304478330880
\,{x_{{6}}}^{11}\\+
16436942247144086690315529389700884710622790625823197634388451443696054490105773730764452786436016477036778127360
\,{x_{{6}}}^{10}\\-
4550559903116335272190334356678067567300330339887054893274986715736952210291974416159798729836648828774753239040
\,{x_{{6}}}^{9}\\+
1143794330365551469610644486847760213393907293445506682022696752757580949711190716433894667610583913035584634880
\,{x_{{6}}}^{8}\\-
257892406115483250820149690568604951656065829658549404453378180558623971925103779553458445170103335234043904000
\,{x_{{6}}}^{7}\\+
51366626789697031994971660978007265502691340725200934436341529069242803083658647745944596999843687882240819200
\,{x_{{6}}}^{6}\\-
8857785333824662183640035325209323345883967696840724309403070568683468357320498685584956576771894509410713600
\,{x_{{6}}}^{5}\\+
1286350853431778836288662240252507434753495656746971709685224196828127797048648651729278000729076056824217600
\,{x_{{6}}}^{4}\\-
151108572289028044685257699340398716743469878173371315869891303511532560180920169534954354971901547669094400
\,{x_{{6}}}^{3}\\+
13468295516322163243823865461588689875067172048266895252627491145464137024105265628046711165122125391462400
\,{x_{{6}}}^{2}\\-
809955593679404419990999198050851090431254722130327121234710273131829159152130816806933719380294578995200
\,x_{{6}}\\+
24652795478782568705193226321305604963662010504130505462968280403875490149216360605689272167309928038400$

\section{Appendix III}
We put $h(x_5)$ here as follows:\\
\tiny$128347343962975370412693426749009849339839357016801236766964799153788520545469283287700441036952900345128917401600000000\\
\,{x_{{5}}}^{178}-\\
5530540078084901614541347163272585556710461227261360504342982180428179351592239804543551577526948963951376020275200000000\\
\,{x_{{5}}}^{177}+\\
1214368671181023056340825625079201542939242778107087349459445314331171252355905862118372161779163451531290451629834240000\\00
\,{x_{{5}}}^{176}-\\
1812431994310165286339190290550274741253258982737776277937193581998625964287755164528412074760604334773217549337021644800\\000
\,{x_{{5}}}^{175}+\\
2068707462078478021064883505615211507123991283159373507839725373497169580163620572530599699435828318654118957649735188480\\0000
\,{x_{{5}}}^{174}-\\
1925873300646058612284297034856109247401795114532998418465537834769953774453717368954658193957404620362604996521357381468\\16000
\,{x_{{5}}}^{173}+\\
1522757295376626038741802388644779881996787547117412902238876564126995898728424975593066546705697211558618302730037040565\\452800
\,{x_{{5}}}^{172}-\\
1051322757936520155495220079568528552005553685193943817063614397077959353964356464752768181592436787886207528323855554564\\8496640
\,{x_{{5}}}^{171}+\\
6466180208832624164623954167835744881120569147912311659285213246326350815257644212205523180135164250850094582806452137537\\9881984
\,{x_{{5}}}^{170}-\\
3596848876442193547475956779805286385701797301462769959601207412927544020389729830738737743522457707905475592562932467680\\57942016
\,{x_{{5}}}^{169}+\\
1830830541452091752073244673215535267393611413556147916846999545960035711373596330479080243786654533431614760172044776515\\267723264
\,{x_{{5}}}^{168}-\\
8607332481708493271697961883983485815650906558642787236387410082195867547956110322590237108086473316084344853203251259425\\888927744
\,{x_{{5}}}^{167}+\\
3765785531756162462427350264287410124363832376273812417161365892528408491688570163592369263093746241632193297632948022624\\0175538176
\,{x_{{5}}}^{166}-\\
1542758760588866064808912185086604663134560725127199843177604916390452553607595186339595386002418176472602234569399028917\\30183127040
\,{x_{{5}}}^{165}+\\
5948861302729476122393831944548721298294439852821189927023221924120621699253980927796068501984745210247055008344663587788\\42761396224
\,{x_{{5}}}^{164}-\\
2168433620503978001015246311858351279579697918033759070972940543387586412850181135391452008869571808403033665890101664503\\880721891328
\,{x_{{5}}}^{163}+\\
7499537721285873739505304287210342067791732541484579259340739857323203882008983520798041021251358693804347870758083031024\\468473413632
\,{x_{{5}}}^{162}-\\
2468704930469775775549100387314406027416273460312995966687429927495116482035439308730930902233085159700116864761193703030\\6596946706432
\,{x_{{5}}}^{161}+\\
7755911459172691169162387197427066684138277328892163727843824411761749593503084147541538475456398973104702372303567199338\\1088770654208
\,{x_{{5}}}^{160}-\\
2331058312235712666892225448509846596174222520233966816340610963617511249198797415785868029717774184443317302451929362040\\81047705485312
\,{x_{{5}}}^{159}+\\
6716277499879468660401153158232638278819814773807434406849956148142762041754503108242985523991280790795627581337968578165\\36286227333120
\,{x_{{5}}}^{158}-\\
1858444060643361072383491574826396089107694238070005703922032639993004377971311240260904533619650739293066599774730420765\\791351805575168
\,{x_{{5}}}^{157}+\\
4946684911247295558626514887387912259913504869666357827546724104367661923235136463821908209878377733679310763952445549592\\908904016248832
\,{x_{{5}}}^{156}-\\
1268362635471390146625495033502720783248432474716429078211589854255894741904562053987685615084639680167706889959649283875\\6448960051412992
\,{x_{{5}}}^{155}+\\
3136829044349927529309080246078610614695759437088711569871550101091262533047682613429105463982477859241444089535123758134\\7039301509414912
\,{x_{{5}}}^{154}-\\
7491223764029709356811637399377955598275358997168945003640574807572416180566257375258081450158359154489493889604479591633\\1762994599493632
\,{x_{{5}}}^{153}+\\
1729321274353975460068086313299385761058347135937046518707069672519313793171224006543407273167152230907451040387888013318\\93454464611385344
\,{x_{{5}}}^{152}-\\
3862448271322219100148003314179251789945179543745535319206591457036005905551041782644218932825546267455936900955711508691\\93583252475281408
\,{x_{{5}}}^{151}+\\
8353697560808298624818573838240776081557945392632408597890113984450485534542780037572958427247794539389093564782532916911\\25573957343313920
\,{x_{{5}}}^{150}-\\
1750875437025128396945875516893813752771485565823419734087812202563484252488690007146280268562438058613572858052007235138\\002290270874894336
\,{x_{{5}}}^{149}+\\
3558717741163152898610084645653818401412789520711877110687028727558579770637872151654439202485152044014730072561485319944\\228895779548413952
\,{x_{{5}}}^{148}-\\
7018883023774221747505869337866007795290076146413652181055494807400316300676735421670090898086016716547456109999916206449\\519670646700572672
\,{x_{{5}}}^{147}+\\
1344090326577454676269080536638959740286362021751705582475905543979195627608206166108067704687629566692088196824886282111\\2215871390889988096
\,{x_{{5}}}^{146}-\\
2500357340996334350493014184080272021854591086360947102851033455325624121591770297969978787420202499919740127471441716297\\0442944448491134976
\,{x_{{5}}}^{145}+\\
4520587476454335797927157969001199900667087694446069724497212679260433239135486965740591399930085092271431240280202006273\\9314054267520240640
\,{x_{{5}}}^{144}-\\
7946807071999575205546372233951826638836565949636313435627902629207612818360294911794695283214829026118001658227256208134\\5749785210716211200
\,{x_{{5}}}^{143}+\\
1358819274464904722904368996520264658037602452075211613568903004595164450690664725365245465720025382923226535487437168271\\01277137327336608768
\,{x_{{5}}}^{142}-\\
2260730501023715221402951820117237525988685309020907058903498223984445669060339191697014056761522313232661929997845934749\\92095051616500285440
\,{x_{{5}}}^{141}+\\
3660842969126465175883093795434890919005846002838643575918432039468292728514894869069401979867167460458228431903630436903\\41034489112715775488
\,{x_{{5}}}^{140}-\\
57711894360699221710021632077546747141049919132051451118883569811576518841725496626317030674140166222839953247106293261686\\6722439728606024704
\,{x_{{5}}}^{139}+\\
88590360337026478263637661218289704622147477120797995792259616642457384487164259362378208804951177351417626524836035815855\\3596241524859969856
\,{x_{{5}}}^{138}-\\
13243593542905229239104932446222722176189832466152708436166911969003720588150326322038614993840850184087510296807151247921\\96657926110337094016
\,{x_{{5}}}^{137}+\\
19282281473410535418050532777533532869464145616742040289190223888954237651154333070833464824361114582899354355893859852910\\04765281444161453376
\,{x_{{5}}}^{136}-\\
27343184778566301245697638082020449115287839566295218491265463524964789763744342140049519481489031047423991498340090089037\\38220390470222017280
\,{x_{{5}}}^{135}+\\
37761371663291914344043172602728576069421507588325510595315115650429218955393674783844925170536510050310104954088789332000\\27289534047414741248
\,{x_{{5}}}^{134}-\\
50778818993359011070181288002348728645172153207204768262401563804792898691324670301084026269897601577889867509279571581822\\80430915966281621888
\,{x_{{5}}}^{133}+\\
66470826431264506725268252441958349989067586344273044684172374564969673914111921389593200972749451069705616951995554397162\\21552032257127773000
\,{x_{{5}}}^{132}-\\
84665730446582041918171153590547170741297717467316880323800825659789885267675243367605866820091964253027116271823832882947\\75015222009735128560
\,{x_{{5}}}^{131}+\\
10486888119163061675447751156948857121859199963141219964176558780631056040684574668787571259362320276852048973761325818054\\937135826435624366351
\,{x_{{5}}}^{130}-\\
12620536013278324560228947084595780237143275014528607919005920036866493148030308250538675171805810291495028831899621383322\\298441867693259250546
\,{x_{{5}}}^{129}+\\
14739776737863001042567666839590226112058259549544710673907166304300126061157057052491061865523699228635651853581189588696\\247164325419697733428
\,{x_{{5}}}^{128}-\\
16679460369617362662719340877280409880674634391915309198415911964751127421921600749464236131604806858424593842748837891846\\739757489894904575926
\,{x_{{5}}}^{127}+\\
18246061855069356262774427930563565665891620910923173909675603823473951667167929641807805709338972897645339868438683701760\\461288425937234279024
\,{x_{{5}}}^{126}-\\
19233285570060442488386914937827877833117782052023146760898126187418807939640764053489420635252356121614705421200423646460\\481711509959509996810
\,{x_{{5}}}^{125}+\\
19443561542206869977818115693871312507881639500992954223365519335213019660010193595130336288020823745865621810816118340474\\023713736998330239034
\,{x_{{5}}}^{124}-\\
18713558278608281062081771967599945053247549774478586447692193518839999432427875633869392221142542916733641040206241694047\\577434653558675107746
\,{x_{{5}}}^{123}+\\
16940702713246375549201339934890455814313161656505614285438423919187724969115857333000689856362928884964018194309672467846\\101734015661861925761
\,{x_{{5}}}^{122}-\\
14106868018312175283411426758066900824971466460722743232145611371702530230608598928878075127752965426412640251894401792077\\955028894600122973020
\,{x_{{5}}}^{121}+\\
10295093031677265847549491836410135945032518966928003240117761989185601767384020728189302501462292684543984779278993192166\\480607950060017596026
\,{x_{{5}}}^{120}-\\
56955931757311713250542011136580134827210012000416028674940282054640428556099691079062154130308932174574948553626237631614\\79145129817542588416
\,{x_{{5}}}^{119}+\\
59845702255823773594572639594670346834536718872857896346986903258304795382857646805219789178415115069538483599936926543953\\5413981694461515732
\,{x_{{5}}}^{118}+\\
46278038460576254401890695841715450584245508027898552895811009196131523939137191042432623029599335802073326694035221787070\\30205932996838300968
\,{x_{{5}}}^{117}-\\
95705144423768811129922653417032357524282570334568308511775658040090103089348805301978632738388114618292528003582343002110\\10561199525416340128
\,{x_{{5}}}^{116}+\\
13816357541942362853812005959596217263811954198763888031918529922869020142707740777124766845337383030929312305375453019848\\859819280608263097912
\,{x_{{5}}}^{115}-\\
16998762591513613402205589295835895767820651652591273663955679899317436248352168540147170954921068343887990626294466265620\\529906129147460870106
\,{x_{{5}}}^{114}+\\
18842664725874904761582680706937296082479141662342395954790075813878725006427224089402377457362416718305196830885963529603\\974259123094715664788
\,{x_{{5}}}^{113}-\\
19200042801081826463945410967341636351725716070301634363644751842135677495525102964833404254403218149575076166154487886901\\083456510120496325024
\,{x_{{5}}}^{112}+\\
18070657052022404756902351480006894645422849368713506412749886434655518138174980842745337661934344467314412302415031328692\\627546584176909135628
\,{x_{{5}}}^{111}-\\
15604413386847734385815400120003927902497985226759843334259915008135757467031268664231870103696132891900831564223206756506\\497003192621371827344
\,{x_{{5}}}^{110}+\\
12084476395928257126009035173251416686493471710527219618463993257752279283341616182259877513798705304117531535646864445617\\965407465906487555092
\,{x_{{5}}}^{109}-\\
78931866904733395034289256132372966791572253304031661980346275203024602129850729610902464030334058964117115195632584864941\\27971601291451064660
\,{x_{{5}}}^{108}+\\
34654996331738116300513568927403911926897583255662792934426368710859590885641946318677096585786926931022969967024549321238\\38414077492176951172
\,{x_{{5}}}^{107}+\\
76342008137265855352457154173381913204276007231457903654058261392175970556076497411261325380320027940810270982640154787009\\6967821982063779234
\,{x_{{5}}}^{106}-\\
44089803863368391781962443736090686629848075360746732521233299213774457394769510576494086504194180112749229865300915479725\\31039930948862934080
\,{x_{{5}}}^{105}+\\
71792449556531971953932454872195603675644492773984566939341763623444370920314174139255900502559173080009628439424070427939\\05460470523437518140
\,{x_{{5}}}^{104}-\\
89027748353493143968044907141596602812663575433727226473159497890050299318071488925474014828449673921704912480683055317484\\90006229283334704360
\,{x_{{5}}}^{103}+\\
95390674243254028951416387101583884086507501671268109816504623055013937666607034381708557029350366091476713277207949663517\\75487499094709589528
\,{x_{{5}}}^{102}-\\
91712234227698594536947570100774813311059646124353327376104771944853102915172522736039016832906690161497750518546322055944\\02097572116070256360
\,{x_{{5}}}^{101}+\\
79831777547393387233398507707238880663607776647088597506503954257562846903175308829998405072345711476181455288410415947100\\68921223755775729952
\,{x_{{5}}}^{100}-\\
62259913093396359529440393894415639279058410216823142055147288941120570247242267720798257239989029000507597202679444889292\\84093955133745238552
\,{x_{{5}}}^{99}+\\
41789821240458939607128369026605760499680705283128064529550596100332375129726584250761953092423420196295474837565550618296\\69044944383370612043
\,{x_{{5}}}^{98}-\\
21117226671067749526925361477915031392700848141975664954737679906141142109136681558694635419241480249635576154592696844169\\41442053607155074002
\,{x_{{5}}}^{97}+\\
25218654402723218219412784265723850953865720122335202548657396853917596689033709014977411645525251066039904060353323050936\\5241977745737342108
\,{x_{{5}}}^{96}+\\
12351901037265407518282511061924740081620887845329673297102975249899282902269201828534732039482980880851177263059482460946\\46808832942615408698
\,{x_{{5}}}^{95}-\\
22597244646886550293078056370393329489404383276336677456171260270575647121965608337549367316252638863990999613113150114620\\72671019300635050048
\,{x_{{5}}}^{94}+\\
28031880378023975171918667387500994519280618075336086932158347006240582466329666264537234857543990217318762234697012190043\\18019354780247752102
\,{x_{{5}}}^{93}-\\
29089255504626968924361800462538757287633097787777324974718727706835499245661123896717969138624170884825368135175763189069\\73135074449468395158
\,{x_{{5}}}^{92}+\\
26645021033609399461262534085181854856939338699023228792812734128507042303568792307871574093207225493150390219840028345114\\74969500545023873550
\,{x_{{5}}}^{91}-\\
21814922959691964094476059682501210954703338243601457652314421767688853540495125406400582874348545913416657521694300262271\\13470519156640604723
\,{x_{{5}}}^{90}+\\
15758654837135889183460680619241121345718529494988694667242400885689975561363363561580029598485923625345885310278805771313\\29351010801960397708
\,{x_{{5}}}^{89}-\\
95172661954258507553271371937762198333761721598902982633603037568503484657512654656156066079406907401567336725746575017299\\4856172321973028678
\,{x_{{5}}}^{88}+\\
39014686092257133831443531466603091263676174943552517365091898167412001742049805163875032554485238175235423078617584040535\\8712766643471427288
\,{x_{{5}}}^{87}+\\
56306353613503845332057336721374540379515945535305732546721819714852460457836450478125330158592696380190021620866439445174\\028209651652318772
\,{x_{{5}}}^{86}-\\
36374308180934624338524000268106860653494568971058062405309950289805717416519666771402797225714893990126502337369595545627\\6055560680623137904
\,{x_{{5}}}^{85}+\\
53297956771021117315491231524854588607242419943130087864632930472382619648513793937279756303165103303724757897883570829598\\7475117960798034280
\,{x_{{5}}}^{84}-\\
58297741240687620422289275345171965040842707518430485059229706416906860879716902929719624361568531406401401132718540294171\\9154090756673793056
\,{x_{{5}}}^{83}+\\
54330903210510886919361286368047497716299260441536648053969514497720268765137053322555373482108793150393869774749051953152\\0530201751173553744
\,{x_{{5}}}^{82}-\\
44704275809797281898701588342830510390573895877345201155814289370013643210189774929323922082067974186887763600200081515869\\7080190365055345344
\,{x_{{5}}}^{81}+\\
32503608535394811085953898691740466307592868141835482024433247841500930421182775408446308972220230791884217504329273145047\\1599173728260916256
\,{x_{{5}}}^{80}-\\
20215044023687390802933425302157410839300360854477633793570033856904646215630026149855169020769412283681057493550633190564\\9518775406220896384
\,{x_{{5}}}^{79}+\\
95455511760176906919642371018800777844928981892707676741336945239143854492767770189799804927761612219179909567218538004185\\721626064469287232
\,{x_{{5}}}^{78}-\\
14147289720312394799626262849538252360861445526076119383146324714652244427564337843245262917773455576022560697589042884454\\464859890845895424
\,{x_{{5}}}^{77}-\\
39302651966794739241441801902560933280468734145339788320977576917795749884335851551263189459606607301337904666767757125978\\267514864559508352
\,{x_{{5}}}^{76}+\\
67316792605993791532762566613828344487040070952746968642533721000873868547897567812986029980401576153885902870070627992555\\092522387884253696
\,{x_{{5}}}^{75}-\\
75203875030209922332892102625703996246116196358548901487920954173893395949417607502630035981178016934328943997944060685041\\395830099475730688
\,{x_{{5}}}^{74}+\\
69356238253577045740751896975855014387640153539352455575996756589040915266505150451115111665471412061460145771548352734492\\425874846575889408
\,{x_{{5}}}^{73}-\\
55882145172451240580650159385164235785228893390588449769171108831427098245073648575921095037230474829566721300916894295493\\703778969485416960
\,{x_{{5}}}^{72}+\\
39755816674798794007370663733978784382203908845747672022768045381155025481228631577698942706829288156509263124975455285819\\579519137495484416
\,{x_{{5}}}^{71}-\\
24451718575218720970468823294677958604313573707162914153481709414216683555862120054548038297516111048610601728370244621651\\212620833740092416
\,{x_{{5}}}^{70}+\\
11956562829943089240709532476926166890966612723278822428118425819851692352915774525692323779136753502937127680835171950829\\831879231882358784
\,{x_{{5}}}^{69}-\\
30226179685352039524800895955345518361428984559404018686106510870471446895431897107265209451384703040089178242959815118647\\16436962424064000
\,{x_{{5}}}^{68}-\\
24690249453372989486548988650013722604899775366506767841752505768075950190862100587797874419081111566489917193304089340530\\75924353398038528
\,{x_{{5}}}^{67}+\\
51391162600474090628179820949437302702433716303933541947640669239446112757970577632494101706386877398251723711426473015455\\92070388948152320
\,{x_{{5}}}^{66}-\\
58016933737449570249418245581618155774558431581978900974847618045478308614747170956293182772236500164744849510928662342594\\25447096176164864
\,{x_{{5}}}^{65}+\\
52470207983394212985532467799261748523998335482054751876048963426595198925301259805545847617844690549121399379171943055954\\80992515734233088
\,{x_{{5}}}^{64}-\\
41181913824363202156024495448388860536445289957372511185158547964785391835612056853597056909058460857622368249998517670677\\79205471540314112
\,{x_{{5}}}^{63}+\\
28654171224447462196366498412975425948857786241381232214366058391961459443574050054046213103309661027283066591164906676834\\58750069620195328
\,{x_{{5}}}^{62}-\\
17539935888893859800208050537393306053478743006695073079714961776584629936856395525152189863849060518976797817425307957491\\90097221369921536
\,{x_{{5}}}^{61}+\\
90170636498893406824621487169529644598554287065722061778535708629457657462303066216592842316306489740001611819155809947780\\5445026163949568
\,{x_{{5}}}^{60}-\\
32603661531740924584844086706839266469149059970378843940642961411723253467203426415115021686700741917193295500308066940884\\1522536068415488
\,{x_{{5}}}^{59}-\\
11936558749239290836832309161032373802380504842588523102621146188184590128081771504527234555171543128611069846184140813371\\994814585241600
\,{x_{{5}}}^{58}+\\
17367788399603318397345516421677058663345724218952740860889840237862455416104920898634010668268825804035383463218018712005\\1297876981579776
\,{x_{{5}}}^{57}-\\
22116933081401017809909955289146481313870821367815606692489632240257072881381344071353955714357189678764343537608818241583\\9343764015939584
\,{x_{{5}}}^{56}+\\
20537730348353175922642973998378638386027232303528240034100887500777897088339400802178768382455131879787194540315140326183\\6715143544700928
\,{x_{{5}}}^{55}-\\
16242888037740361052528503359930824371322590988352274414172280309236828717087724837144652333925765454150079341617874649791\\8662720286097408
\,{x_{{5}}}^{54}+\\
114497696417758830197884111014698275137907535546875784137114831967225750146579289875639311614494765405835224440977096348370\\018726215417856
\,{x_{{5}}}^{53}-\\
729014545543073753790652421369886470888781883047576051376117884950795563340996027921182895994576279970678981168896388322476\\21086697488384
\,{x_{{5}}}^{52}+\\
416913257353352701963521372823394455738834133092343639285187122160809371135395447116412349334667379719063547121602154863733\\77831168114688
\,{x_{{5}}}^{51}-\\
207711649274708177123401975289823487557380477178611698044681330049965675246332266068387382905655874262667655727588838695814\\06138364067840
\,{x_{{5}}}^{50}+\\
816854024833697254047187243814337994162598294423541545890172062484427295343707105226053583362081314562481035222874332853256\\1279110348800
\,{x_{{5}}}^{49}-\\
144902380859816206608390794831745077273651088901696066845363525363734998515480962817310030216670490686027871363219912680451\\9523866116096
\,{x_{{5}}}^{48}-\\
155294518679120509418701272345361442131232284506148369824023354394210186222397898249937716080135372521447976527531641259161\\0529404944384
\,{x_{{5}}}^{47}+\\
246327392876394518418207281982112962012310972779666085586573466205858045853257705081642517716622737771701145414666936967327\\0877159948288
\,{x_{{5}}}^{46}-\\
235861637461994723114858865845752407871264144004202233092449567274236237832118393570438936089077517372625163144258722869296\\4070228230144
\,{x_{{5}}}^{45}+\\
187333592910790309558745043476603327571926520280422505218976817041003091498588933060037429602469331176392729457559868920025\\3003504812032
\,{x_{{5}}}^{44}-\\
133441330636232855439063727913384242514777704174868004324740215758834128192380876932023977918212310664367343831200293350654\\8313157009408
\,{x_{{5}}}^{43}+\\
879939050015245402885487709871665292173388006220640074169754233172423958808989610875007232500580654624855569083619929423061\\464777752576
\,{x_{{5}}}^{42}-\\
545871352085453629145805969945989221505151324342931115292760961248386185520016634954290794646935576612407195912199375452072\\008476524544
\,{x_{{5}}}^{41}+\\
321539361809710217991965283618576504821314970135012052850959973697973432905823889804387799747203933936974404799154501801235\\565214957568
\,{x_{{5}}}^{40}-\\
180886319908987016805945966138642725696390791099428116350016486118298447380423873998403451084902107355251191742163898327194\\352031367168
\,{x_{{5}}}^{39}+\\
975591535677821713968581599540664091395578426591479240443201789011632520044638702182797490480132976057131516041547334603979\\19085527040
\,{x_{{5}}}^{38}-\\
505771062475735734546421248450914914769317392653165401507004897347248621285361619386861242551398569264824052603616415068700\\15194824704
\,{x_{{5}}}^{37}+\\
252491841888224928759870248139713509284565485598208706035503639179459074390425142437761999325069058755914017519845822237389\\30798264320
\,{x_{{5}}}^{36}-\\
121531555111204348187289925549538244593267043921539071732513184654057994565662200881721531737013789979276751326227796956990\\78523387904
\,{x_{{5}}}^{35}+\\
564472593274100680600148912652239028720020979640638052759402890006138135299939335815204919186461581503379024157095331210078\\8278919168
\,{x_{{5}}}^{34}-\\
253126487273471433364948693115559920574526815925215831578482781218900380864004328804719564096780464499776632622378966261176\\4875886592
\,{x_{{5}}}^{33}+\\
109621314158756279378197806414111283700938814106296707538719647548646472860482993785928951975224278513883142785815248134972\\9463042048
\,{x_{{5}}}^{32}-\\
4585096771161110970920487751193294646122645548961171185426818258611867615855062369619815031522862463686011697365569566376356\\85646336
\,{x_{{5}}}^{31}+\\
185206322387095209162134553638724181255063099693953347166513428516243205486880448492237894644849278724477797771619439736452\\403953664
\,{x_{{5}}}^{30}-\\
722282917313616580717016568145931274352825287602271305037529260205924907179318739751751977755696269612161967084692046690788\\82484224
\,{x_{{5}}}^{29}+\\
271851440255818060509335490529765189991654424752954306581208549856161217808095244129356303624892099807183160403117357912246\\90311168
\,{x_{{5}}}^{28}-\\
986957294587225067939468327940238019061084871030747194704248509002668589811357593311246913911366435897988275414621217940493\\3562368
\,{x_{{5}}}^{27}+\\
345397762963202387374064203242664310678712047222430316372625747429612451883616916424192412939624916317705178923902677734880\\4444160
\,{x_{{5}}}^{26}-\\
116425760423306465214911023686588832585931810791078378406770158253540037018508994444858400360221965940799287682361320095740\\4348416
\,{x_{{5}}}^{25}+\\
377642767236105053773739732119515032245566246383740440081945941862204403899387217417962246137784628667728143242559640823257\\890816
\,{x_{{5}}}^{24}-\\
117746194354247061813124832871526951862062016052302572623527888730558772370864326119937274147018027728377536767742531764934\\934528
\,{x_{{5}}}^{23}+\\
352459870759017596756740697434679649855872761844905021374167118499217354362152667054855737000124806799768236468068217344218\\89024
\,{x_{{5}}}^{22}-\\
101148462211640964398258040702198264139855120333737285521450150284231522877704654697828850473172705829129258172716088307335\\82336
\,{x_{{5}}}^{21}+\\
277845776669249670026113592719534200179051070301050586342016529351199814983917459135671441046941536923534166276679528022540\\2880
\,{x_{{5}}}^{20}-\\
729222059942677239194861472602631970992826200048350708453908761070046160711682680923725375727555598800883196185895432265662\\464
\,{x_{{5}}}^{19}+\\
182490369232807069021403284819016676754281859773867649736818460685067399592662641294817600630086707838402766296441916835758\\080
\,{x_{{5}}}^{18}-\\
434447774417556252713537169449578275615658559171010875350475357432739654092695255407634470093845253375407341706207100319825\\92
\,{x_{{5}}}^{17}+\\
981312487711337377849734520545821034395980798786557655574632650995129713487289075138443547286859623905492491939765776338124\\8
\,{x_{{5}}}^{16}-\\
209672188120683804006152344764047079267202374011796846480825399560699264551760692952860839830270392859460630702219370640179\\2
\,{x_{{5}}}^{15}+\\
422314275671514886321123779568396098193816231242663172075394857732652916567647564791990662433359094935353521403228822962176\\
\,{x_{{5}}}^{14}-\\
79864922656123001672373003402242843914722460212594718228143218584471442458879920199517152848106944554516039278408268513280\\
\,{x_{{5}}}^{13}+\\
14114974194660343006125317651559136005123731524715696467962306387120914690527501028427382041548468495779517799388749496320\\
\,{x_{{5}}}^{12}-\\
2318631152616942863910854119763409708464510724161955816200337405320112889645914590041456718031593814708975738270882201600\\
\,{x_{{5}}}^{11}+\\
351711906340592001610413792726774343523694998021397689259098322027975998519200903802531516313342353703169108580342169600\\
\,{x_{{5}}}^{10}-\\
48881654482529098774644700465700353976861287780733065236104378937903373457008646072769666061615264294470604491325440000\\
\,{x_{{5}}}^{9}+\\
6165199924406916311726448672009618574787780580033514548805089543085277996422090595590674195421977333160669442211840000\\
\,{x_{{5}}}^{8}-\\
697265508611783523063822245221103038364102732880164181773707035276354526553076359731686155685276694260890153779200000\\
\,{x_{{5}}}^{7}+\\
69639850156490923537983768960854449000882112021803741111317445152925716960736160612227466714320924661507948544000000
\,{x_{{5}}}^{6}-\\
6019516033064772902099011570341461065066869308802412228392754929923286004456239967801891827762038101078179840000000
\,{x_{{5}}}^{5}+\\
437977099237740958620049330989996973827136218815395628707280026388262269517157957648785965624334795512217600000000
\,{x_{{5}}}^{4}-\\
25760391661695283164506462715706115713693855148545427268637078512865404994468648889469516554411358289920000000000
\,{x_{{5}}}^{3}+\\
1148563712556502089867697244210114379584113909091646727528636419580863895074813100054898930060623872000000000000
\,{x_{{5}}}^{2}-\\
34512293527648499324118594567227987096710626145681634865537971099433193107503302883739921127833600000000000000
\,x_{{5}}+\\
524157766562226639572627086349581115653126065201703317806174894734427971871427112480895139840000000000000000$

\normalsize\section{Appendix IV}
We put $h(x_5)$ here as follows:\\
\tiny$1283473439629753704126934267490098493398393570168012367669647991537885205454692832877004410369529003451289174016000000\\00
\,{x_{{5}}}^{178}-\\
5530540078084901614541347163272585556710461227261360504342982180428179351592239804543551577526948963951376020275200000000\\
\,{x_{{5}}}^{177}+\\
1214368671181023056340825625079201542939242778107087349459445314331171252355905862118372161779163451531290451629834240000\\00
\,{x_{{5}}}^{176}-\\
1812431994310165286339190290550274741253258982737776277937193581998625964287755164528412074760604334773217549337021644800\\000
\,{x_{{5}}}^{175}+\\
2068707462078478021064883505615211507123991283159373507839725373497169580163620572530599699435828318654118957649735188480\\0000
\,{x_{{5}}}^{174}-\\
1925873300646058612284297034856109247401795114532998418465537834769953774453717368954658193957404620362604996521357381468\\16000
\,{x_{{5}}}^{173}+\\
1522757295376626038741802388644779881996787547117412902238876564126995898728424975593066546705697211558618302730037040565\\452800
\,{x_{{5}}}^{172}-\\
1051322757936520155495220079568528552005553685193943817063614397077959353964356464752768181592436787886207528323855554564\\8496640
\,{x_{{5}}}^{171}+\\
6466180208832624164623954167835744881120569147912311659285213246326350815257644212205523180135164250850094582806452137537\\9881984
\,{x_{{5}}}^{170}-\\
3596848876442193547475956779805286385701797301462769959601207412927544020389729830738737743522457707905475592562932467680\\57942016
\,{x_{{5}}}^{169}+\\
1830830541452091752073244673215535267393611413556147916846999545960035711373596330479080243786654533431614760172044776515\\267723264
\,{x_{{5}}}^{168}-\\
8607332481708493271697961883983485815650906558642787236387410082195867547956110322590237108086473316084344853203251259425\\888927744
\,{x_{{5}}}^{167}+\\
3765785531756162462427350264287410124363832376273812417161365892528408491688570163592369263093746241632193297632948022624\\0175538176
\,{x_{{5}}}^{166}-\\
1542758760588866064808912185086604663134560725127199843177604916390452553607595186339595386002418176472602234569399028917\\30183127040
\,{x_{{5}}}^{165}+\\
5948861302729476122393831944548721298294439852821189927023221924120621699253980927796068501984745210247055008344663587788\\42761396224
\,{x_{{5}}}^{164}-\\
2168433620503978001015246311858351279579697918033759070972940543387586412850181135391452008869571808403033665890101664503\\880721891328
\,{x_{{5}}}^{163}+\\
7499537721285873739505304287210342067791732541484579259340739857323203882008983520798041021251358693804347870758083031024\\468473413632
\,{x_{{5}}}^{162}-\\
2468704930469775775549100387314406027416273460312995966687429927495116482035439308730930902233085159700116864761193703030\\6596946706432
\,{x_{{5}}}^{161}+\\
7755911459172691169162387197427066684138277328892163727843824411761749593503084147541538475456398973104702372303567199338\\1088770654208
\,{x_{{5}}}^{160}-\\
2331058312235712666892225448509846596174222520233966816340610963617511249198797415785868029717774184443317302451929362040\\81047705485312
\,{x_{{5}}}^{159}+\\
6716277499879468660401153158232638278819814773807434406849956148142762041754503108242985523991280790795627581337968578165\\36286227333120
\,{x_{{5}}}^{158}-\\
1858444060643361072383491574826396089107694238070005703922032639993004377971311240260904533619650739293066599774730420765\\791351805575168
\,{x_{{5}}}^{157}+\\
4946684911247295558626514887387912259913504869666357827546724104367661923235136463821908209878377733679310763952445549592\\908904016248832
\,{x_{{5}}}^{156}-\\
1268362635471390146625495033502720783248432474716429078211589854255894741904562053987685615084639680167706889959649283875\\6448960051412992
\,{x_{{5}}}^{155}+\\
3136829044349927529309080246078610614695759437088711569871550101091262533047682613429105463982477859241444089535123758134\\7039301509414912
\,{x_{{5}}}^{154}-\\
7491223764029709356811637399377955598275358997168945003640574807572416180566257375258081450158359154489493889604479591633\\1762994599493632
\,{x_{{5}}}^{153}+\\
1729321274353975460068086313299385761058347135937046518707069672519313793171224006543407273167152230907451040387888013318\\93454464611385344
\,{x_{{5}}}^{152}-\\
3862448271322219100148003314179251789945179543745535319206591457036005905551041782644218932825546267455936900955711508691\\93583252475281408
\,{x_{{5}}}^{151}+\\
8353697560808298624818573838240776081557945392632408597890113984450485534542780037572958427247794539389093564782532916911\\25573957343313920
\,{x_{{5}}}^{150}-\\
1750875437025128396945875516893813752771485565823419734087812202563484252488690007146280268562438058613572858052007235138\\002290270874894336
\,{x_{{5}}}^{149}+\\
3558717741163152898610084645653818401412789520711877110687028727558579770637872151654439202485152044014730072561485319944\\228895779548413952
\,{x_{{5}}}^{148}-\\
7018883023774221747505869337866007795290076146413652181055494807400316300676735421670090898086016716547456109999916206449\\519670646700572672
\,{x_{{5}}}^{147}+\\
1344090326577454676269080536638959740286362021751705582475905543979195627608206166108067704687629566692088196824886282111\\2215871390889988096
\,{x_{{5}}}^{146}-\\
2500357340996334350493014184080272021854591086360947102851033455325624121591770297969978787420202499919740127471441716297\\0442944448491134976
\,{x_{{5}}}^{145}+\\
4520587476454335797927157969001199900667087694446069724497212679260433239135486965740591399930085092271431240280202006273\\9314054267520240640
\,{x_{{5}}}^{144}-\\
7946807071999575205546372233951826638836565949636313435627902629207612818360294911794695283214829026118001658227256208134\\5749785210716211200
\,{x_{{5}}}^{143}+\\
1358819274464904722904368996520264658037602452075211613568903004595164450690664725365245465720025382923226535487437168271\\01277137327336608768
\,{x_{{5}}}^{142}-\\
2260730501023715221402951820117237525988685309020907058903498223984445669060339191697014056761522313232661929997845934749\\92095051616500285440
\,{x_{{5}}}^{141}+\\
3660842969126465175883093795434890919005846002838643575918432039468292728514894869069401979867167460458228431903630436903\\41034489112715775488
\,{x_{{5}}}^{140}-\\
57711894360699221710021632077546747141049919132051451118883569811576518841725496626317030674140166222839953247106293261686\\6722439728606024704
\,{x_{{5}}}^{139}+\\
88590360337026478263637661218289704622147477120797995792259616642457384487164259362378208804951177351417626524836035815855\\3596241524859969856
\,{x_{{5}}}^{138}-\\
13243593542905229239104932446222722176189832466152708436166911969003720588150326322038614993840850184087510296807151247921\\96657926110337094016
\,{x_{{5}}}^{137}+\\
19282281473410535418050532777533532869464145616742040289190223888954237651154333070833464824361114582899354355893859852910\\04765281444161453376
\,{x_{{5}}}^{136}-\\
27343184778566301245697638082020449115287839566295218491265463524964789763744342140049519481489031047423991498340090089037\\38220390470222017280
\,{x_{{5}}}^{135}+\\
37761371663291914344043172602728576069421507588325510595315115650429218955393674783844925170536510050310104954088789332000\\27289534047414741248
\,{x_{{5}}}^{134}-\\
50778818993359011070181288002348728645172153207204768262401563804792898691324670301084026269897601577889867509279571581822\\80430915966281621888
\,{x_{{5}}}^{133}+\\
66470826431264506725268252441958349989067586344273044684172374564969673914111921389593200972749451069705616951995554397162\\21552032257127773000
\,{x_{{5}}}^{132}-\\
84665730446582041918171153590547170741297717467316880323800825659789885267675243367605866820091964253027116271823832882947\\75015222009735128560
\,{x_{{5}}}^{131}+\\
10486888119163061675447751156948857121859199963141219964176558780631056040684574668787571259362320276852048973761325818054\\937135826435624366351
\,{x_{{5}}}^{130}-\\
12620536013278324560228947084595780237143275014528607919005920036866493148030308250538675171805810291495028831899621383322\\298441867693259250546
\,{x_{{5}}}^{129}+\\
14739776737863001042567666839590226112058259549544710673907166304300126061157057052491061865523699228635651853581189588696\\247164325419697733428
\,{x_{{5}}}^{128}-\\
16679460369617362662719340877280409880674634391915309198415911964751127421921600749464236131604806858424593842748837891846\\739757489894904575926
\,{x_{{5}}}^{127}+\\
18246061855069356262774427930563565665891620910923173909675603823473951667167929641807805709338972897645339868438683701760\\461288425937234279024
\,{x_{{5}}}^{126}-\\
19233285570060442488386914937827877833117782052023146760898126187418807939640764053489420635252356121614705421200423646460\\481711509959509996810
\,{x_{{5}}}^{125}+\\
19443561542206869977818115693871312507881639500992954223365519335213019660010193595130336288020823745865621810816118340474\\023713736998330239034
\,{x_{{5}}}^{124}-\\
18713558278608281062081771967599945053247549774478586447692193518839999432427875633869392221142542916733641040206241694047\\577434653558675107746
\,{x_{{5}}}^{123}+\\
16940702713246375549201339934890455814313161656505614285438423919187724969115857333000689856362928884964018194309672467846\\101734015661861925761
\,{x_{{5}}}^{122}-\\
14106868018312175283411426758066900824971466460722743232145611371702530230608598928878075127752965426412640251894401792077\\955028894600122973020
\,{x_{{5}}}^{121}+\\
10295093031677265847549491836410135945032518966928003240117761989185601767384020728189302501462292684543984779278993192166\\480607950060017596026
\,{x_{{5}}}^{120}-\\
56955931757311713250542011136580134827210012000416028674940282054640428556099691079062154130308932174574948553626237631614\\79145129817542588416
\,{x_{{5}}}^{119}+\\
59845702255823773594572639594670346834536718872857896346986903258304795382857646805219789178415115069538483599936926543953\\5413981694461515732
\,{x_{{5}}}^{118}+\\
46278038460576254401890695841715450584245508027898552895811009196131523939137191042432623029599335802073326694035221787070\\30205932996838300968
\,{x_{{5}}}^{117}-\\
957051444237688111299226534170323575242825703345683085117756580400901030893488053019786327383881146182925280035823430021101\\0561199525416340128
\,{x_{{5}}}^{116}+\\
13816357541942362853812005959596217263811954198763888031918529922869020142707740777124766845337383030929312305375453019848\\859819280608263097912
\,{x_{{5}}}^{115}-\\
16998762591513613402205589295835895767820651652591273663955679899317436248352168540147170954921068343887990626294466265620\\529906129147460870106
\,{x_{{5}}}^{114}+\\
18842664725874904761582680706937296082479141662342395954790075813878725006427224089402377457362416718305196830885963529603\\974259123094715664788
\,{x_{{5}}}^{113}-\\
19200042801081826463945410967341636351725716070301634363644751842135677495525102964833404254403218149575076166154487886901\\083456510120496325024
\,{x_{{5}}}^{112}+\\
18070657052022404756902351480006894645422849368713506412749886434655518138174980842745337661934344467314412302415031328692\\627546584176909135628
\,{x_{{5}}}^{111}-\\
15604413386847734385815400120003927902497985226759843334259915008135757467031268664231870103696132891900831564223206756506\\497003192621371827344
\,{x_{{5}}}^{110}+\\
12084476395928257126009035173251416686493471710527219618463993257752279283341616182259877513798705304117531535646864445617\\965407465906487555092
\,{x_{{5}}}^{109}-\\
78931866904733395034289256132372966791572253304031661980346275203024602129850729610902464030334058964117115195632584864941\\27971601291451064660
\,{x_{{5}}}^{108}+\\
34654996331738116300513568927403911926897583255662792934426368710859590885641946318677096585786926931022969967024549321238\\38414077492176951172
\,{x_{{5}}}^{107}+\\
76342008137265855352457154173381913204276007231457903654058261392175970556076497411261325380320027940810270982640154787009\\6967821982063779234
\,{x_{{5}}}^{106}-\\
44089803863368391781962443736090686629848075360746732521233299213774457394769510576494086504194180112749229865300915479725\\31039930948862934080
\,{x_{{5}}}^{105}+\\
71792449556531971953932454872195603675644492773984566939341763623444370920314174139255900502559173080009628439424070427939\\05460470523437518140
\,{x_{{5}}}^{104}-\\
89027748353493143968044907141596602812663575433727226473159497890050299318071488925474014828449673921704912480683055317484\\90006229283334704360
\,{x_{{5}}}^{103}+\\
95390674243254028951416387101583884086507501671268109816504623055013937666607034381708557029350366091476713277207949663517\\75487499094709589528
\,{x_{{5}}}^{102}-\\
91712234227698594536947570100774813311059646124353327376104771944853102915172522736039016832906690161497750518546322055944\\02097572116070256360
\,{x_{{5}}}^{101}+\\
79831777547393387233398507707238880663607776647088597506503954257562846903175308829998405072345711476181455288410415947100\\68921223755775729952
\,{x_{{5}}}^{100}-\\
62259913093396359529440393894415639279058410216823142055147288941120570247242267720798257239989029000507597202679444889292\\84093955133745238552
\,{x_{{5}}}^{99}+\\
41789821240458939607128369026605760499680705283128064529550596100332375129726584250761953092423420196295474837565550618296\\69044944383370612043
\,{x_{{5}}}^{98}-\\
21117226671067749526925361477915031392700848141975664954737679906141142109136681558694635419241480249635576154592696844169\\41442053607155074002
\,{x_{{5}}}^{97}+\\
25218654402723218219412784265723850953865720122335202548657396853917596689033709014977411645525251066039904060353323050936\\5241977745737342108
\,{x_{{5}}}^{96}+\\
12351901037265407518282511061924740081620887845329673297102975249899282902269201828534732039482980880851177263059482460946\\46808832942615408698
\,{x_{{5}}}^{95}-\\
22597244646886550293078056370393329489404383276336677456171260270575647121965608337549367316252638863990999613113150114620\\72671019300635050048
\,{x_{{5}}}^{94}+\\
28031880378023975171918667387500994519280618075336086932158347006240582466329666264537234857543990217318762234697012190043\\18019354780247752102
\,{x_{{5}}}^{93}-\\
29089255504626968924361800462538757287633097787777324974718727706835499245661123896717969138624170884825368135175763189069\\73135074449468395158
\,{x_{{5}}}^{92}+\\
26645021033609399461262534085181854856939338699023228792812734128507042303568792307871574093207225493150390219840028345114\\74969500545023873550
\,{x_{{5}}}^{91}-\\
21814922959691964094476059682501210954703338243601457652314421767688853540495125406400582874348545913416657521694300262271\\13470519156640604723
\,{x_{{5}}}^{90}+\\
15758654837135889183460680619241121345718529494988694667242400885689975561363363561580029598485923625345885310278805771313\\29351010801960397708
\,{x_{{5}}}^{89}-\\
95172661954258507553271371937762198333761721598902982633603037568503484657512654656156066079406907401567336725746575017299\\4856172321973028678
\,{x_{{5}}}^{88}+\\
39014686092257133831443531466603091263676174943552517365091898167412001742049805163875032554485238175235423078617584040535\\8712766643471427288
\,{x_{{5}}}^{87}+\\
56306353613503845332057336721374540379515945535305732546721819714852460457836450478125330158592696380190021620866439445174\\028209651652318772
\,{x_{{5}}}^{86}-\\
36374308180934624338524000268106860653494568971058062405309950289805717416519666771402797225714893990126502337369595545627\\6055560680623137904
\,{x_{{5}}}^{85}+\\
53297956771021117315491231524854588607242419943130087864632930472382619648513793937279756303165103303724757897883570829598\\7475117960798034280
\,{x_{{5}}}^{84}-\\
58297741240687620422289275345171965040842707518430485059229706416906860879716902929719624361568531406401401132718540294171\\9154090756673793056
\,{x_{{5}}}^{83}+\\
54330903210510886919361286368047497716299260441536648053969514497720268765137053322555373482108793150393869774749051953152\\0530201751173553744
\,{x_{{5}}}^{82}-\\
44704275809797281898701588342830510390573895877345201155814289370013643210189774929323922082067974186887763600200081515869\\7080190365055345344
\,{x_{{5}}}^{81}+\\
32503608535394811085953898691740466307592868141835482024433247841500930421182775408446308972220230791884217504329273145047\\1599173728260916256
\,{x_{{5}}}^{80}-\\
20215044023687390802933425302157410839300360854477633793570033856904646215630026149855169020769412283681057493550633190564\\9518775406220896384
\,{x_{{5}}}^{79}+\\
95455511760176906919642371018800777844928981892707676741336945239143854492767770189799804927761612219179909567218538004185\\721626064469287232
\,{x_{{5}}}^{78}-\\
14147289720312394799626262849538252360861445526076119383146324714652244427564337843245262917773455576022560697589042884454\\464859890845895424
\,{x_{{5}}}^{77}-\\
39302651966794739241441801902560933280468734145339788320977576917795749884335851551263189459606607301337904666767757125978\\267514864559508352
\,{x_{{5}}}^{76}+\\
67316792605993791532762566613828344487040070952746968642533721000873868547897567812986029980401576153885902870070627992555\\092522387884253696
\,{x_{{5}}}^{75}-\\
75203875030209922332892102625703996246116196358548901487920954173893395949417607502630035981178016934328943997944060685041\\395830099475730688
\,{x_{{5}}}^{74}+\\
69356238253577045740751896975855014387640153539352455575996756589040915266505150451115111665471412061460145771548352734492\\425874846575889408
\,{x_{{5}}}^{73}-\\
55882145172451240580650159385164235785228893390588449769171108831427098245073648575921095037230474829566721300916894295493\\703778969485416960
\,{x_{{5}}}^{72}+\\
39755816674798794007370663733978784382203908845747672022768045381155025481228631577698942706829288156509263124975455285819\\579519137495484416
\,{x_{{5}}}^{71}-\\
24451718575218720970468823294677958604313573707162914153481709414216683555862120054548038297516111048610601728370244621651\\212620833740092416
\,{x_{{5}}}^{70}+\\
11956562829943089240709532476926166890966612723278822428118425819851692352915774525692323779136753502937127680835171950829\\831879231882358784
\,{x_{{5}}}^{69}-\\
30226179685352039524800895955345518361428984559404018686106510870471446895431897107265209451384703040089178242959815118647\\16436962424064000
\,{x_{{5}}}^{68}-\\
24690249453372989486548988650013722604899775366506767841752505768075950190862100587797874419081111566489917193304089340530\\75924353398038528
\,{x_{{5}}}^{67}+\\
51391162600474090628179820949437302702433716303933541947640669239446112757970577632494101706386877398251723711426473015455\\92070388948152320
\,{x_{{5}}}^{66}-\\
58016933737449570249418245581618155774558431581978900974847618045478308614747170956293182772236500164744849510928662342594\\25447096176164864
\,{x_{{5}}}^{65}+\\
52470207983394212985532467799261748523998335482054751876048963426595198925301259805545847617844690549121399379171943055954\\80992515734233088
\,{x_{{5}}}^{64}-\\
41181913824363202156024495448388860536445289957372511185158547964785391835612056853597056909058460857622368249998517670677\\79205471540314112
\,{x_{{5}}}^{63}+\\
28654171224447462196366498412975425948857786241381232214366058391961459443574050054046213103309661027283066591164906676834\\58750069620195328
\,{x_{{5}}}^{62}-\\
17539935888893859800208050537393306053478743006695073079714961776584629936856395525152189863849060518976797817425307957491\\90097221369921536
\,{x_{{5}}}^{61}+\\
90170636498893406824621487169529644598554287065722061778535708629457657462303066216592842316306489740001611819155809947780\\5445026163949568
\,{x_{{5}}}^{60}-\\
32603661531740924584844086706839266469149059970378843940642961411723253467203426415115021686700741917193295500308066940884\\1522536068415488
\,{x_{{5}}}^{59}-\\
11936558749239290836832309161032373802380504842588523102621146188184590128081771504527234555171543128611069846184140813371\\994814585241600
\,{x_{{5}}}^{58}+\\
17367788399603318397345516421677058663345724218952740860889840237862455416104920898634010668268825804035383463218018712005\\1297876981579776
\,{x_{{5}}}^{57}-\\
22116933081401017809909955289146481313870821367815606692489632240257072881381344071353955714357189678764343537608818241583\\9343764015939584
\,{x_{{5}}}^{56}+\\
20537730348353175922642973998378638386027232303528240034100887500777897088339400802178768382455131879787194540315140326183\\6715143544700928
\,{x_{{5}}}^{55}-\\
16242888037740361052528503359930824371322590988352274414172280309236828717087724837144652333925765454150079341617874649791\\8662720286097408
\,{x_{{5}}}^{54}+\\
114497696417758830197884111014698275137907535546875784137114831967225750146579289875639311614494765405835224440977096348370\\018726215417856
\,{x_{{5}}}^{53}-\\
729014545543073753790652421369886470888781883047576051376117884950795563340996027921182895994576279970678981168896388322476\\21086697488384
\,{x_{{5}}}^{52}+\\
416913257353352701963521372823394455738834133092343639285187122160809371135395447116412349334667379719063547121602154863733\\77831168114688
\,{x_{{5}}}^{51}-\\
207711649274708177123401975289823487557380477178611698044681330049965675246332266068387382905655874262667655727588838695814\\06138364067840
\,{x_{{5}}}^{50}+\\
816854024833697254047187243814337994162598294423541545890172062484427295343707105226053583362081314562481035222874332853256\\1279110348800
\,{x_{{5}}}^{49}-\\
144902380859816206608390794831745077273651088901696066845363525363734998515480962817310030216670490686027871363219912680451\\9523866116096
\,{x_{{5}}}^{48}-\\
155294518679120509418701272345361442131232284506148369824023354394210186222397898249937716080135372521447976527531641259161\\0529404944384
\,{x_{{5}}}^{47}+\\
246327392876394518418207281982112962012310972779666085586573466205858045853257705081642517716622737771701145414666936967327\\0877159948288
\,{x_{{5}}}^{46}-\\
235861637461994723114858865845752407871264144004202233092449567274236237832118393570438936089077517372625163144258722869296\\4070228230144
\,{x_{{5}}}^{45}+\\
187333592910790309558745043476603327571926520280422505218976817041003091498588933060037429602469331176392729457559868920025\\3003504812032
\,{x_{{5}}}^{44}-\\
133441330636232855439063727913384242514777704174868004324740215758834128192380876932023977918212310664367343831200293350654\\8313157009408
\,{x_{{5}}}^{43}+\\
879939050015245402885487709871665292173388006220640074169754233172423958808989610875007232500580654624855569083619929423061\\464777752576
\,{x_{{5}}}^{42}-\\
545871352085453629145805969945989221505151324342931115292760961248386185520016634954290794646935576612407195912199375452072\\008476524544
\,{x_{{5}}}^{41}+\\
321539361809710217991965283618576504821314970135012052850959973697973432905823889804387799747203933936974404799154501801235\\565214957568
\,{x_{{5}}}^{40}-\\
180886319908987016805945966138642725696390791099428116350016486118298447380423873998403451084902107355251191742163898327194\\352031367168
\,{x_{{5}}}^{39}+\\
975591535677821713968581599540664091395578426591479240443201789011632520044638702182797490480132976057131516041547334603979\\19085527040
\,{x_{{5}}}^{38}-\\
505771062475735734546421248450914914769317392653165401507004897347248621285361619386861242551398569264824052603616415068700\\15194824704
\,{x_{{5}}}^{37}+\\
252491841888224928759870248139713509284565485598208706035503639179459074390425142437761999325069058755914017519845822237389\\30798264320
\,{x_{{5}}}^{36}-\\
121531555111204348187289925549538244593267043921539071732513184654057994565662200881721531737013789979276751326227796956990\\78523387904
\,{x_{{5}}}^{35}+\\
564472593274100680600148912652239028720020979640638052759402890006138135299939335815204919186461581503379024157095331210078\\8278919168
\,{x_{{5}}}^{34}-\\
253126487273471433364948693115559920574526815925215831578482781218900380864004328804719564096780464499776632622378966261176\\4875886592
\,{x_{{5}}}^{33}+\\
109621314158756279378197806414111283700938814106296707538719647548646472860482993785928951975224278513883142785815248134972\\9463042048
\,{x_{{5}}}^{32}-\\
4585096771161110970920487751193294646122645548961171185426818258611867615855062369619815031522862463686011697365569566376356\\85646336
\,{x_{{5}}}^{31}+\\
185206322387095209162134553638724181255063099693953347166513428516243205486880448492237894644849278724477797771619439736452\\403953664
\,{x_{{5}}}^{30}-\\
722282917313616580717016568145931274352825287602271305037529260205924907179318739751751977755696269612161967084692046690788\\82484224
\,{x_{{5}}}^{29}+\\
271851440255818060509335490529765189991654424752954306581208549856161217808095244129356303624892099807183160403117357912246\\90311168
\,{x_{{5}}}^{28}-\\
986957294587225067939468327940238019061084871030747194704248509002668589811357593311246913911366435897988275414621217940493\\3562368
\,{x_{{5}}}^{27}+\\
345397762963202387374064203242664310678712047222430316372625747429612451883616916424192412939624916317705178923902677734880\\4444160
\,{x_{{5}}}^{26}-\\
116425760423306465214911023686588832585931810791078378406770158253540037018508994444858400360221965940799287682361320095740\\4348416
\,{x_{{5}}}^{25}+\\
377642767236105053773739732119515032245566246383740440081945941862204403899387217417962246137784628667728143242559640823257\\890816
\,{x_{{5}}}^{24}-\\
117746194354247061813124832871526951862062016052302572623527888730558772370864326119937274147018027728377536767742531764934\\934528
\,{x_{{5}}}^{23}+\\
352459870759017596756740697434679649855872761844905021374167118499217354362152667054855737000124806799768236468068217344218\\89024
\,{x_{{5}}}^{22}-\\
101148462211640964398258040702198264139855120333737285521450150284231522877704654697828850473172705829129258172716088307335\\82336
\,{x_{{5}}}^{21}+\\
277845776669249670026113592719534200179051070301050586342016529351199814983917459135671441046941536923534166276679528022540\\2880
\,{x_{{5}}}^{20}-\\
729222059942677239194861472602631970992826200048350708453908761070046160711682680923725375727555598800883196185895432265662\\464
\,{x_{{5}}}^{19}+\\
182490369232807069021403284819016676754281859773867649736818460685067399592662641294817600630086707838402766296441916835758\\080
\,{x_{{5}}}^{18}-\\
4344477744175562527135371694495782756156585591710108753504753574327396540926952554076344700938452533754073417062071003198259\\2
\,{x_{{5}}}^{17}+\\
9813124877113373778497345205458210343959807987865576555746326509951297134872890751384435472868596239054924919397657763381248\\
\,{x_{{5}}}^{16}-\\
2096721881206838040061523447640470792672023740117968464808253995606992645517606929528608398302703928594606307022193706401792\\
\,{x_{{5}}}^{15}+\\
422314275671514886321123779568396098193816231242663172075394857732652916567647564791990662433359094935353521403228822962176\\
\,{x_{{5}}}^{14}-\\
79864922656123001672373003402242843914722460212594718228143218584471442458879920199517152848106944554516039278408268513280\\
\,{x_{{5}}}^{13}+\\
14114974194660343006125317651559136005123731524715696467962306387120914690527501028427382041548468495779517799388749496320\\
\,{x_{{5}}}^{12}-\\
2318631152616942863910854119763409708464510724161955816200337405320112889645914590041456718031593814708975738270882201600\\
\,{x_{{5}}}^{11}+\\
351711906340592001610413792726774343523694998021397689259098322027975998519200903802531516313342353703169108580342169600\\
\,{x_{{5}}}^{10}-\\
48881654482529098774644700465700353976861287780733065236104378937903373457008646072769666061615264294470604491325440000\\
\,{x_{{5}}}^{9}+\\
6165199924406916311726448672009618574787780580033514548805089543085277996422090595590674195421977333160669442211840000\\
\,{x_{{5}}}^{8}-\\
697265508611783523063822245221103038364102732880164181773707035276354526553076359731686155685276694260890153779200000\\
\,{x_{{5}}}^{7}+\\
69639850156490923537983768960854449000882112021803741111317445152925716960736160612227466714320924661507948544000000\\
\,{x_{{5}}}^{6}-\\
6019516033064772902099011570341461065066869308802412228392754929923286004456239967801891827762038101078179840000000\\
\,{x_{{5}}}^{5}+\\
437977099237740958620049330989996973827136218815395628707280026388262269517157957648785965624334795512217600000000\\
\,{x_{{5}}}^{4}-\\
25760391661695283164506462715706115713693855148545427268637078512865404994468648889469516554411358289920000000000\\
\,{x_{{5}}}^{3}+\\
1148563712556502089867697244210114379584113909091646727528636419580863895074813100054898930060623872000000000000\\
\,{x_{{5}}}^{2}-\\
34512293527648499324118594567227987096710626145681634865537971099433193107503302883739921127833600000000000000\\
\,x_{{5}}+\\
524157766562226639572627086349581115653126065201703317806174894734427971871427112480895139840000000000000000$

\end{document}